\documentclass{amsart}
\usepackage{amssymb}
\usepackage{amsfonts}
\usepackage{graphicx}
\usepackage{hyperref}
\usepackage{amsmath}

\newtheorem{theorem}{Theorem}[section]
\theoremstyle{plain}

\newtheorem{corollary}{Corollary}[section]

\newtheorem{definition}{Definition}[section]

\newtheorem{lemma}{Lemma}[section]
\newtheorem{notation}{Notation}[section]

\newtheorem{proposition}{Proposition}[section]
\newtheorem{remark}{Remark}[section]

\numberwithin{equation}{section}
\numberwithin{figure}{section}
\input{tcilatex}

\begin{document}
\title[\textbf{Branch values in Ahlfors' theory}]{\textbf{Branch values in
Ahlfors' theory of covering surfaces}}
\author{Zonghan Sun }
\address{Department of Mathematical Sciences, Tsinghua University, Beijing
100084, P. R. China. \textit{Email:} \textit{sun-zh13@mails.tsinghua.edu.cn
\&\ bipfiic2008xj@163.com}}
\author{Guangyuan Zhang}
\address{Department of Mathematical Sciences, Tsinghua University, Beijing
100084, P. R. China. \textit{Email:} \textit{gyzhang@math.tsinghua.edu.cn}}

\begin{abstract}
In the study of the constant in Ahlfors' second fundamental theorem
involving a set $E_{q}$ of $q$ points, branch values of covering surfaces
outside $E_{q}$ bring a lot of troubles. To avoid this situation, for a
given surface $\Sigma $, it is useful to construct a new surface $\Sigma
_{0} $ such that $L(\partial \Sigma _{0})$ $\leq L(\partial \Sigma ),$ and $%
H(\Sigma _{0})\geq H(\Sigma ),$ and all branch values of $\Sigma _{0}$ are
contained in $E_{q}.$ One special case is discussed in \cite{Z2}. The goal
of this paper is to prove the existence of such $\Sigma _{0}$, which
generalizes Lemma 9.1 and Theorem 10.1 in \cite{Z2}.
\end{abstract}

\thanks{Project 11531007 supported by NSFC}
\maketitle
\tableofcontents

\section{Introduction}

In this paper, the Riemann sphere $S$\ is the unit sphere in $\mathbb{R}%
^{3}, $ centered at the origin, and identified with $\overline{\mathbb{C}}$
via the stereographic projection $P$ in \cite{Ah0}. The Euclidean metric in $%
\mathbb{R}^{3}$ induces the spherical metric on $S$ and on $\overline{%
\mathbb{C}}.$ Then the \emph{spherical length} $L$ and the \emph{spherical
area} $A$ on $S$ have natural interpretations on $\mathbb{C}$:

\begin{equation*}
dL=\frac{2|dz|}{1+|z|^{2}},\text{ }dA=\frac{4dx\wedge dy}{(1+|z|^{2})^{2}},%
\text{ }z\in \mathbb{C}.
\end{equation*}

For a set $V$ in $\overline{%
\mathbb{C}
},$ $\partial V,$ $\overline{V}$ and $V^{\circ }$ denote its boundary,
closure and interior respectively. The notation $a\overset{\alpha }{%
\rightarrow }b$ denotes an oriented curve $\alpha $ from $a$ to $b$ on $S.$
For an arc $\alpha =\alpha (t),$ $t\in \lbrack -1,1],$ $-\alpha $ denotes
the opposite arc $(-\alpha )(t)=\alpha (-t).$

\begin{remark}
As a convention, all curves and arcs in this paper are assumed oriented, and
a subarc of an arc $\alpha $ always inherits the orientation of $\alpha .$
When specific set operations and set relations are used, curves and arcs
will be regarded as sets.
\end{remark}

$S$, and (closed) Jordan domains on $S$ are oriented inward, induced by the
stereographic projection $P,$ i.e. $P$ induces each inward normal on $%
S\backslash \{\infty \}$ into an upward normal on $\mathbb{C}$. This leads
to the following rule of the orientation of the boundary of (closed) Jordan
domains. For a closed Jordan domain $\overline{U},$ let $h$ be a M\"{o}bius
transformation with $h(P^{-1}(U))\subset 
\mathbb{C}
.$ Then as the boundary of $\overline{U}$ or $U,$ $\partial U$ is oriented
by $h\circ P^{-1}$ and the anticlockwise orientation of $\partial h(P^{-1}(%
\overline{U}))=h(P^{-1}(\partial U)).$ We denote by $\Delta $ the unit disk $%
\left\vert z\right\vert <1.$ As the boundary of $\Delta ,$ $\partial \Delta $
is oriented anticlockwise ($1\rightarrow i\rightarrow -1\rightarrow
-i\rightarrow 1$), but as the boundary of $S\backslash \overline{\Delta },$ $%
\partial (S\backslash \overline{\Delta })$ is oriented clockwise ($%
1\rightarrow -i\rightarrow -1\rightarrow i\rightarrow 1$).

A (closed) Jordan domain $U$ ($\overline{U}$) on $S$ bounded by a Jordan
curve $\gamma ,$ is called \emph{enclosed by} $\gamma ,$ if $\partial U$ and 
$\gamma $ have the same orientation. As a convention, the subsets of $S$
would always be seen from the origin of $%
\mathbb{R}
^{3}$. Then, a Jordan domain $U$ is on the left hand side of $\partial U.$
Similarly, let $U$ be a domain on $S,$ such that $\partial U$ is a finite
disjoint union $\tbigcup\limits_{1\leq j\leq m}\alpha _{j}$ of Jordan
curves. Then as a part of $\partial U,$ $\alpha _{j}$ is suitably oriented,
such that $U$ is on the left hand side of $\alpha _{j}.$

\subsection{Covering surfaces}

\begin{definition}
\label{partition} (1) For a simple arc $p_{0}\overset{\alpha }{\rightarrow }%
p_{n}$,\ a \emph{partition} of $\alpha $ is a collection $\{p_{j-1}\overset{%
\alpha _{j}}{\rightarrow }p_{j}|$ $1\leq j\leq n\}$ of subarcs of $\alpha $,
which is denoted by $\alpha =\alpha _{1}+\cdots +\alpha _{n}.$

(2) For an arc $\alpha $ and a continuous $\overline{%
\mathbb{C}
}$-valued function $f$ on $\alpha ,$ $(f,\alpha )$ denotes the path $%
f:\alpha \rightarrow \overline{%
\mathbb{C}
}.$

(3) Let $\gamma =(f,\alpha ),$ where $\alpha $ is a simple arc. Then, for a
partition 
\begin{equation*}
\alpha =p_{0}\overset{\alpha _{1}}{\rightarrow }p_{1}\overset{\alpha _{2}}{%
\rightarrow }p_{2}\text{ }\cdots \text{ }p_{n-1}\overset{\alpha _{n}}{%
\rightarrow }p_{n}=\alpha _{1}+\cdots +\alpha _{n},
\end{equation*}%
$\gamma =(f,\alpha _{1})+\cdots +(f,\alpha _{n})$ is called a \emph{partition%
} of $\gamma $, and the end points $\{f(p_{0}),\cdots ,f(p_{n})\}$ are
called the \emph{vertices}.
\end{definition}

\begin{definition}
An arc $\alpha $ on $\overline{%
\mathbb{C}
}$ or $S$ is called a \emph{simple analytic arc}, if $\alpha $ is simple and
compact, and there is a conformal mapping $\varphi $ from some neighbourhood 
$V_{\alpha }$ of $\alpha $ into $\overline{%
\mathbb{C}
},$ such that $\varphi (\alpha )\subset 
\mathbb{R}
$. An arc $\alpha $ is called a \emph{piecewise analytic arc}, if $\alpha $
could be partitioned into a finite number of simple analytic arcs.
\end{definition}

\begin{definition}
\label{light} A mapping $\widetilde{f}$ from a domain $W\subset S$ to $S$ is
called an \emph{orientation-preserving} \emph{light mapping,} if $\widetilde{%
f}$ is continuous, open,\ orientation-preserving,\emph{\ }and \emph{discrete}
(that is to say, for each $p\in S,$ $\widetilde{f}^{-1}(p)$ is discrete in $%
W $). More generally, a mapping $f$ from a subset $K$ of $S$ to $S$ is
called an \emph{orientation-preserving} \emph{light mapping}, if $f$ can be
extended to an orientation-preserving light\emph{\ }mapping $\widetilde{f}$
defined on a domain $W\supset K.$ The set of all orientation-preserving
light mappings on $K\subset S$ is denoted by $OPL(K)$.
\end{definition}

\begin{definition}
\label{surface} A \emph{covering} \emph{surface} $\Sigma $ is a pair $(f,%
\overline{U})$, such that the following hold.

1. $U$ is a domain on $S,$ such that $\partial U$ is a finite (possibly
empty) disjoint union $\tbigcup\limits_{1\leq j\leq m}\alpha _{j}$\ of
Jordan curves.

2. $f\in OPL(\overline{U})$.

3. Each closed curve $f:\alpha _{j}\rightarrow S$ is piecewise analytic,
denoted by $(f,\alpha _{j}).$

The \emph{boundary }$\partial \Sigma $ is defined as the formal sum of the
closed curves $\{(f,\alpha _{j})|$ $1\leq j\leq m\}.$
\end{definition}

There are only three types of covering surfaces discussed in this paper as
follows.

\begin{definition}
\label{family F,D} Let $\Sigma =(f,\overline{U})$ be a covering surface.

(1) $\Sigma $ is called a \emph{closed surface,} if $U=\overline{U}=S$.

(2) $\Sigma $ is called a \emph{simply-connected} \emph{surface,} if $U$ is
a Jordan domain. Then $\partial \Sigma $ is the closed curve $(f,\partial
U). $

(3) Throughout, the family of all simply-connected surfaces is denoted by $%
\mathbf{F}$.

(4) $\Sigma $ is called a \emph{doubly-connected surface, }if $U$ is an open
annulus, bounded by two disjoint Jordan curves $\partial _{in}U$ and $%
\partial _{ex}U$.
\end{definition}

Simply-connected surfaces are the most important among three types of
covering surfaces. Some readers may regard the lightness in Definition \ref%
{surface} as an artificial requirement, but actually this condition is
natural and appropriate, as in Proposition \ref{iso to mero}.

Each piecewise analytic arc $\beta $ could be partitioned into simple
analytic arcs: 
\begin{equation}
\beta =\beta _{1}+\cdots +\beta _{n};  \label{pa}
\end{equation}%
and the \emph{(spherical) length} $L(\beta )$ is defined as 
\begin{equation*}
L(\beta )\overset{def}{=}\tsum\limits_{1\leq j\leq n}L(\beta _{j})\in 
\mathbb{R}
^{+}.
\end{equation*}%
For a covering surface $\Sigma =(f,\overline{U}),$ such that $\partial U$ is
a disjoint union $\tbigcup\limits_{1\leq j\leq m}\alpha _{j}$\ of Jordan
curves, the \emph{(spherical)} \emph{perimeter} $L(\partial \Sigma )$ is
defined as$\ $%
\begin{equation*}
L(\partial \Sigma )\overset{def}{=}\tsum\limits_{1\leq j\leq m}L(f,\alpha
_{j}),
\end{equation*}%
which is also denoted by $L(f,\partial U)$ or $L(f,\tbigcup\limits_{1\leq
j\leq m}\alpha _{j}).$ $L(\beta )$ is independent of the partition (\ref{pa}%
) of $\beta $, and then $L(\partial \Sigma )$ is well-defined.\ For a closed
surface $\Sigma ,$ we have $\partial \Sigma =\emptyset ,$ and then $%
L(\partial \Sigma )=0.$

For each covering surface $\Sigma =(f,\overline{U})$ and each $w\in S,$ the 
\emph{covering number} $\overline{n}(f,w)=\overline{n}(\Sigma ,w)$ is
defined as the number $\#(f^{-1}(w)\cap U)$ of $w$-points of $f$ in $U,$
ignoring multiplicity. Then, the \emph{(spherical) area }$A(\Sigma )=A(f,%
\overline{U})=A(f,U)$ is defined as%
\begin{equation*}
\iint_{S}\overline{n}(f,w)d\sigma (w)=\iint_{%
\mathbb{C}
}\frac{4\overline{n}(f,u+iv)dudv}{(1+u^{2}+v^{2})^{2}},
\end{equation*}%
where $d\sigma $ is the spherical area element on $S.$

We could understand $A(\Sigma )=A(f,\overline{U})$ in another equivalent
way, and some notations are introduced first. For a set $K\subset \overline{%
\mathbb{C}
},$ $Mero^{\ast }(K)$ denotes the set of all non-constant meromorphic
functions on some neighbourhood of $K.$ For two (closed) domains $K$ and $G$
on $S$, $Homeo^{+}(K,G)$ denotes the set of all orientation-preserving
homeomorphisms from $K$ onto $G.$ For two oriented simple arcs $\alpha $ and 
$\beta $ on $S$, $Homeo^{+}(\alpha ,\beta )$ denotes the set of all
orientation-preserving homeomorphisms from $\alpha $ onto $\beta .$

For $\Sigma =(f,\overline{U}),$ by Proposition \ref{iso to mero}, there
exists a closed domain $\overline{V}$ on $\overline{%
\mathbb{C}
},$ and two mappings $g\in Mero^{\ast }(\overline{V})$ and $\varphi \in
Homeo^{+}(\overline{U},\overline{V}),$ such that $f=g|_{\overline{V}}\circ
\varphi .$ Then,\emph{\ }$(g,\overline{V})=\Sigma ^{\prime }$ is also a
covering surface, and for each $w\in S,$ $\overline{n}(\Sigma ,w)=\overline{n%
}(\Sigma ^{\prime },w).$ Thus by definition, 
\begin{equation*}
A(\Sigma )=A(\Sigma ^{\prime })=\tiint\nolimits_{S}\overline{n}(\Sigma
^{\prime },w)d\sigma (w)=\tiint\nolimits_{V}\frac{4|g^{\prime }(z)|^{2}dxdy}{%
(1+|g(z)|^{2})^{2}}.
\end{equation*}%
This integral is independent of the choices of the meromorphic function $g$
and the homeomorphism $\varphi ,$ as long as $f=g|_{\overline{V}}\circ
\varphi $. Especially, each $f\in Mero^{\ast }(\overline{\Delta })$ is
contained in $OPL(\overline{\Delta })$, and $\Sigma =(f,\overline{\Delta }%
)\in \mathbf{F}$ (see Definition \ref{family F,D} for $\mathbf{F}$)
satisfies:%
\begin{equation*}
L(\partial \Sigma )=\int_{\partial \Delta }\frac{2|f^{\prime }(z)||dz|}{%
1+|f(z)|^{2}},\text{ }A(\Sigma )=\tiint\nolimits_{\Delta }\frac{4|f^{\prime
}(z)|^{2}dxdy}{(1+|f(z)|^{2})^{2}}.
\end{equation*}

\subsection{The main theorem}

Ahlfors' Second Fundamental Theorem (SFT) is the following inequality (also
see \cite{Dr}, \cite{Du}, \cite{Ere}, \cite{Ha} and \cite{Z2}), parallel to
Nevanlinna's SFT:

\begin{theorem}[Ahlfors' SFT]
For any set $E_{q}=\{a_{1},a_{2},\dots ,a_{q}\}$ of $q$ distinct points on $%
S,$ $q\geq 3,$ there exists a constant $h\in 
\mathbb{R}
^{+}$, dependent only on $E_{q},$ such that for each $\Sigma \in \mathbf{F},$
\begin{equation*}
(q-2)A(\Sigma )\leq 4\pi \sum_{v=1}^{q}\overline{n}(\Sigma
,a_{v})+hL(\partial \Sigma ).
\end{equation*}
\end{theorem}

Throughout, $a_{1},a_{2},\dots ,a_{q}$ are called the \emph{special points},
and $E_{q}$ is called the \emph{special set}. For $\Sigma =(f,\overline{U}),$
\emph{the total covering number }$\overline{n}(f,E_{q})=\overline{n}(\Sigma
,E_{q})$,\emph{\ the reduced area }$R(\Sigma )$ and \emph{the RL-ratio} $%
H(\Sigma )$ are defined as follows:%
\begin{eqnarray*}
&&\overline{n}(f,E_{q})\overset{def}{=}\sum_{v=1}^{q}\overline{n}(f,a_{v}),
\\
&&R(\Sigma )\overset{def}{=}(q-2)A(\Sigma )-4\pi \overline{n}(f,E_{q}),\text{
} \\
&&H(\Sigma )\overset{def}{=}R(\Sigma )/L(\partial \Sigma ).
\end{eqnarray*}%
Then Ahlfors' SFT could be expressed as%
\begin{equation*}
h(E_{q})\overset{def}{=}\sup \{H(\Sigma )|\text{ }\Sigma \in \mathbf{F}%
\}<+\infty .
\end{equation*}%
Let $\mathbf{F}_{0}(E_{q})\subset \mathbf{F}$ be the family of all
simply-connected surfaces $\Sigma =(f,\overline{U}),$ such that $f(U)\cap
E_{q}=\emptyset .$ A direct consequence of Ahlfors' SFT is%
\begin{equation*}
h_{0}(E_{q})\overset{def}{=}\sup \{H(\Sigma )|\text{ }\Sigma \in \mathbf{F}%
_{0}(E_{q})\}\leq h(E_{q})<+\infty .
\end{equation*}

\begin{definition}
\label{branch point,D} Let $\Sigma =(f,\overline{U})$ be a covering surface,
and $p\in \overline{U}.$ If for each neighbourhood $V_{p}$ of $p$ in $%
\overline{U},$ $f$ is not injective on $V_{p}\cap U$, then $p$ is called a 
\emph{branch point }(of $\Sigma $), and $f(p)$ is called a \emph{branch value%
} (of $\Sigma $). The set of all branch points of $\Sigma $ is denoted by $%
C(f)$ or $C(\Sigma ),$ and the set of all branch values of $\Sigma $ is
denoted by $CV(f)$ or $CV(\Sigma )$. Each point $p\in \overline{U}\backslash
C(\Sigma )$ is called a \emph{regular point}.

\ $\ z\in C(\Sigma )$ is called a \emph{special branch point}$,$ or a \emph{%
non-special branch point}, if $f(z)\in E_{q}$ or $f(z)\notin E_{q}$
respectively. $z\in C(\Sigma )$ is called an \emph{interior branch point},
or a \emph{boundary branch point}, if $z\in U$ or $z\in \partial U$
respectively.
\end{definition}

In the study of the constants $h(E_{q})$ and $h_{0}(E_{q})$ in Ahlfors' SFT,
non-special branch points bring a lot of troubles. To avoid this situation,
for a given surface $\Sigma \in \mathbf{F}$ (or $\mathbf{F}_{0}(E_{q})$), it
is useful to construct a new surface $\Sigma _{0}\in \mathbf{F}$ (or $%
\mathbf{F}_{0}(E_{q})$)$,$ such that $L(\partial \Sigma _{0})$ $\leq
L(\partial \Sigma ),$ $H(\Sigma _{0})\geq H(\Sigma ),$ and $CV(\Sigma
_{0})\subset E_{q}.$ The construction of $\Sigma _{0}\in \mathbf{F}%
_{0}(E_{3})$ in the special case $E_{3}=\{0,1,\infty \},$ is discussed in 
\cite{Z2}. The goal of this paper is to prove the following main theorem,
which generalizes Lemma 9.1 and Theorem 10.1 in \cite{Z2}. In fact, we will
prove Theorem \ref{branch-special}, which is slightly stronger than the
following theorem.

\begin{theorem}
For each $\Sigma \in \mathbf{F},$ there exists another surface $\Sigma
_{0}\in \mathbf{F},$ such that the following hold.

(1) For each $a\in E_{q},$ $\overline{n}(\Sigma _{0},a)\leq \overline{n}%
(\Sigma ,a).$

(2) $L(\partial \Sigma _{0})\leq L(\partial \Sigma ),$ $H(\Sigma _{0})\geq
H(\Sigma ).$

(3) $CV(\Sigma _{0})\subset E_{q}.$
\end{theorem}

When $\Sigma \in \mathbf{F}_{0}(E_{q}),$ by (1), $\overline{n}(\Sigma
_{0},E_{q})=0,$ and then we have $\Sigma _{0}=(f_{0},\overline{U_{0}})\in 
\mathbf{F}_{0}(E_{q}).$ In addition, since 
\begin{equation*}
f_{0}(C(\Sigma _{0})\cap U_{0})=CV(\Sigma _{0})\cap f(U_{0})\subset
E_{q}\cap f(U_{0})=\emptyset ,
\end{equation*}%
we have $C(\Sigma _{0})\subset \partial U_{0},$ and then $f_{0}|_{U_{0}}$ is
a local homeomorphism.

As a simple corollary of this theorem, 
\begin{eqnarray*}
h(E_{q}) &=&\sup \{H(\Sigma )|\text{ }\Sigma \in \mathbf{F},\text{ }%
CV(\Sigma )\subset E_{q}\}, \\
h_{0}(E_{q}) &=&\sup \{H(\Sigma )|\text{ }\Sigma \in \mathbf{F}_{0}(E_{q}),%
\text{ }CV(\Sigma )\subset E_{q}\}.
\end{eqnarray*}%
This provides the convenience that in the study of $h(E_{q})$ and $%
h_{0}(E_{q}),$ only the surfaces $\Sigma \in \mathbf{F}$ with $CV(\Sigma
)\subset E_{q}$ need to be concerned.

\section{Elementary properties of surfaces}

\subsection{Isomorphisms of surfaces}

In this subsection, it is shown that the requirement about lightness in
Definition \ref{surface} is natural and appropriate.

For a simple analytic arc $\alpha (s)=x(s)+iy(s)$ in $%
\mathbb{C}
,$ where $s\in \lbrack 0,L(\alpha )]$ is the arc-length parameter of $\alpha 
$, both $x(s)$ and $y(s)$ are real analytic functions of $s.$ For each $%
s_{0}\in \lbrack 0,L(\alpha )],$ after a rotation of $\alpha ,$ the tangent
vector of $\alpha $ at $s=s_{0}$ could be assumed parallel to $%
\mathbb{R}
$. By the implicit function theorem, for $s\approx s_{0}$, $y(s)$ is a real
analytic implicit function $y(x)$ of $x(s).$ Then by the uniqueness theorem,
for two simple analytic arcs $\alpha $ and $\beta $ on $S$, $\alpha \cap
\beta $ has only a finite number of connected components. Each component $%
\gamma $ of $\alpha \cap \beta $ is either a singleton, or a simple analytic
subarc if $\gamma $ inherits the orientation of $\alpha .$ If a component $%
\gamma $ is not a singleton, then by the uniqueness of analytic extension,
each end point of $\gamma $ must be an end point of either $\alpha $ or $%
\beta .$

\begin{lemma}
\label{partitioned} Each piecewise analytic path $\alpha $ could be
partitioned into simple analytic arcs $\alpha =\alpha _{1}+\alpha
_{2}+\cdots +\alpha _{n},$ such that for each pair $(i,j)$ with $1\leq
i,j\leq n,$ either $\alpha _{i}=\pm \alpha _{j}$, or $\alpha _{i}^{\circ
}\cap \alpha _{j}=\emptyset =\alpha _{j}^{\circ }\cap \alpha _{i}.$
\end{lemma}

\begin{proof}
\ \ Suppose $(f,I)$ is a parameterization of $\alpha .$ The interval $I$
could be partitioned into 
\begin{equation*}
I=p_{0}\overset{I_{1}}{\rightarrow }p_{1}\overset{I_{2}}{\rightarrow }p_{2}%
\text{ }\cdots \text{ }p_{m-1}\overset{I_{m}}{\rightarrow }p_{m},
\end{equation*}%
such that each $\beta _{j}=(f,I_{j})$ is a simple analytic arc. Let $%
B_{v}=\{f(p_{0}),\cdots ,$ $f(p_{m})\}$. Then $f^{-1}(B_{v})$ is a finite
subset of $I,$ which leads to a finer partition $I=I_{1}^{\prime }+\cdots
+I_{k}^{\prime }.$ Then the partition 
\begin{equation*}
\alpha =\gamma _{1}+\cdots +\gamma _{k}=(f,I_{1}^{\prime })+\cdots
+(f,I_{k}^{\prime }),
\end{equation*}
is finer than $\alpha =\beta _{1}+\cdots +\beta _{m},$ and for each $%
j=1,\cdots ,k,$ $\gamma _{j}^{\circ }\cap B_{v}=\emptyset ,$ and the end
points of $\gamma _{j}$ are contained in $B_{v}.$ We claim for each pair of
two simple analytic subarcs $\gamma _{i}$ and $\gamma _{j},$ either $%
\#(\gamma _{i}\cap \gamma _{j})<+\infty ,$ or $\gamma _{i}=\pm \gamma _{j}$.
When $\#(\gamma _{i}\cap \gamma _{j})=+\infty ,$ some component $\gamma
_{ij} $ of $\gamma _{i}\cap \gamma _{j}$ is a subarc of $\gamma _{i},$ if $%
\gamma _{ij}$ inherits the orientation of $\gamma _{i}.$ Then by the
uniqueness of analytic extension, the end points of $\gamma _{ij}$ are
contained in $B_{v}. $ Since $\gamma _{j}^{\circ }\cap B_{v}=\emptyset
=\gamma _{i}^{\circ }\cap B_{v},$ we have $\gamma _{i}=\gamma _{ij}=\gamma
_{j}$ as sets, and then $\gamma _{i}=\pm \gamma _{j}$.

\ \ Let $B=B_{v}\cup (\tbigcup\limits_{\#(\gamma _{i}\cap \gamma
_{j})<+\infty }$ $(\gamma _{i}\cap \gamma _{j})).$ Then we have $\#B<+\infty
.$ The partition $\alpha =\gamma _{1}+\cdots +\gamma _{k}$ would be refined
into a new partition $\alpha =\alpha _{1}+\cdots +\alpha _{n}$ by $B,$ such
that for each $j=1,\cdots ,n,$ $\alpha _{j}^{\circ }\cap B=\emptyset ,$ and
the end points of $\gamma _{j}$ are contained in $B.$ We claim for each pair
of two subarcs $\alpha _{i}$ and $\alpha _{j},$ either $\alpha _{i}^{\circ
}\cap \alpha _{j}=\emptyset =\alpha _{j}^{\circ }\cap \alpha _{i},$ or $%
\alpha _{i}=\pm \alpha _{j}.$ When $\#(\alpha _{i}\cap \alpha _{j})=+\infty
, $ some component $\alpha _{ij}$ of $\alpha _{i}\cap \alpha _{j}$ is a
subarc $\alpha _{i},$ if $\alpha _{ij}$ inherits the orientation of $\alpha
_{i}.$ Then the end points of $\alpha _{ij}$ are contained in $B.$ Since $%
\alpha _{i}^{\circ }\cap B=\emptyset =\alpha _{j}^{\circ }\cap B,$ we have $%
\alpha _{i}=\alpha _{ij}$ $=\alpha _{j}$ as sets, and then $\alpha _{i}=\pm
\alpha _{j}.$ When $\#(\alpha _{i}\cap \alpha _{j})<+\infty ,$ by
definition, we have $\alpha _{i}\cap \alpha _{j}\subset B.$ Since $\alpha
_{j}^{\circ }\cap B=\emptyset =\alpha _{i}^{\circ }\cap B,$ we obtain $%
\alpha _{i}^{\circ }\cap \alpha _{j}=\emptyset =\alpha _{j}^{\circ }\cap
\alpha _{i}.$ Hence, $\alpha =\alpha _{1}+\cdots +\alpha _{n}$ is the
desired partition.
\end{proof}

\begin{definition}
\label{admissible} By Lemma \ref{partitioned}, for each $\Sigma \in \mathbf{F%
}$, $\partial \Sigma $ could always be partitioned into simple analytic arcs 
$\alpha _{1}+\alpha _{2}+\cdots +\alpha _{n},$ such that for each pair $%
(i,j) $ with $1\leq i,j\leq n,$ either $\alpha _{i}=\pm \alpha _{j}$, or $%
\alpha _{i}^{\circ }\cap \alpha _{j}=\emptyset =\alpha _{j}^{\circ }\cap
\alpha _{i} $. Such a partition is called \emph{an admissible partition},
and each $\alpha _{j}$ is called \emph{an admissible subarc} of $\partial
\Sigma .$
\end{definition}

For an admissible subarc $\Gamma $ of $\partial \Sigma ,$ there is exactly
one component of $S\backslash \partial \Sigma $ on the left hand side of $%
\Gamma ,$ and exactly one component of $S\backslash \partial \Sigma $ on the
right hand side of $\Gamma $. This property doesn't hold for most
non-admissible subarcs.

\begin{lemma}
\label{finite components} For each $\Sigma \in \mathbf{F},$ the set $%
S\backslash \partial \Sigma $ has only a finite number of components. Let $W$
be a component of $S\backslash \partial \Sigma ,$ and let $\Gamma $ be an
admissible subarc of $\partial \Sigma .$ Then the following hold.

(1) $W$ is a simply-connected domain, and $\partial W$ is a finite union of
admissible subarcs of $\partial \Sigma $.

(2) If $\Gamma ^{\circ }\cap \partial W\neq \emptyset ,$ then $\Gamma
\subset \partial W$.

(3) There is exactly one component $W_{\Gamma }$ of $S\backslash \partial
\Sigma $, called \emph{the component on the left hand side of} $\Gamma ,$
such that $\Gamma ^{\circ }\cap \partial W_{\Gamma }\neq \emptyset ,$ and $%
\Gamma $ is a subarc of $\partial W_{\Gamma }$. Similarly, there is exactly
one component $W_{-\Gamma }$ of $S\backslash \partial \Sigma $, called \emph{%
the component on the right hand side of} $\Gamma ,$ such that $\Gamma
^{\circ }\cap \partial W_{-\Gamma }\neq \emptyset ,$ and $-\Gamma $ is a
subarc of $\partial W_{-\Gamma }$. For other components $W$ of $S\backslash
\partial \Sigma ,$ $\Gamma ^{\circ }\cap \partial W\neq \emptyset .$ It is
possible that $W_{\Gamma }=W_{-\Gamma },$ and then $W_{\Gamma }$ is not a
Jordan domain.
\end{lemma}

\begin{proof}
\ \ A component $W$ of the open set $S\backslash \partial \Sigma $ must be a
domain. If $W$ is not simply-connected, then $W$ must separate $\partial W,$
and hence $W$ also separates $\partial \Sigma ,$ which is a contradiction to
the connectedness of $\partial \Sigma .$ Thus, each component $W$ is a
simply-connected domain. Let $\Gamma $ be an admissible subarc of $\partial
\Sigma .$ Then $\Gamma ^{\circ }$ is disjoint from other admissible subarcs
of $\partial \Sigma $ which are not coincident with $\Gamma .$ Let $p\in
\Gamma ^{\circ }\cap \partial W.$ Then there is a small disk $D_{p}$
centered at $p$, such that $\Gamma ^{\circ }\cap D_{p}=\partial \Sigma \cap
D_{p}$ divides $D_{p}$ into two components $D^{+}$ and $D^{-}$, on the left
hand side and on the right hand side of $\Gamma ^{\circ }\cap D_{p}$
respectively. Evidently, at least one of $D^{+}$ and $D^{-}$ is contained in 
$W,$ and then $\Gamma ^{\circ }\cap D_{p}\subset \partial W,$ which implies $%
\Gamma ^{\circ }\cap \partial W$ is open in $\Gamma ^{\circ }.$ Since $%
\Gamma ^{\circ }\cap \partial W$ is also closed in $\Gamma ^{\circ },$ and $%
\Gamma ^{\circ }$ is connected, we have $\Gamma ^{\circ }\cap \partial
W=\Gamma ^{\circ }$ and $\Gamma \subset \partial W.$ Hence, $\partial W$ is
the union of all admissible subarcs $\Gamma _{j}$ of $\partial \Sigma ,$
such that $\Gamma _{j}^{\circ }\cap \partial W\neq \emptyset .$

\ \ By the notations above, $D^{+}$ is contained in a unique component $%
W_{\Gamma }$ of $S\backslash \partial \Sigma ,$ and then $\Gamma $ is a
subarc of $\partial W_{\Gamma }.$ Similarly, $D^{-}$ is contained in a
unique component $W_{-\Gamma }$ of $S\backslash \partial \Sigma ,$ and then $%
-\Gamma $ is a subarc of $\partial W_{-\Gamma }.$ If $W_{\Gamma }=W_{-\Gamma
},$ then 
\begin{equation*}
p\in D_{p}=(\Gamma ^{\circ }\cap D_{p})\cup D^{+}\cup D^{-}\subset \overline{%
W_{\Gamma }},
\end{equation*}
and hence $p\in \partial W_{\Gamma }$ is contained in $\overline{W_{\Gamma }}%
^{\circ },$ which implies $W_{\Gamma }$ is not a Jordan domain. For a third
component $W_{0}$ of $S\backslash \partial \Sigma ,$ $W_{0}\cap
D^{+}=\emptyset =W_{0}\cap D^{-},$ and then $p\notin \overline{W_{0}},$
which implies $\overline{W_{0}}\cap \Gamma ^{\circ }=\emptyset .$ Because
each admissible subarc corresponds to two components at most, $S\backslash
\partial \Sigma $ has only a finite number of components. \ 
\end{proof}

\begin{definition}
\label{isomorphic} Two covering surfaces $(f,\overline{U})$ and $(g,%
\overline{V})$ are called \emph{isomorphic}, if there exists $\varphi \in
Homeo^{+}(\overline{U},\overline{V}),$ such that $f=g\circ \varphi .$ For
two closed curves $(f_{1},\alpha _{1})$ and $(f_{2},\alpha _{2})$, we define 
$(f_{1},\alpha _{1})=(f_{2},\alpha _{2})$ $\emph{(}$\emph{up to a
reparametrization)}, if there exists $\psi \in Homeo^{+}(\alpha _{1},\alpha
_{2}),$ such that $f_{2}|_{\alpha _{2}}\circ \psi =f_{1}|_{\alpha _{1}}.$
\end{definition}

When $\Sigma _{1}$ and $\Sigma _{2}\in \mathbf{F}$ are isomorphic, $\partial
\Sigma _{1}=\partial \Sigma _{2}$ up to a reparametrization, $A(\Sigma
_{1})=A(\Sigma _{2}),$ $\overline{n}(\Sigma _{1},E_{q})=\overline{n}(\Sigma
_{2},E_{q}),$ and $H(\Sigma _{1})=H(\Sigma _{2}).$ Thus, $\Sigma _{1}$ and $%
\Sigma _{2}$ can replace each other in our study. Evidently, each $\Sigma
\in \mathbf{F}$ is isomorphic to some $(f,\overline{\Delta })\in \mathbf{F}.$
A surprising result is that each surface is isomorphic to another surface
defined by a non-constant meromorphic function. The following powerful
Stoilow's Theorem claims that lightness is a topological characterization of
non-constant meromorphic functions.

\begin{theorem}
(Stoilow \cite{S}, pp. 120--121) Let $U$ be a domain on $S$, and $f\in
OPL(U).$ Then there exists a domain $V$ on $S,$ and two mappings $\varphi
\in Homeo^{+}(U,V)$ and $g\in Mero^{\ast }(V),$ such that $f=g\circ \varphi
. $
\end{theorem}

For a covering surface $(f,\overline{U})$, by definition, $f$ extends to $%
f_{1}\in OPL(W)$ on a domain $W$ containing $\overline{U}$. By Stoilow's
Theorem, there exists a domain $V$ on $S$ and two mappings $\varphi \in
Homeo^{+}(W,V)$ and $g\in Mero^{\ast }(V),$ such that $f_{1}=g\circ \varphi
. $ Then $\varphi (\overline{U})$ is a closed domain, and $\partial \varphi (%
\overline{U})=\varphi (\partial U)$ is a finite disjoint union of Jordan
curves. Evidently, $(f,\overline{U})$ is isomorphic to the covering surface $%
(g|_{\varphi (\overline{U})},\varphi (\overline{U}))$, which proves the
following proposition.

\begin{proposition}
\label{iso to mero} Each covering surface $\Sigma $ is isomorphic to another
covering surface $(f_{0},\overline{U_{0}})$, with $f_{0}\in Mero^{\ast }(%
\overline{U_{0}})$.
\end{proposition}

Especially, each closed surface is isomorphic to another closed surface
defined by a rational function. Readers could realize the requirement about
lightness in Definition \ref{surface} is natural and appropriate now.

\subsection{Local behavior and boundary behavior}

In this subsection, the local behavior and the boundary behavior of an
orientation-preserving light mapping are discussed. The following lemma
could be proved by Proposition \ref{iso to mero}, which claims locally $f$
behaves similarly to a power mapping $z\mapsto z^{d},$ via local coordinates
transformations.

\begin{notation}
\label{interval} Throughout, we denote the upper open semi-disk $\{z\in
\Delta |$ $\func{Im}z>0\}$ by $\Delta ^{+},$ and the simple subarc of $%
\partial \Delta $ from $1$ to $-1$ by $1\overset{\partial \Delta }{%
\rightarrow }-1.$ The oriented line segment from $a$ to $b$ in $%
\mathbb{C}
$ is denoted by $[a,b]$, even if $a>b,$ or $a,b\notin 
\mathbb{R}
.$ The line segment $[a,b]$ is exactly the closed interval $[a,b],$ when $%
a<b,$ and the orientation is ignored
\end{notation}

\begin{lemma}
\label{cov-1} Let $\Sigma =(f,\overline{\Delta })\in \mathbf{F},$ and $p\in
S $. Then for each sufficiently small open disk $D(p)$ centered at $p,$ the
following hold.

(i) $f^{-1}(D(p))\cap \Delta $ is a disjoint union of Jordan domains $U_{j}$%
, $j=1,2,\dots ,n,$ such that for each $1\leq j\leq n,$ $\overline{U_{j}}%
\cap f^{-1}(p)=\{x_{j}\}$.

(ii) If $x_{j}\in \Delta ,$ then there exist two homeomorphisms $\varphi
_{_{j}}\in Homeo^{+}(\overline{\Delta },\overline{U_{j}})$ with $\varphi
_{_{j}}(0)=x_{j}$ and $\psi _{j}\in Homeo^{+}(\overline{D(p)},\overline{%
\Delta })$ with $\psi _{j}(p)=0,$ such that $\psi _{j}\circ f|_{\overline{%
U_{j}}}\circ \varphi _{_{j}}(z)=z^{d_{j}}$ on $\overline{\Delta }$, with $%
d_{j}\in 
\mathbb{N}
_{+}.$

(iii) If $x_{j}\in \partial \Delta ,$ then there exist two homeomorphisms $%
\varphi _{_{j}}\in Homeo^{+}(\overline{\Delta ^{+}},\overline{U_{j}})$ with $%
\varphi _{_{j}}(0)=x_{j}$ and $\psi _{j}\in Homeo^{+}(\overline{D(p)},%
\overline{\Delta })$ with $\psi _{j}(p)=0,$ such that $\varphi
_{_{j}}([-1,1])=\partial \Delta \cap \partial U_{j}$ and $\psi _{j}\circ f|_{%
\overline{U_{j}}}\circ $ $\varphi _{_{j}}(z)=z^{d_{j}}$ on $\overline{\Delta
^{+}},$ with $d_{j}\in 
\mathbb{N}
_{+}.$
\end{lemma}

\begin{proof}
\ \ By Proposition \ref{iso to mero}, we may assume $\Sigma =(f,\overline{U}%
)\in \mathbf{F},$ and $f\in Mero^{\ast }(\overline{U}).$ For $p\in S,$ $%
f^{-1}(p)=\{x_{1},\cdots ,x_{n}\}$ is a finite set. When $x_{j}\in U,$ via
two local coordinates transformations, we may assume $p=0=x_{j},$ and $f$ is
holomorphic near $0$. Then, $f(z)=z^{d_{j}}g(z),$ with $d_{j}\in 
\mathbb{N}
_{+},$ such that $g(z)$ is also holomorphic near $0,$ but $g(0)\neq 0.$
Thus, $g(z)^{1/d_{j}}$ has a single-valued branch $G(z)$ near $0,$ and then
as a single-valued branch of $f(z)^{1/d_{j}},$ $h(z)=zG(z)$ is biholomorphic
near $0.$ Therefore, $f\circ h^{-1}(z)=$ $z^{d_{j}}$ holds on a small disk
centered at $0.$ In conclusion, there exist two homeomorphisms $\varphi
_{_{j}}\in Homeo^{+}(\overline{\Delta },\overline{U_{j}})$ with $\varphi
_{_{j}}(0)=x_{j}$ and $\psi _{j}\in Homeo^{+}(\overline{D(p)},\overline{%
\Delta })$ with $\psi _{j}(p)=0,$ such that $\psi _{j}\circ f|_{\overline{%
U_{j}}}\circ \varphi _{_{j}}(z)=z^{d_{j}}$ on $\overline{\Delta }.$

\ \ When $x_{j}\in \partial U,$ and $D(p)$ is sufficiently small, there are
two short subarcs $a_{j}\overset{\alpha _{j}}{\rightarrow }x_{j}$ and $x_{j}%
\overset{\beta _{j}}{\rightarrow }b_{j}$ of $\partial U,$ such that two
simple analytic arcs $f(\alpha _{j})$ and $f(\beta _{j})$ satisfy $\partial
D(p)\cap f(\alpha _{j})=\{f(a_{j})\},$ and $\partial D(p)\cap f(\beta
_{j})=\{f(b_{j})\}.$ By Lemma \ref{partitioned}, we may assume either $%
f(\alpha _{j})=-$ $f(\beta _{j}),$ or $f(\alpha _{j})\cap f(\beta
_{j})=\{p\}.$

\ \ When $f(\alpha _{j})=-$ $f(\beta _{j}),$ since $f\in Mero^{\ast }(%
\overline{U}),$ we may assume the component of $f^{-1}(\overline{D(p)})$
containing $x_{j}$ is a closed Jordan domain $\overline{U_{j}},$ and $\alpha
_{j}\cup \beta _{j}=\partial U_{j}\cap \partial U.$ In addition, we may
assume $x_{j}$ is the unique possible branch point of $f$ in $\overline{U_{j}%
}.$ Then $f|_{\overline{\partial U_{j}\backslash \partial U}}$ is a $d_{j}$%
-to-$1$ local homeomorphism onto $\partial D(p)$, if $a_{j}$ and $b_{j}$ in $%
\overline{\partial U_{j}\backslash \partial U}$ are identified. Let $\psi
_{j}\in Homeo^{+}(\overline{D(p)},\overline{\Delta }),$ with 
\begin{equation*}
\psi _{j}(f(\beta _{j}))=[0,1]=-\psi _{j}(f(\alpha _{j})).
\end{equation*}%
For each $z\in \overline{U_{j}}\backslash \{x_{j}\},$ there is a simple path 
$b_{j}\overset{\gamma _{z}}{\rightarrow }z$ in $\overline{U_{j}}\backslash
\{x_{j}\},$ such that $1\overset{\psi _{j}(f(\gamma _{z}))}{\rightarrow }%
\psi _{j}(f(z))$ is a rectifiable path in $\overline{\Delta }\backslash
\{0\}.$ For a fixed point $z\in \overline{U_{j}}\backslash \{x_{j}\}$, since 
$\overline{U_{j}}\backslash \{x_{j}\}$ is simply-connected, all paths $b_{j}%
\overset{\gamma _{z}}{\rightarrow }z$ in $\overline{U_{j}}\backslash
\{x_{j}\}$ are homotopic in $\overline{U_{j}}\backslash \{x_{j}\}$, and then
their images $\psi _{j}(f(\gamma _{z}))$ are homotopic in $\overline{\Delta }%
\backslash \{0\}.$ Let%
\begin{equation*}
\varphi _{j}(z)\overset{def}{=}\exp (\frac{1}{2d_{j}}\tint\nolimits_{\psi
_{j}(f(\gamma _{z}))}\frac{dw}{w})\text{ for }z\in \overline{U_{j}}%
\backslash \{x_{j}\},\text{ and }\varphi _{j}(x_{j})\overset{def}{=}0.
\end{equation*}%
Then $\varphi _{j}(z)$ is well-defined, since $\exp (\frac{1}{2d_{j}}%
\tint\nolimits_{\psi _{j}(f(\gamma _{z}))}\frac{dw}{w})$ is independent of $%
\gamma _{z}$ for a fixed point $z.$ Thus, we have%
\begin{equation*}
\varphi _{j}(z)^{2d_{j}}=\exp (\tint\nolimits_{\psi _{j}(f(\gamma _{z}))}%
\frac{dw}{w})=\exp (\ln (\psi _{j}(f(z))))=\psi _{j}(f(z)).
\end{equation*}%
Since lightness is a local property, locally $\varphi _{j}(z)$ is a
single-valued branch of $(\psi _{j}\circ f(z))^{1/(2d_{j})},$ and then we
have $\varphi _{j}\in OPL(\overline{U_{j}})$. For $z=a_{j},$ $\gamma _{a_{j}}
$ could be chosen as $\overline{\partial U_{j}\backslash \partial U},$ and
then $\psi _{j}(f(\gamma _{a_{j}}))=\partial \Delta $ with the multiplicity $%
d_{j}$. Hence $\varphi _{j}(\overline{\partial U_{j}\backslash \partial U}%
)=(1\overset{\partial \Delta }{\rightarrow }-1),$ and $\varphi
_{j}(a_{j})=-1.$ Similarly, we could verify $\varphi _{j}(b)=1,$ $\varphi
_{j}(\alpha _{j})=[-1,0],$ and $\varphi _{j}(\beta _{j})=[0,1].$ In a word, $%
\varphi _{j}$ maps $\partial U_{j}=(\partial U_{j}\backslash \partial
U)+\alpha _{j}+\beta _{j}$ homeomorphically onto $\partial \Delta ^{+},$ and
by the principle of arguments for $\varphi _{j}\in OPL(\overline{U_{j}})$,
we have $\varphi _{j}\in $ $Homeo^{+}(\overline{U_{j}},\overline{\Delta ^{+}}%
).$ Finally, 
\begin{equation*}
\psi _{j}\circ f\circ \varphi _{j}^{-1}(z)=(\varphi _{j}(\varphi
_{j}^{-1}(z)))^{2d_{j}}=z^{2d_{j}}
\end{equation*}%
holds on $\overline{\Delta ^{+}}.$

\ \ Now we consider the case that $f(\alpha _{j})\cap $ $f(\beta _{j})=\{p\}.
$ Since $f\in Mero^{\ast }(\overline{U}),$ we may assume the component $%
\overline{U_{j}}$ of $f^{-1}(\overline{D(p)})$ containing $x_{j}$ is a
closed Jordan domain, and $\alpha _{j}\cup \beta _{j}=\partial U_{j}\cap
\partial U.$ In addition, we may assume $x_{j}$ is the unique possible
branch point of $f$ in $\overline{U_{j}}.$ Then $(f,\overline{\partial
U_{j}\backslash \partial U})$ is a path in $\partial D(p)$ from $f(b_{j})$
to $f(a_{j}),$ and we may assume $(f,\overline{\partial U_{j}\backslash
\partial U})$ covers $f(b_{j})\overset{\partial D(p)}{\rightarrow }f(a_{j})$
for $d_{j}$ times, and then $(f,\overline{\partial U_{j}\backslash \partial U%
})$ covers $f(a_{j})\overset{\partial D(p)}{\rightarrow }f(b_{j})$ for $%
(d_{j}-1)$ times. Let $\psi _{j}\in Homeo^{+}(\overline{D(p)},\overline{%
\Delta }),$ such that 
\begin{equation*}
\psi _{j}(f(\beta _{j}))=[0,1]=-\psi _{j}(f(\alpha _{j})).
\end{equation*}%
For each $z\in \overline{U_{j}}\backslash \{x_{j}\},$ there is a simple path 
$b_{j}\overset{\gamma _{z}}{\rightarrow }z$ in $\overline{U_{j}}\backslash
\{x_{j}\},$ and then $1\overset{\psi _{j}(f(\gamma _{z}))}{\rightarrow }\psi
_{j}(f(z))$ is a rectifiable path in $\overline{\Delta }\backslash \{0\}.$
Since $\overline{U_{j}}\backslash \{x_{j}\}$ is simply-connected, for a
fixed point $z\in \overline{U_{j}}\backslash \{x_{j}\}$, all paths $b_{j}%
\overset{\gamma _{z}}{\rightarrow }z$ in $\overline{U_{j}}\backslash
\{x_{j}\}$ are homotopic in $\overline{U_{j}}\backslash \{x_{j}\}$, and then
their images $\psi _{j}(f(\gamma _{z}))$ are homotopic in $\overline{\Delta }%
\backslash \{0\}.$ Let%
\begin{equation*}
\varphi _{j}(z)\overset{def}{=}\exp (\frac{1}{2d_{j}-1}\tint\nolimits_{\psi
_{j}(f(\gamma _{z}))}\frac{dw}{w})\text{ for }z\in \overline{U_{j}}%
\backslash \{x_{j}\},\text{ and }\varphi _{j}(x_{j})\overset{def}{=}0.
\end{equation*}%
Then $\varphi _{j}(z)$ is well-defined, since $\exp (\frac{1}{2d_{j}-1}%
\tint\nolimits_{\psi _{j}(f(\gamma _{z}))}\frac{dw}{w})$ is independent of $%
\gamma _{z}$ for a fixed point $z.$ Evidently,%
\begin{equation*}
\varphi _{j}(z)^{2d_{j}-1}=\exp (\tint\nolimits_{\psi _{j}(f(\gamma _{z}))}%
\frac{dw}{w})=\exp (\ln (\psi _{j}(f(z))))=\psi _{j}(f(z)).
\end{equation*}

Locally, $\varphi _{j}(z)$ is a single-valued branch of $(\psi _{j}\circ
f(z))^{1/(2d_{j}-1)},$ and then we have $\varphi _{j}\in OPL(\overline{U_{j}}%
)$. For $z=a_{j},$ $\gamma _{a_{j}}$ could be chosen as $\overline{\partial
U_{j}\backslash \partial U},$ and then 
\begin{equation*}
\psi _{j}(f(\gamma _{a_{j}}))=(1\overset{\partial \Delta }{\rightarrow }%
-1)+(-1\overset{\partial \Delta }{\rightarrow }1)+\cdots +(1\overset{%
\partial \Delta }{\rightarrow }-1),
\end{equation*}%
where the number of $(1\overset{\partial \Delta }{\rightarrow }-1)$ is $%
d_{j},$ and the number of $(-1\overset{\partial \Delta }{\rightarrow }1)$ is 
$(d_{j}-1)$. Hence, $\varphi _{j}(\overline{\partial U_{j}\backslash
\partial U})=(1\overset{\partial \Delta }{\rightarrow }-1),$ and $\varphi
_{j}(a_{j})=-1.$ Similarly, we could verify $\varphi _{j}(b)=1,$ $\varphi
_{j}(\alpha _{j})=[-1,0],$ and $\varphi _{j}(\beta _{j})=[0,1].$ In a word, $%
\varphi _{j}$ maps $\partial U_{j}=(\partial U_{j}\backslash \partial
U)+\alpha _{j}+\beta _{j}$ homeomorphically onto $\partial \Delta ^{+},$ and
by the principle of arguments, we have $\varphi _{j}\in $ $Homeo^{+}(%
\overline{U_{j}},\overline{\Delta ^{+}}).$ Finally, 
\begin{equation*}
\psi _{j}\circ f\circ \varphi _{j}^{-1}(z)=(\varphi _{j}(\varphi
_{j}^{-1}(z)))^{2d_{j}-1}=z^{2d_{j}-1}
\end{equation*}%
holds on $\overline{\Delta ^{+}}.$
\end{proof}

\begin{remark}
\label{cov-2} Lemma \ref{cov-1} also works for covering surfaces $\Sigma =(f,%
\overline{U})\notin \mathbf{F}$. For $p\in \overline{U},$ there is a closed
neighbourhood $\overline{V}$ of $p$ in $\overline{U},$ such that $V$ is a
Jordan domain. Thus, Lemma \ref{cov-1} could be applied to $f|_{\overline{V}%
}\in OPL(\overline{V}),$ and this remark follows.
\end{remark}

\begin{definition}
\label{folded} As in Lemma \ref{cov-1}, let $\Sigma =(f,\overline{U})\in 
\mathbf{F}$, $x_{j}\in \overline{U},$ and $p=f(x_{j})\in S.$ When $x_{j}\in
U,$ the \emph{multiplicity} $v_{f}(x_{j})$ is defined as the exponent $d_{j}$
in (ii) of Lemma \ref{cov-1}. When $x_{j}\in \partial U,$ let $d_{j}$ be the
exponent in (iii) of Lemma \ref{cov-1}. If $d_{j}$ is even, then $x_{j}$ is
called a \emph{folded point}. and the \emph{multiplicity} $v_{f}(x_{j})$ is
defined as $\frac{d_{j}}{2}$. If $p$ is a folded point with $f(p)\in E_{q}$,
then $p$ is called a \emph{special folded point}. If $d_{j}$ is odd, then
the \emph{multiplicity} $v_{f}(x_{j})$ is defined as $\frac{d_{j}+1}{2}.$ In
each case of Lemma \ref{cov-1}, \emph{the branch index} $b(f,x_{j})$ is
defined as $v_{f}(x_{j})-1.$ By Remark \ref{cov-2}, these definitions could
be generalized to non-simply-connected covering surfaces.
\end{definition}

For a covering surface $\Sigma =(f,\overline{U}),$ there are only a finite
number of folded points and branch points. The mapping $f|_{\partial U}$ is
locally injective at $z\in \partial U,$ iff $z$ is not a folded point of $%
\Sigma .$

\begin{definition}
\label{branch} For a covering surface $\Sigma =(f,\overline{U})$ and $w\in
S, $ we define 
\begin{equation*}
n(f,w)\overset{def}{=}\tsum\limits_{z\in f^{-1}(w)\cap \overline{U}%
}v_{f}(z)-\#f^{-1}(w)\cap \partial U,
\end{equation*}%
\begin{equation*}
B(f,w)\overset{def}{=}\tsum\limits_{z\in f^{-1}(w)\cap \overline{U}%
}b(f,z)=n(f,w)-\overline{n}(f,w),
\end{equation*}%
and 
\begin{equation*}
B(f,E_{q})\overset{def}{=}\tsum\limits_{a_{j}\in E_{q}}B(f,a_{j}),\text{ }%
B(f,E_{q}^{c})\overset{def}{=}\tsum\limits_{w\in CV(\Sigma )\backslash
E_{q}}B(f,w).
\end{equation*}
\end{definition}

For convenience, in notations $\overline{n},n,b,B,$ the mapping $f$ could be
replaced by the corresponding covering surface $\Sigma =(f,\overline{U})$,
like $n(f,w)=n(\Sigma ,w).$ Two isomorphic covering surfaces $(f,\varphi (%
\overline{U}))$ and $(f\circ \varphi ,\overline{U})$ also share the same
quantities like $n(\Sigma ,w),$ $B(\Sigma ,w),$ and $CV(\Sigma ),$ and so
they can often replace each other in our study. However, usually we have 
\begin{equation*}
C(f)=\varphi (C(f\circ \varphi ))\neq C(f\circ \varphi ).
\end{equation*}

Let $\Sigma =(f,S)$ be a closed surface. By Proposition \ref{iso to mero}, $%
\Sigma $ is isomorphic to another closed surface $(g,S),$ such that $g$ is a
rational function of degree $d\geq 1.$ The \emph{degree} $\deg (\Sigma )$ of 
$\Sigma $ is defined as $d.$

\begin{proposition}
\label{R<0} For each closed surface $\Sigma ,$ $R(\Sigma )=-8\pi -4\pi
B(\Sigma ,E_{q}^{c})\leq -8\pi .$
\end{proposition}

\begin{proof}
\ \ Let $\deg (\Sigma )=d.$ Then $A(\Sigma )=4\pi d,$ $n(\Sigma ,a)=d$ for
each $a\in E_{q},$ and $n(\Sigma ,E_{q})=dq.$ By Riemann-Hurewitz Formula, 
\begin{equation*}
B(\Sigma ,E_{q})+B(\Sigma ,E_{q}^{c})=2d-2,
\end{equation*}%
and thus 
\begin{equation*}
\overline{n}(\Sigma ,E_{q})=n(\Sigma ,E_{q})-B(\Sigma
,E_{q})=dq-2d+2+B(\Sigma ,E_{q}^{c}).
\end{equation*}%
Therefore, we have 
\begin{eqnarray*}
R(\Sigma ) &=&(q-2)A(\Sigma )-4\pi \overline{n}(\Sigma ,E_{q}) \\
&=&4\pi (q-2)d-4\pi \lbrack dq-2d+2+B(\Sigma ,E_{q}^{c})] \\
&=&-8\pi -4\pi B(\Sigma ,E_{q}^{c})\leq -8\pi .
\end{eqnarray*}
\end{proof}

In a word, for a closed surface $\Sigma $, $R(\Sigma )$ assumes its maximum $%
-8\pi ,$ iff $CV(\Sigma )\subset E_{q}.$ Because $CV(\Sigma )\subset E_{q}$
is a beneficial condition to enlarge $R(\Sigma )$ for closed surfaces, it is
reasonable to regard $CV(\Sigma )\subset E_{q}$ as a beneficial condition to
enlarge $H(\Sigma )$ for $\Sigma \in \mathbf{F}$.

\begin{definition}
\label{covering sum} For $\Sigma \in \mathbf{F}$, let $U_{1},\cdots ,U_{m}$
be all the components of $S\backslash \partial \Sigma .$ By Lemma \ref{iso
to mero}, for each $U_{j},$ $n(\Sigma ,U_{j})\overset{def}{=}n(\Sigma ,w)$
is independent of $w\in U_{j}.$ The\emph{\ covering sum} $sum(\Sigma )$ is
defined as $\tsum\limits_{1\leq j\leq m}n(\Sigma ,U_{j})\in 
\mathbb{N}
_{+}.$
\end{definition}

For $\Sigma =(f,\overline{U})\in \mathbf{F}$ as above, and $w\in
U_{j}\backslash CV(\Sigma )$, we have 
\begin{equation*}
\#f^{-1}(w)\cap U=\overline{n}(\Sigma ,w)=n(\Sigma ,w)=n(\Sigma ,U_{j}).
\end{equation*}%
Since $\overline{n}(\Sigma ,w)=n(\Sigma ,w)$ almost everywhere, by the
definition of $A(\Sigma ),$ 
\begin{equation*}
A(\Sigma )=\tiint\nolimits_{%
\mathbb{C}
}\frac{4n(\Sigma ,u+iv)dudv}{(1+|u|^{2}+|v|^{2})^{2}}=\tsum\limits_{1\leq
j\leq m}n(\Sigma ,U_{j})A(U_{j}).
\end{equation*}

\begin{definition}
\label{bound-multi,D} Let $\Sigma \in \mathbf{F},$ and let $a\overset{\Gamma 
}{\rightarrow }b$ be an admissible subarc of $\partial \Sigma .$ The number
of times such that $\partial \Sigma $ passes through $\Gamma $ from $a$ to $%
b,$ is called the \emph{boundary multiplicity} $m^{+}(\partial \Sigma
,\Gamma ).$ The number of times such that $\partial \Sigma $ passes through $%
-\Gamma $ from $b$ to $a,$ is denoted by $m^{-}(\partial \Sigma ,\Gamma
)=m^{+}(\partial \Sigma ,-\Gamma ).$
\end{definition}

In other words, for $\Sigma =(f,\overline{U})\in \mathbf{F},$ $%
m^{+}(\partial \Sigma ,\Gamma )$ is the number of subarcs $\alpha $ of $%
\partial U,$ such that $(f,\alpha )=\Gamma .$ By Definition \ref{admissible}%
, $\partial \Sigma $ has an admissible partition $\partial \Sigma =\gamma
_{1}+\cdots +\gamma _{l}.$ Let $\Gamma _{1},\cdots ,$ $\Gamma _{k}$ be some
admissible subarcs of $\partial \Sigma $ in $\{\gamma _{1},\cdots ,\gamma
_{l}\},$ such that each admissible subarc $\gamma _{j}$ is coincident with
exactly one of $\Gamma _{1},\cdots ,\Gamma _{k}.$ Evidently, we have 
\begin{eqnarray*}
L(\partial \Sigma ) &=&L(\gamma _{1})+L(\gamma _{2})+\cdots +L(\gamma _{l})
\\
&=&\tsum\limits_{1\leq j\leq k}(m^{+}(\partial \Sigma ,\Gamma
_{j})+m^{-}(\partial \Sigma ,\Gamma _{j}))L(\Gamma _{j}).
\end{eqnarray*}

For each $\Gamma _{j},$ let $U_{j}^{+}$ and $U_{j}^{-}$ be the components of 
$S\backslash \partial \Sigma ,$ on the left hand side and on the right hand
side of $\Gamma _{j}$ respectively (possibly $U_{j}^{+}=U_{j}^{-}$). For
each $w\in \Gamma _{j}^{\circ },$ 
\begin{equation*}
n(\Sigma ,w)=n(\Sigma ,U_{j}^{+})-m^{+}(\partial \Sigma ,\Gamma
_{j})=n(\Sigma ,U_{j}^{-})-m^{-}(\partial \Sigma ,\Gamma _{j}).
\end{equation*}%
For $\Sigma _{1},$ $\Sigma _{2}\in \mathbf{F},$ if $\partial \Sigma
_{1}=\partial \Sigma _{2}$ up to a reparametrization (see Definition \ref%
{isomorphic}), then $m^{+}(\partial \Sigma _{1},\Gamma )=m^{+}(\partial
\Sigma _{2},\Gamma )$ for each admissible subarc $\Gamma $ of $\partial
\Sigma _{1}$ (and $\partial \Sigma _{2}$).

\subsection{Arcs and lifts}

For $f\in OPL(\overline{U})$ and a path $\beta =\beta (t)\subset S,$ a path $%
\alpha =\alpha (t)\subset \overline{U}$ is called a \emph{lift} of $\beta $
by $f$, if for each $t,$ $f(\alpha (t))=\beta (t).$ By Lemma \ref{cov-1},
for $p_{0}\in U,$ a sufficiently short path $\beta $ from $f(p_{0})$ has
exactly $v_{f}(p_{0})$ lifts from $p_{0}.$ Since $f$ is locally homeomorphic
at interior regular points, each lift could be uniquely extended, until it
meets $\partial U$ or $C(f).$ The following lemmas are based on this idea.
Throughout, for an arc $\alpha ,$ $\alpha ^{\circ }$ denotes the interior of 
$\alpha $, namely the open subarc of $\alpha $ by removing the end points.

\begin{lemma}
\label{lift-3} Let $\Sigma =(f,\overline{U})\in \mathbf{F},$ and let $%
p_{0}\in \partial U$ be a non-folded point. Let $f(p_{0})\overset{\beta }{%
\rightarrow }q_{2}$ be a simple path on $S$, on the left hand side of $%
\partial \Sigma $ near $f(p_{0}),$ such that $\beta \cap \partial \Sigma
=\{f(p_{0})\}$ and $\beta ^{\circ }\cap CV(\Sigma )=\emptyset .$ Then $\beta 
$ has exactly $d$ lifts $p_{0}\overset{\alpha _{1}}{\rightarrow }%
p_{1},\cdots ,$ $p_{0}\overset{\alpha _{d}}{\rightarrow }p_{d}$ from $p_{0},$
such that $d=v_{f}(p_{0}),$ and $\alpha _{1}^{\circ },\cdots ,\alpha
_{d}^{\circ }\subset U\backslash C(\Sigma ).$
\end{lemma}

\begin{proof}
\ \ Let $d=v_{f}(p_{0}).$ Let $\beta $ be parametrized by $t\in \lbrack
0,1]. $ We claim there exists $\delta >0,$ such that the initial subarc $%
\beta ([0,\delta ])$ of $\beta $ has exactly $d$ lifts $\alpha _{1,\delta
},\cdots ,\alpha _{d,\delta }$ from $p_{0}.$ By Lemma \ref{cov-1},
temporarily we may also assume $p_{0}=0=f(p_{0}),$ $f(z)=z^{2d-1}$ on $%
\overline{\Delta ^{+}},$ and $U\cap \Delta =\Delta ^{+}.$ Since $\beta
^{\circ }\cap \lbrack -1,1]=\emptyset ,$ and $\beta $ is on the left hand
side of $\partial \Sigma $ near $0,$ there exists $\delta >0,$ such that $%
\beta ([0,\delta ])$ is contained in $\Delta ^{+}\cup \{0\}.$ Hence, when we
regard $0<\arg \beta (t)<\pi $ for $0<t\leq \delta ,$ all $d$ lifts of $%
\beta ([0,\delta ])$ could be expressed as $\gamma _{1}(t)=\beta
(t)^{1/(2d-1)},$ $\gamma _{2}(t)=e^{\frac{2\pi i}{2d-1}}\beta
(t)^{1/(2d-1)}, $ $\cdots ,$ and $\gamma _{d}(t)=e^{\frac{2(d-1)\pi i}{2d-1}%
}\beta (t)^{1/(2d-1)}$. Our claim is verified.

\ We define $T$ as 
\begin{equation*}
T\overset{def}{=}\{t\in (0,1]:\text{the subarc }\beta ([0,t])\text{ of }%
\beta \text{ has exactly }d\text{ lifts from }p_{0}\}.
\end{equation*}%
Evidently, $t\in T$ implies $(0,t]\subset T,$ and thus $T$ is a non-empty
interval. We claim $T$ is both relatively open and relatively closed in $%
(0,1],$ and then $T=(0,1].$

\ \ Assume $(0,t_{0})\subset T,$ and $\alpha _{t_{0}}^{\ast }$ is one lift
of $\beta ([0,t_{0}))$ from $p_{0}$. We claim $\underset{t\rightarrow
t_{0}^{-}}{\lim }\alpha _{t_{0}}^{\ast }(t)$ must be a point in $%
f^{-1}(\beta (t_{0}))\cap U,$ and then $(0,t_{0}]\subset T$. The sequence $%
\{\alpha _{t_{0}}^{\ast }(t_{0}-\frac{1}{n}):nt_{0}>1,$ $n\in 
\mathbb{N}
\}$ always has a convergent subsequence $\{\alpha _{t_{0}}^{\ast }(t_{0}-%
\frac{1}{n_{k}})\}$, converging to some $z_{t_{0}}\in \overline{U}.$ Then we
have 
\begin{equation*}
f(z_{t_{0}})=\underset{k\rightarrow \infty }{\lim }f(\alpha _{t_{0}}^{\ast
}(t_{0}-\frac{1}{n_{k}}))=\underset{k\rightarrow \infty }{\lim }\beta (t_{0}-%
\frac{1}{n_{k}})=\beta (t_{0})\notin \partial \Sigma ,
\end{equation*}%
and then $z_{t_{0}}\in U\cap f^{-1}(\beta (t_{0})).$ By Lemma \ref{cov-1},
temporarily we may assume $z_{t_{0}}=0=\beta (t_{0}),$ and $f(z)=z^{d}$ on $%
\overline{\Delta }$. Then for all large $k,$ $\alpha _{t_{0}}^{\ast }(t_{0}-%
\frac{1}{n_{k}})\in \Delta ,$ and there is $t_{1}\in (0,t_{0}),$ such that $%
\beta ([t_{1},t_{0}])\subset \Delta .$ Evidently, $\beta
([t_{1},t_{0}))\subset $ $\Delta \backslash \{0\}$ has exactly $d$ lifts in $%
\Delta \backslash \{0\}$, denoted by $\widetilde{\gamma }_{1}(t)=\beta
(t)^{1/d},$ $\widetilde{\gamma }_{2}(t)=e^{\frac{2\pi i}{d}}\beta (t)^{1/d},$
$\cdots ,$ and $\widetilde{\gamma }_{d}(t)=e^{\frac{2\pi (d-1)i}{d}}\beta
(t)^{1/d}.$ $\alpha _{t_{0}}^{\ast }([t_{1},t_{0}))$ must be one of these
lifts, say $\widetilde{\gamma }_{j}(t)=e^{\frac{2\pi (j-1)i}{d}}\beta
(t)^{1/d}.$ Then 
\begin{equation*}
\underset{t\rightarrow t_{0}^{-}}{\lim }\alpha _{t_{0}}^{\ast }(t)=\underset{%
t\rightarrow t_{0}^{-}}{\lim }e^{\frac{2\pi (j-1)i}{d}}\beta (t)^{1/d}=e^{%
\frac{2\pi (j-1)i}{d}}\beta (t_{0})^{1/d}=0=z_{t_{0}},
\end{equation*}

and hence $\alpha _{t_{0}}(t)=\left\{ 
\begin{array}{c}
\alpha _{t_{0}}^{\prime }(t)\text{ for }0\leq t<t_{0}, \\ 
z_{t_{0}}\text{ for }t=t_{0},%
\end{array}%
\right. $ is a lift of $\beta ([0,t_{0}])$. Therefore, $\beta ([0,t_{0}])$
has exactly $d$ lifts from $p_{0},$ and $(0,t_{0}]\subset T$.

\ \ Assume for some $\delta \leq t<1,$ $(0,t]\subset T.$ The subarc $\beta
([0,t])$ has exactly $d$ lifts $p_{0}\overset{\alpha _{1,t}}{\rightarrow }%
p_{1,t},\cdots ,$ $p_{0}\overset{\alpha _{d,t}}{\rightarrow }p_{d,t}$ from $%
p_{0}.$ We claim that $t<\sup T.$ Since $\beta \cap \partial \Sigma
=\{f(p_{0})\}$ and $\beta ^{\circ }\cap CV(\Sigma )=\emptyset ,$ we have $%
p_{1,t},\cdots ,$ $p_{d,t}\in U\backslash C(\Sigma ).$ Then by Lemma \ref%
{cov-1}, $f$ is locally homeomorphic at $p_{1,t},\cdots ,$ $p_{d,t}$, and
thus each lift in $\{\alpha _{1,t},\cdots ,\alpha _{d,t}\}$ could be
slightly extended, in neighbourhoods of $p_{1,t},\cdots ,$ $p_{d,t}$
respectively. In other words, there exists $t_{1}\in (t,1],$ such that the
subarc $\beta ([0,t_{1}])$ also has exactly $d$ lifts $p_{0}\overset{\alpha
_{1,t_{1}}}{\rightarrow }p_{1,t_{1}},\cdots ,$ $p_{0}\overset{\alpha
_{d,t_{1}}}{\rightarrow }p_{d,t_{1}}$ from $p_{0}.$ Therefore, we have $%
(0,t_{1}]\subset T$, and $T$ is relatively open in $(0,1].$

\ \ In conclusion, $T=(0,1],$ and $\beta $ has exactly $d$ lifts $p_{0}%
\overset{\alpha _{1}}{\rightarrow }p_{1},\cdots ,$ $p_{0}\overset{\alpha _{d}%
}{\rightarrow }p_{d}$ from $p_{0}.$ Since $\beta \cap \partial \Sigma
=\{f(p_{0})\}$ and $\beta ^{\circ }\cap CV(\Sigma )=\emptyset ,$ we have $%
\alpha _{1}^{\circ },\cdots ,\alpha _{d}^{\circ }\subset U\backslash
C(\Sigma ).$
\end{proof}

\begin{lemma}
\label{lifts-1} Let $\Sigma =(f,\overline{U})\in \mathbf{F},$ $p_{0}\in U,$
and let $f(p_{0})\overset{\beta ^{\prime }}{\rightarrow }q_{2}$ be a simple
path on $S,$ such that $\beta ^{\prime \circ }\cap CV(\Sigma )=\emptyset .$
Then there exists a subarc $f(p_{0})\overset{\beta }{\rightarrow }q_{1}$ of $%
\beta ^{\prime },$ such that the following hold.

(1) $\beta $ has exactly $d$ lifts $p_{0}\overset{\alpha _{1}}{\rightarrow }%
p_{1},\cdots ,$ $p_{0}\overset{\alpha _{d}}{\rightarrow }p_{d}$ from $p_{0},$
and $d=v_{f}(p_{0}).$

(2) $\alpha _{1}^{\circ },\cdots ,\alpha _{d}^{\circ }\subset U\backslash
C(\Sigma ).$

(3) Either $\beta =\beta ^{\prime },$ or for some $j=1,2,\cdots ,d,$ $%
p_{j}\in \partial U.$
\end{lemma}

\begin{proof}
\ \ As the case in Lemma \ref{lift-3}, let $d=v_{f}(p_{0}),$ and let $\beta
^{\prime }$ be parametrized by $t\in \lbrack 0,1].$ The initial subarc of $%
\beta ^{\prime }$ has exactly $d$ lifts from $p_{0}.$ We define%
\begin{equation*}
T\overset{def}{=}\left\{ t\in (0,1]:%
\begin{array}{c}
\text{the subarc }\beta ^{\prime }([0,t])\text{ of }\beta ^{\prime }\text{
has exactly }d\text{ lifts }\alpha _{1,t},\cdots , \\ 
\alpha _{d,t}\text{ from }p_{0},\text{ such that }\alpha _{1,t}^{\circ
},\cdots ,\alpha _{d,t}^{\circ }\subset U\backslash C(\Sigma ).%
\end{array}%
\right\}
\end{equation*}%
Evidently, for $t\in T$, we have $(0,t]\subset T,$ and so $T$ is an
interval. When $(0,t_{1})\subset T,$ as in Lemma \ref{lift-3}, the subarc $%
\beta ^{\prime }([0,t_{1}])$ of $\beta ^{\prime }$ also has exactly $d$
lifts $\alpha _{1,t_{1}},\cdots ,\alpha _{d,t_{1}}$ from $p_{0},$ such that $%
\alpha _{1,t_{1}}^{\circ },\cdots ,\alpha _{d,t_{1}}^{\circ }\subset
U\backslash C(\Sigma ),$ which implies $(0,t_{1}]\subset T$.

\ Thus, we have $T=(0,t_{2}]$ with $0<t_{2}\leq 1.$ Then the subarc $\beta
=\beta ^{\prime }([0,t_{2}])$ of $\beta ^{\prime }$ has exactly $d$ lifts $%
\alpha _{1,t_{2}},\cdots ,\alpha _{d,t_{2}}$ from $p_{0},$ such that $\alpha
_{1,t_{2}}^{\circ },\cdots ,\alpha _{d,t_{2}}^{\circ }\subset U\backslash
C(\Sigma ).$ The maximality of $t_{2}$ in $T$ implies either $t_{2}=1,$ or
for some $j=1,2,\cdots ,d,$ $p_{j}\in \partial U.$
\end{proof}

\begin{lemma}
\label{lifts-2} Let $\Sigma =(f,\overline{U})\in \mathbf{F},$ such that $%
C(\Sigma )\subset \partial U\cup f^{-1}(E_{q})$. Let $p_{0}\overset{\alpha
_{1}^{\prime }}{\rightarrow }p_{1}^{\prime }$ be a subarc of $\partial U,$
such that $\alpha _{1}^{\prime \circ }\cap (f^{-1}(E_{q})\cup C(\Sigma
))=\emptyset .$ Assume $p_{0}$ is not a folded point of $\Sigma ,$ and $%
\beta ^{\prime }=f(\alpha _{1}^{\prime })$ is a simple subarc of $\partial
\Sigma .$ Then there exists a subarc $f(p_{0})\overset{\beta }{\rightarrow }%
q_{1}$ of $\beta ^{\prime },$ such that the following hold.

(1) $\beta $ has exactly $d$ lifts $p_{0}\overset{\alpha _{1}}{\rightarrow }%
p_{1},\cdots ,$ $p_{0}\overset{\alpha _{d}}{\rightarrow }p_{d}$ from $p_{0},$
and $d=v_{f}(p_{0}).$

(2) $\alpha _{1}$ is a subarc of $\alpha _{1}^{\prime },$ and $\alpha
_{2}^{\circ },\cdots ,\alpha _{d}^{\circ }\subset U\backslash C(f).$

(3) Either $\beta =\beta ^{\prime },$ or for some $j=2,3,\cdots ,d,$ $%
p_{j}\in \partial U.$
\end{lemma}

\begin{proof}
\ \ As the case in Lemma \ref{lift-3}, let $d=v_{f}(p_{0}),$ and let $\beta
^{\prime }$ be parametrized by $t\in \lbrack 0,1].$ The initial subarc of $%
\beta ^{\prime }$ has exactly $d$ lifts from $p_{0},$ and exactly one lift
is a subarc of $\alpha _{1}^{\prime }.$ We define%
\begin{equation*}
T\overset{def}{=}\left\{ t\in (0,1]:%
\begin{array}{c}
\text{the subarc }\beta ^{\prime }([0,t])\text{ of }\beta ^{\prime }\text{
has exactly }d\text{ lifts }\alpha _{1,t},\cdots ,\alpha _{d,t} \\ 
\text{from }p_{0},\text{ such that }\alpha _{1,t}\subset \alpha _{1}^{\prime
}\text{ and }\alpha _{2,t}^{\circ },\cdots ,\alpha _{d,t}^{\circ }\subset
U\backslash C(\Sigma ).%
\end{array}%
\right\}
\end{equation*}%
Evidently, $T$ is an interval. When $(0,t_{1})\subset T,$ as in Lemma \ref%
{lift-3}, the subarc $\beta ^{\prime }([0,t_{1}])$ of $\beta ^{\prime }$
also has exactly $d$ lifts $\alpha _{1,t_{1}},\cdots ,\alpha _{d,t_{1}}$
from $p_{0},$ such that $\alpha _{1,t_{1}}\subset \alpha _{1}^{\prime },$
and $\alpha _{2,t_{1}}^{\circ },\cdots ,\alpha _{d,t_{1}}^{\circ }\subset
U\backslash C(\Sigma ),$ which implies $t_{1}\in T.$ Thus, $T=(0,t_{2}]$
with $0<t_{2}\leq 1,$ and then the subarc $\beta =\beta ^{\prime
}([0,t_{2}]) $ of $\beta ^{\prime }$ has exactly $d$ lifts $\alpha
_{1,t_{2}},\cdots ,\alpha _{d,t_{2}}$ from $p_{0},$ such that $\alpha
_{1,t_{2}}$ is a subarc of $\alpha _{1}^{\prime },$ and $\alpha
_{2,t_{2}}^{\circ },\cdots ,\alpha _{d,t_{2}}^{\circ }\subset U\backslash
C(\Sigma ).$ The maximality of $t_{2}$ in $T$ implies either $t_{2}=1$ (or
equivalently, $\beta =\beta ^{\prime }$), or for some $j=2,3,\cdots ,d,$ $%
p_{j}\in \partial U\cup C(\Sigma ).$

\ \ Throughout we assume $2\leq j\leq d.$ To prove (3), we claim $p_{j}\in
C(\Sigma )$ implies $p_{j}\in \partial U$ or $\beta =\beta ^{\prime }.$
Suppose $p_{j}\in C(\Sigma )\backslash \partial U\subset f^{-1}(E_{q}),$ and
then we have $f(p_{j})\in E_{q}\cap \beta .$ Since $\alpha _{1}^{\prime
\circ }\cap f^{-1}(E_{q})=\emptyset $, we have $\beta ^{\prime \circ }\cap
E_{q}=\emptyset $, and thus $\beta =\beta ^{\prime }$ must hold, and our
claim follows.
\end{proof}

In these three lemmas, since $\beta $ is simple, all lifts $\alpha
_{1},\cdots ,\alpha _{d}$ of $\beta $ are simple. We also claim that for
each pair $(i,j)$ with $i\neq j,$ $\alpha _{i}^{\circ }\cap \alpha
_{j}=\emptyset =\alpha _{j}^{\circ }\cap \alpha _{i}.$ Otherwise, we assume $%
z\in \alpha _{i}^{\circ }\cap \alpha _{j},$ and then $\beta $ has at least
two distinct lifts $\alpha _{i}$ and $\alpha _{j}$ through $z.$ Hence, $z\in
C(\Sigma )\cap \alpha _{i}^{\circ }$ and $f(z)\in CV(\Sigma )\cap \beta
^{\circ },$ which is a contradiction to the condition $CV(\Sigma )\cap \beta
^{\circ }=\emptyset $ in Lemma \ref{lifts-1} and Lemma \ref{lift-3}, or to
the fact $C(\Sigma )\cap \alpha _{i}^{\circ }=\emptyset $ in Lemma \ref%
{lifts-2}.

We may assume $\alpha _{1},\cdots ,\alpha _{d}$ are arranged in the
anti-clockwise order at $p_{0}.$ Evidently, for $j=1,2,\cdots ,d-1,$ $f$
maps the angle between $\alpha _{j}$ and $\alpha _{j+1}$ at $p_{0}$ to a
perigon at $f(p_{0}).$ In Lemma \ref{lifts-1}, $f$ also maps the angle
between $\alpha _{d}$ and $\alpha _{1}$ at $p_{0}$ to a perigon at $f(p_{0})$%
. In Lemma \ref{lifts-2}, each point $w\in S\backslash \beta $ near $%
f(p_{0}) $ has at most $v_{f}(p_{0})$ preimages near $p_{0},$ and for $%
j=1,2,\cdots ,d-1,$ $w$ has exactly one preimage between $\alpha _{j}$ and $%
\alpha _{j+1}. $ So there is at most one preimage of $w$ near $p_{0}$
between $\alpha _{d}$ and $\partial U\backslash \alpha _{1}.$ Similarly in
Lemma \ref{lift-3}, there is at most one preimage of $w$ near $p_{0},$
either between $\partial U $ and $\alpha _{1},$ or between $\alpha _{d}$ and 
$\partial U.$

\section{Operations to modify a surface}

In this\ section, we prove some results to modify a surface.

\subsection{ Cutting and sewing a surface}

In this subsection, two basic operations are introduced to deform $\partial
\Sigma $, namely cutting $\Sigma ,$ and sewing $\Sigma $ along two subarcs
of $\partial \Sigma $. These two operations are opposite to each other
somehow.

\begin{lemma}
\label{cut} Let $\Sigma =(f,\overline{U})\in \mathbf{F},$ and $a\in \partial
U.$ Assume $a\overset{\beta }{\rightarrow }b$ is a simple arc in $U\cup
\{a\},$ such that $\beta ^{\circ }\cap C(\Sigma )=\emptyset .$ In addition, $%
(f,\beta )$ is a simple piecewise analytic arc, such that $f(\beta ^{\circ
})\cap E_{q}=\emptyset $. Then there exists $\Sigma _{1}=(f_{1},\overline{%
\Delta })=(f\circ \varphi ,\overline{\Delta })\in \mathbf{F},$ such that the
following hold.

(1) $\varphi \in OPL(\overline{\Delta })$ maps $\overline{\Delta }\backslash
(-i\overset{\partial \Delta }{\rightarrow }i)$ homeomorphically onto $%
\overline{U}\backslash \beta .$

(2) $\varphi |_{-i\overset{\partial \Delta }{\rightarrow }1}\in Homeo^{+}(-i%
\overset{\partial \Delta }{\rightarrow }1,\beta ),$ and $\varphi |_{1\overset%
{\partial \Delta }{\rightarrow }i}\in Homeo^{+}(1\overset{\partial \Delta }{%
\rightarrow }i,-\beta ).$

(3) For each $w\in S\backslash f(\beta ),$ $n(\Sigma _{1},w)=n(\Sigma ,w),$
and then $A(\Sigma _{1})=A(\Sigma ).$

(4) $\partial \Sigma _{1}=\partial \Sigma +(f,\beta )+(f,-\beta )$ is a
partition of $\partial \Sigma _{1}$, and then $L(\partial \Sigma
_{1})=L(\partial \Sigma )+2L(f,\beta ).$

(5) For each $a_{j}\in E_{q}\backslash \{f(b)\},$ $\overline{n}(\Sigma
_{1},a_{j})=\overline{n}(\Sigma ,a_{j}),$ and $\overline{n}(\Sigma
_{1},f(b))=\overline{n}(\Sigma ,f(b))-1.$
\end{lemma}

\begin{proof}
\ \ By topology, there exists $\varphi \in OPL(\overline{\Delta }),$ such
that $\varphi $ maps $\overline{\Delta }\backslash (-i\overset{\partial
\Delta }{\rightarrow }i)$ homeomorphically onto $\overline{U}\backslash
\beta ,$ $\varphi |_{-i\overset{\partial \Delta }{\rightarrow }1}\in
Homeo^{+}(-i\overset{\partial \Delta }{\rightarrow }1,\beta ),$ and $\varphi
|_{1\overset{\partial \Delta }{\rightarrow }i}\in Homeo^{+}(1\overset{%
\partial \Delta }{\rightarrow }i,-\beta ).$ Then $\Sigma _{1}=(f\circ
\varphi ,\overline{\Delta })$ is a surface in $\mathbf{F},$ because $f\circ
\varphi \in OPL(\overline{\Delta }),$ and $\partial \Sigma _{1}$ is
piecewise analytic, since 
\begin{eqnarray*}
\partial \Sigma _{1} &=&(f\circ \varphi ,\partial \Delta )=(f,(\varphi
,\partial \Delta )) \\
&=&(f,\partial U+\beta +(-\beta ))=\partial \Sigma +(f,\beta )+(f,-\beta ).
\end{eqnarray*}%
Thus, $L(\partial \Sigma _{1})=L(\partial \Sigma )+2L(f,\beta ),$ and 
\begin{equation*}
A(\Sigma _{1})=A(f,\varphi (\Delta ))=A(f,U\backslash \beta
)=A(f,U)=A(\Sigma ).
\end{equation*}

\ \ For each $w\in S\backslash f(\beta ),$ $\varphi $ is a bijection from $%
f_{1}^{-1}(w)\cap \Delta $ onto $f^{-1}(w)\cap U,$ and then $\overline{n}%
(\Sigma _{1},w)=\overline{n}(\Sigma ,w).$ Especially, for each $a_{j}\in
E_{q}\backslash \{f(b)\},$ since either $a_{j}\notin f(\beta )$ or $%
a_{j}=f(a),$ we have $\overline{n}(\Sigma _{1},a_{j})=\overline{n}(\Sigma
,a_{j}).$ Since $\varphi ^{-1}(b)=1$ is on $\partial \Delta ,$ $\varphi $
maps $[f_{1}^{-1}(f(b))\cap \Delta ]\cup \{1\}$ bijectively onto $%
f^{-1}(f(b))\cap U,$ and hence $\overline{n}(\Sigma _{1},f(b))=\overline{n}%
(\Sigma ,f(b))-1.$
\end{proof}

Cutting a simply-connected surface inside leads to a doubly-connected
surface, and the proof of the following corollary about this fact is similar
and omitted.

\begin{corollary}
\label{cutin,C} For $\Sigma =(f,\overline{U})\in \mathbf{F},$ let $b_{0}%
\overset{\beta }{\rightarrow }b_{1}$ be a simple arc in $U,$ such that $%
\beta ^{\circ }\cap C(\Sigma )=\emptyset ,$ and $(f,\beta )$ is a simple
piecewise analytic arc. Then there exists a doubly-connected surface $\Sigma
_{1}=(f_{1},\overline{A})=(f\circ \varphi ,\overline{A}),$ such that the
following hold.

(1) $\partial A$ is the disjoint union of two Jordan curves $\partial _{in}A$
and $\partial _{ex}A$.

(2) $\varphi \in OPL(\overline{A})$ maps $\overline{A}\backslash \partial
_{in}A$ homeomorphically onto $\overline{U}\backslash \beta .$ In addition, $%
\partial _{in}A=\alpha _{1}+\alpha _{2},$ such that $\varphi |_{\alpha
_{1}}\in Homeo^{+}(\alpha _{1},\beta )$ and $\varphi |_{\alpha _{2}}\in
Homeo^{+}(\alpha _{2},-\beta ).$ Then up to a reparametrization, $%
(f_{1},\partial _{ex}A)=(f,\partial U)=\partial \Sigma .$ (See Definition %
\ref{isomorphic}.)

(3) For each $w\in S\backslash f(\beta ),$ $n(f_{1},w)=n(f,w),$ and then $%
A(\Sigma _{1})=A(\Sigma ).$

(4) For each $a\in E_{q},$ 
\begin{equation*}
\overline{n}(\Sigma _{1},a)=\overline{n}(\Sigma ,a)-\#[\beta \cap f^{-1}(a)].
\end{equation*}
\end{corollary}

The following lemma is the reversed process of Lemma \ref{cut}, and the
roles of $\Sigma $ and $\Sigma _{1}$ in these two lemmas are interchanged.

\begin{lemma}
\label{glue} For $\Sigma =(f,\overline{U})\in \mathbf{F}$, let $a\overset{%
\alpha }{\rightarrow }p$ and $p\overset{\beta }{\rightarrow }b$ be two
adjacent subarcs of $\partial U.$ Assume there exists $h\in
Homeo^{+}(-\alpha ,\beta ),$ such that $f|_{\beta }\circ h=f|_{\alpha }.$

(A) If $\alpha \cup \beta \subsetneqq \partial U$, then there exists $\Sigma
_{1}=(f_{1},\overline{\Delta })=\left( f\circ \psi ^{-1},\overline{\Delta }%
\right) \in \mathbf{F},$ such that the following hold.

(1) $\psi \in OPL(\overline{U})$ maps $\overline{U}\backslash (\alpha \cup
\beta )$ homeomorphically onto $\overline{\Delta }\backslash \lbrack 0,1],$
and $(\psi ,\beta )=[0,1]=(\psi ,-\alpha )$ up to a reparametrization. (See
Definition \ref{isomorphic} and Notation \ref{interval}.)

(2) For each $z\in \alpha ,$ $\psi (z)=\psi (h(z)),$ and so $f\circ \psi
^{-1}\in OPL(\overline{\Delta })$ is well-defined.

(3) For each $w\in S\backslash f(\alpha ),$ $n(\Sigma _{1},w)=n(\Sigma ,w),$
and then $A(\Sigma _{1})=A(\Sigma )$.

(4) $\partial \Sigma =\partial \Sigma _{1}+(f,\alpha )+(f,\beta )$ is a
partition of $\partial \Sigma ,$ and then $L(\partial \Sigma
_{1})=L(\partial \Sigma )-2L(f,\alpha ).$

(5) For each $a_{j}\in E_{q},$ $\overline{n}(\Sigma _{1},a_{j})=\overline{n}%
(\Sigma ,a_{j})+\#(\beta \backslash \{b\})\cap f^{-1}(a_{j}).$

(B) If $\partial U=\alpha +\beta ,$ then there exists a closed surface $%
\Sigma _{1}=\left( f_{1},S\right) ,$ such that $A(\Sigma _{1})=A(\Sigma ),$
and%
\begin{equation*}
\overline{n}(f_{1},E_{q})=\overline{n}(f,E_{q})+\#\{\alpha \cap
f^{-1}(E_{q})\}\geq \overline{n}(f,E_{q}).
\end{equation*}
\end{lemma}

\begin{proof}
\ \ When $\alpha \cup \beta \subsetneqq \partial U,$ we could replace $%
\Sigma $ by its isomorphic surface, and assume $\overline{U}=\overline{%
\Delta ^{+}},$ $\alpha =[-1,0],$ $\beta =[0,1],$ and for each $z\in \lbrack
-1,0],$ $h(z)=-z,$ as in Lemma 2.7\textbf{\ }in\textbf{\ \cite{Z2}}. Now $%
\psi (z)=z^{2}\in OPL(\overline{\Delta ^{+}})$ maps $\overline{\Delta ^{+}}%
\backslash \lbrack -1,1]$ homeomorphically onto $\overline{\Delta }%
\backslash \lbrack 0,1].$ The mapping 
\begin{equation*}
f_{1}(z)=f\circ \psi ^{-1}(z)=f(\sqrt{z})\in OPL(\overline{\Delta })
\end{equation*}%
is well-defined, independent of the choices of $\sqrt{z}$ for $z\in (0,1].$
In fact, $\Sigma _{1}=(f_{1},\overline{\Delta })\in \mathbf{F}$ is the
desired surface, and the conditions (1), (2), (3), (4), and (5) about $%
\Sigma _{1}$ could be verified directly.

\ \ When $\partial U=\alpha +\beta ,$ we could replace $\Sigma $ by its
isomorphic surface, and assume $\overline{U}$ is the closed upper plane $%
\overline{\mathbb{H}},$ $\alpha =\overline{%
\mathbb{R}
^{-}},$ $\beta =\overline{%
\mathbb{R}
^{+}},$ and $h(z)=-z$ on $\overline{%
\mathbb{R}
^{-}}.$ Then the mapping $f_{1}(z)=f(\sqrt{z})\in OPL(\overline{%
\mathbb{C}
})$ (as a convention, $\sqrt{\infty }=\infty $) defines the desired closed
surface $\Sigma _{1}=(f_{1},\overline{%
\mathbb{C}
}),$ such that 
\begin{equation*}
A(\Sigma _{1})=A(f_{1},\overline{%
\mathbb{C}
})=A(f,\overline{\mathbb{H}})=A(\Sigma ).
\end{equation*}%
Furthermore, since $z\rightarrow z^{2}$ is a bijection from $%
(f^{-1}(E_{q})\cap \mathbb{H)\cup (}f^{-1}(E_{q})\cap \overline{%
\mathbb{R}
^{-}})$ onto $f_{1}^{-1}(E_{q}),$ we obtain 
\begin{equation*}
\overline{n}(f_{1},E_{q})=\overline{n}(f,E_{q})+\#\{\alpha \cap
f^{-1}(E_{q})\}\geq \overline{n}(f,E_{q}).
\end{equation*}
\end{proof}

\begin{remark}
\label{model} We introduce an equivalent relation on $\overline{U}$: $x\sim
y,$ iff either $x=y,$ or $y=h(x)\in \beta .$ When $\alpha \cup \beta \neq
\partial U,$ the quotient space $\overline{U}/\sim $ is a closed topological
disk, homeomorphic to $\overline{\Delta }.$ When $\partial U=\alpha +\beta ,$
the quotient space $\overline{U}/\sim $ is a topological sphere,
homeomorphic to $S.$ Let $[z]$ be the equivalent class of $z\in \overline{U}%
. $ We define $f^{\ast }([z])=f(z)$ for each $[z]\in \overline{U}/\sim .$
Then we obtain an \textquotedblleft abstract surface\textquotedblright\ $%
\Sigma ^{\ast }=(f^{\ast },\overline{U}/\sim ).$ $\Sigma _{1}=\left( f\circ
\psi ^{-1},\overline{\Delta }\right) $ constructed above is just a concrete
model of $\Sigma ^{\ast }.$ In fact, the mapping $\widetilde{\psi }:$ $%
[z]\rightarrow \psi (z)$ is a homeomorphism from $\overline{U}/\sim $ onto $%
\overline{\Delta }$ or $S,$ and thus it is reasonable to consider that $%
\Sigma ^{\ast }$ is isomorphic to $\Sigma _{1}.$
\end{remark}

The following corollary to sew a doubly-connected surface into a
simply-connected surface, has a similar interpretation, and the proof is
omitted.

\begin{corollary}
\label{glue annulus,C} Let $\Sigma =(f,\overline{A})$ be a doubly-connected
surface, such that $\partial A$ consists of two disjoint Jordan curves $%
\partial _{in}A$ and $\partial _{ex}A.$ Suppose $\partial _{in}A=\alpha
_{1}+\alpha _{2}$ and there exists $h\in Homeo^{+}(-\alpha _{1},\alpha
_{2}), $ such that $f|_{\alpha _{2}}\circ h=f|_{\alpha _{1}}.$ Then there
exists $\Sigma _{1}=(f_{1},\overline{\Delta })=(f\circ \psi ^{-1},\overline{%
\Delta })\in \mathbf{F},$ such that the following hold.

(1) $\psi \in OPL(\overline{A})$ maps $\overline{A}\backslash \partial _{in}A
$ homeomorphically onto $\overline{\Delta }\backslash \lbrack -\frac{1}{2},%
\frac{1}{2}],$ and $(\psi ,\alpha _{1})=[-\frac{1}{2},\frac{1}{2}]=(\psi
,-\alpha _{2})$ up to a reparametrization.

(2) For each $z\in \alpha _{1},$ $\psi (z)=\psi (h(z)),$ and so $f\circ \psi
^{-1}\in OPL(\overline{\Delta })$ is well-defined.

(3) For each $w\in S\backslash f(\alpha _{1}),$ $n(\Sigma _{1},w)=n(\Sigma
,w),$ and then $A(\Sigma _{1})=A(\Sigma ).$

(4) For each $a_{j}\in E_{q},$ $\overline{n}(\Sigma _{1},a_{j})=\overline{n}%
(\Sigma ,a_{j})+\#(\beta \cap \{a_{j}\}).$
\end{corollary}

\subsection{Removing non-special folded points}

In order to describe the relation of the boundaries of surfaces in $\mathbf{F%
},$ we need the following conception of closed subarcs. Roughly speaking, a
closed subarc $\alpha $ of a closed curve $\beta $ is a closed curve, which
is the sum of some subarcs of $\beta $ in the order.

\begin{definition}
\label{closed subarc} Let $\alpha _{1}$ and $\alpha _{2}$ be two Jordan
curves. A closed curve $(f_{2},\alpha _{2})$ is called \emph{a closed subarc
of} $(f_{1},\alpha _{1}),$ if either $(f_{2},\alpha _{2})=(f_{1},\alpha
_{1}) $ up to a reparametrization (see Definition \ref{isomorphic}), or the
following situation happens.

\ \ There are two partitions%
\begin{eqnarray*}
\alpha _{1} &=&p_{0}\overset{\gamma _{1}}{\rightarrow }p_{1}\overset{\gamma
_{2}}{\rightarrow }p_{2}\text{ }\cdots \text{ }p_{2m-1}\overset{\gamma _{2m}}%
{\rightarrow }p_{0}, \\
\alpha _{2} &=&\beta _{1}+\cdots +\beta _{m},
\end{eqnarray*}%
such that for each $j=1,\cdots ,m,$ $f_{1}(p_{2j-2})=f_{1}(p_{2j-1}),$ and $%
(f_{1},\gamma _{2j})=(f_{2},\beta _{j})$ up to a reparametrization.
\end{definition}

The definition of closed subarcs ensures the transitivity as follows. If $%
\alpha _{2}$ is a closed subarc of $\alpha _{3},$ then a closed subarc of $%
\alpha _{2}$ is also a closed subarc of $\alpha _{3}.$

\begin{definition}
\label{better} A surface $\Sigma _{2}\in \mathbf{F}$ is called \emph{better
than} $\Sigma _{1}\in \mathbf{F},$ if the following hold.

(1) $H(\Sigma _{2})\geq H(\Sigma _{1}),$ and $sum(\Sigma _{2})\leq
sum(\Sigma _{1}).$ (See Definition \ref{covering sum}.)

(2) For each $a\in E_{q},$ $\overline{n}(\Sigma _{2},a)\leq \overline{n}%
(\Sigma _{1},a).$

(3) $\partial \Sigma _{2}$ is a closed subarc of $(\varphi ,\partial \Sigma
_{1})$, where $\varphi $ is a rotation of $S.$
\end{definition}

Since a rotation $\varphi $ of $S$ preserves the length and the area,
Condition (3) implies $L(\partial \Sigma _{2})\leq L(\partial \Sigma _{1}).$
The rotation $\varphi $ appears only in the next section, and in this
section, we simply require $\partial \Sigma _{2}$ is a closed subarc of $%
\partial \Sigma _{1}.$ The relation \textquotedblleft
better\textquotedblright\ is obviously transitive. The proof of the main
theorem consists of several steps. In each step, a new surface $\Sigma
_{new}\in \mathbf{F}$ is constructed, which is better than the old surface $%
\Sigma _{old}\in \mathbf{F}$ in the last step. The first step is to remove
the non-special folded points by applying Lemma \ref{glue} repeatedly.

\begin{proposition}
\label{no-folded} For each $\Sigma _{1}=(f_{1},\overline{U})\in \mathbf{F}$
with $H(\Sigma _{1})\geq 0$, there exists $\Sigma _{3}\in \mathbf{F},$ such
that $\Sigma _{3}$ is better than $\Sigma _{1}$, and $\Sigma _{3}$ has no
non-special folded points.
\end{proposition}

\begin{proof}
\ \ When $\Sigma _{1}$ has no non-special folded points, $\Sigma _{3}=\Sigma
_{1}$ is desired. So we may assume $\Sigma _{1}$ has a non-special folded
point $p$. Let $a\overset{\alpha }{\rightarrow }p$ and $p\overset{\beta }{%
\rightarrow }b$ be maximal subarcs of $\partial U$, such that $\alpha
^{\circ }\cap \beta ^{\circ }=\emptyset ,$ $(f_{1},\alpha )=(f_{1},-\beta )$
up to a reparametrization, and $f_{1}(\beta ^{\circ })\cap E_{q}=\emptyset .$
We claim $a\neq b.$ Otherwise $\alpha +\beta =\partial U$, and by Lemma \ref%
{glue}, $\Sigma _{1}$ could be sewn into a closed surface $\Sigma ^{\prime
}, $ such that 
\begin{equation*}
\overline{n}(\Sigma ^{\prime },E_{q})=\overline{n}(\Sigma
_{1},E_{q})+\#(\beta \cap f_{1}^{-1}(E_{q}))\leq \overline{n}(\Sigma
_{1},E_{q})+1.
\end{equation*}%
Then by Proposition \ref{R<0}, 
\begin{equation*}
R(\Sigma _{1})\leq R(\Sigma ^{\prime })+4\pi \leq -4\pi ,
\end{equation*}%
which is a contradiction to $H(\Sigma _{1})\geq 0.$

\ \ Since $\alpha \cup \beta \neq \partial U,$ by Lemma \ref{glue}, $\Sigma
_{1}$ could be sewn into $\Sigma _{2}=(f_{1}\circ \psi ^{-1},\overline{%
\Delta })\in \mathbf{F},$ where $\psi \in OPL(\overline{U})$ maps $\overline{%
U}\backslash (\alpha \cup \beta )$ homeomorphically onto $\overline{\Delta }%
\backslash \lbrack 0,1],$ and $(\psi ,-\alpha )=(\psi ,\beta )=[0,1]$ up to
a reparametrization. Then 
\begin{equation*}
\partial \Sigma _{1}=(f_{1},\partial U\backslash (\alpha \cup \beta )+\alpha
+\beta )=\partial \Sigma _{2}+(f_{1},\alpha )+(f_{1},\beta ),
\end{equation*}%
and so $\partial \Sigma _{2}$ is a closed subarc of $\partial \Sigma _{1},$
and $L(\partial \Sigma _{2})<L(\partial \Sigma _{1}).$ Since $f_{1}(\beta
^{\circ })\cap E_{q}=\emptyset ,$ for each $a_{j}\in E_{q},$ $\overline{n}%
(\Sigma _{2},a_{j})=\overline{n}(\Sigma _{1},a_{j}).$ Since $A(\Sigma
_{2})=A(\Sigma _{1}),$ we obtain $R(\Sigma _{2})=R(\Sigma _{1})\geq 0,$ and
then $H(\Sigma _{2})\geq H(\Sigma _{1})\geq 0.$ Furthermore, each component $%
W$ of $S\backslash \partial \Sigma _{1}$ is contained in one component $V(W)$
of $S\backslash \partial \Sigma _{2},$ with $n(\Sigma _{1},W)=n(\Sigma
_{2},V(W)).$ Hence $sum(\Sigma _{2})\leq sum(\Sigma _{1}),$ and
consequently, $\Sigma _{2}$ is better than $\Sigma _{1}.$

\ \ When $f_{1}(a)=f_{1}(b)\in E_{q},$ either $\psi (a)=1=\psi (b)$ is a
special folded point of $\Sigma _{2},$ or $1$ is not a folded point. When $%
f_{1}(a)=f_{1}(b)\notin E_{q},$ by the maximality of $\alpha $ and $\beta ,$ 
$1$ is not a folded point of $\Sigma _{2}.$ Evidently, $\psi ^{-1}$ is an
injection from other folded points of $\Sigma _{2},$ to other folded points
of $\Sigma _{1}.$ In a word, $\Sigma _{2}$ has fewer non-special folded
points than $\Sigma _{1}$ does.

\ \ The previous process to remove non-special folded points, could be
applied to $\Sigma _{2}$ again, if $\Sigma _{2}$ still has non-special
folded points. In finite steps, we obtain $\Sigma _{3}\in \mathbf{F}$
without non-special folded points, such that $\Sigma _{3}$ is better than $%
\Sigma _{1}.$
\end{proof}

\subsection{ Removing interior non-special branch points}

In this section, we introduce how to move the interior branch points of a
surface to the boundary.

\begin{proposition}
\bigskip \label{in to bd} Let $\Sigma _{1}=(f_{1},\overline{\Delta })\in 
\mathbf{F},$ and let $p_{0}\in \Delta $ be a non-special branch point of $%
\Sigma _{1}.$ Then there exists $\Sigma _{2}=(f_{2},\overline{\Delta })\in 
\mathbf{F},$ such that $\Sigma _{2}$ is better than $\Sigma _{1},$ and one
of the following holds.

(i) $sum(\Sigma _{2})<sum(\Sigma _{1}).$

(ii) $f_{2}|_{\partial \Delta }=f_{1}|_{\partial \Delta }$, and 
\begin{equation*}
\#C(f_{2})\backslash (\partial \Delta \cup
f_{2}^{-1}(E_{q}))<\#C(f_{1})\backslash (\partial \Delta \cup
f_{1}^{-1}(E_{q})).
\end{equation*}
\end{proposition}

\begin{proof}
\ \ Let $d=v_{f_{1}}(p_{0})\geq 2,$ and $q_{0}=f_{1}(p_{0})\notin E_{q}.$
Let $q_{0}\overset{\beta _{1}}{\rightarrow }a_{1}$ be a simple path with $%
a_{1}\in E_{q}$, such that $\beta _{1}^{\circ }\cap (E_{q}\cup CV(\Sigma
_{1}))=\emptyset .$ Then by Lemma \ref{lifts-1}, there is a subarc $q_{0}%
\overset{\beta }{\rightarrow }q_{1}$ of $\beta _{1},$ such that $\beta $ has
exactly $d$ lifts $p_{0}\overset{\alpha _{1}}{\rightarrow }p_{1},\cdots ,$ $%
p_{0}\overset{\alpha _{d}}{\rightarrow }p_{d}$ from $p_{0},$ arranged in the
clockwise order at $p_{0}$. In addition, $q_{1}=a_{1}\in E_{q},$ or for some 
$j=1,2,\cdots ,d,$ $p_{j}\in \partial \Delta .$ Let $h_{ij}\in
Homeo^{+}(\alpha _{i},\alpha _{j}),$ such that $f_{1}|_{\alpha
_{i}}=f_{1}|_{\alpha _{j}}\circ h_{ij}$ on $\alpha _{i}.$

\ \ Now there are only five cases to discuss. Case (1): for some $i\neq j,$ $%
p_{i}=p_{j}\in \partial \Delta ,$ and then $q_{1}=f_{1}(p_{i})=f_{1}(p_{j})%
\in CV(\Sigma _{1}),$ and so $q_{1}=a_{1}\in E_{q}.$ Case (2): for some $%
i\neq j,$ $p_{i}=p_{j}\in \Delta ,$ and then $q_{1}\in CV(\Sigma _{1}),$ and
so $q_{1}=a_{1}\in E_{q}.$ Case (3): $p_{1},\cdots ,p_{d}$ are distinct, and
there is only one $p_{j}\in \partial \Delta ,$ say $p_{1}\in \partial \Delta
.$ Case (4): $p_{1},\cdots ,p_{d}$ are distinct in $\Delta ,$ and then $%
f_{1}(p_{1})=q_{1}=a_{1}\in E_{q}.$ Case (5): for some $i\neq j,$ $%
p_{i},p_{j}\in \partial \Delta ,$ but $p_{i}\neq p_{j}.$ It doesn't matter
that two cases happen for some $\Sigma _{1}\in \mathbf{F}.$ For instance,
when $p_{1}=p_{2}\in \partial \Delta ,$ and $p_{3}=p_{4}\in \Delta ,$ we
could just deal with $\Sigma _{1}$ as in Case (1), and ignore the fact that
Case (2) also happens. In each case, the corresponding figure shows the
process to construct the desired new surface in $\mathbf{F}.$ See Notation %
\ref{interval} for the notations $[0,1],$ $[-\frac{1}{2},\frac{1}{2}]$ in
the figures. \FRAME{dtbpF}{4.4339in}{1.4892in}{0pt}{}{}{in-bd,case1,2.emf}{%
\special{language "Scientific Word";type "GRAPHIC";display
"USEDEF";valid_file "F";width 4.4339in;height 1.4892in;depth
0pt;original-width 13.3536in;original-height 7.5014in;cropleft "0";croptop
"1";cropright "1";cropbottom "0.4232";filename
'in-bd,case1,2.emf';file-properties "XNPEU";}}

\ \ In Case (1), we may assume the Jordan curve $\alpha _{j}+(-\alpha _{i})$
encloses a Jordan domain $D$ in $\Delta $. By Lemma \ref{glue}, $(f_{1},%
\overline{D})\in \mathbf{F}$ could be sewn into a closed surface $\Sigma
_{3}=(f_{1}\circ \psi _{S}^{-1},S),$ where $\psi _{S}\in OPL(\overline{D})$
maps $\overline{D}$ onto $S,$ and $\psi _{S}|_{\alpha _{i}}=\psi
_{S}|_{\alpha _{j}}\circ h_{ij}$ on $\alpha _{i}.$ Moreover, by topology,
there is $\psi \in OPL(\overline{\Delta }\backslash D),$ such that $\psi |_{%
\overline{\Delta }\backslash \overline{D}}\in Homeo^{+}(\overline{\Delta }%
\backslash \overline{D},\overline{\Delta }\backslash \lbrack 0,1]),$ $(\psi
,\alpha _{i})=[0,1]=(\psi ,\alpha _{j}),$ and $\psi |_{\alpha _{i}}=\psi
|_{\alpha _{j}}\circ h_{ij}$ on $\alpha _{i},$ as in the figure for Case 1.
Then $f_{2}\overset{def}{=}f_{1}\circ \psi ^{-1}\in OPL(\overline{\Delta })$
is well-defined, and we will prove $\Sigma _{2}=(f_{2},\overline{\Delta }%
)\in \mathbf{F}$ is better than $\Sigma _{1},$ and $sum(\Sigma
_{2})<sum(\Sigma _{1})$ later.

\ \ In Case (2), we may assume the Jordan curve $\alpha _{j}+(-\alpha _{i})$
encloses a Jordan domain $D$ in $\Delta .$ By Lemma \ref{glue}, $(f_{1},%
\overline{D})\in \mathbf{F}$ could be sewn into\ a closed surface $\Sigma
_{3}=(f_{1}\circ \psi _{S}^{-1},S),$ where $\psi _{S}\in OPL(\overline{D})$
maps $\overline{D}$ onto $S,$ and $\psi _{S}|_{\alpha _{i}}=\psi
_{S}|_{\alpha _{j}}\circ h_{ij}$ on $\alpha _{i}.$ Moreover, $(f_{1},%
\overline{\Delta }\backslash D)$ is a doubly-connected surface, and $\alpha
_{i}+(-\alpha _{j})=\partial _{in}(\overline{\Delta }\backslash D)$. By
topology, there exists $\psi \in OPL(\overline{\Delta }\backslash D)$, such
that $\psi $ maps $\overline{\Delta }\backslash \overline{D}$
homeomorphically onto $\overline{\Delta }\backslash \lbrack -\frac{1}{2},%
\frac{1}{2}],$ $(\psi ,\alpha _{i})=(\psi ,\alpha _{j})=[-\frac{1}{2},\frac{1%
}{2}],$ and $\psi |_{\alpha _{i}}=\psi |_{\alpha _{j}}\circ h_{ij}$ on $%
\alpha _{i},$ as in the figure for Case 2. By Corollary \ref{glue annulus,C}%
, $(f_{1},\overline{\Delta }\backslash D)$ could be sewn into a new surface%
\begin{equation*}
\Sigma _{2}=(f_{2},\overline{\Delta })=(f_{1}\circ \psi ^{-1},\overline{%
\Delta })\in \mathbf{F}.
\end{equation*}%
We claim $\Sigma _{2}$ is better than $\Sigma _{1},$ and $sum(\Sigma
_{2})<sum(\Sigma _{1}).$

\ \ In Case (1) or Case (2), since up to a reparametrization, 
\begin{equation*}
\partial \Sigma _{2}=(f_{1}\circ \psi ^{-1},\partial \Delta
)=(f_{1},\partial \Delta )=\partial \Sigma _{1},
\end{equation*}%
we have $L(\partial \Sigma _{2})=L(\partial \Sigma _{1}).$ For each $w\in
S\backslash \beta ,$ we have 
\begin{equation}
\overline{n}(\Sigma _{2},w)+\overline{n}(\Sigma _{3},w)=\#f_{1}^{-1}(w)\cap
(\Delta \backslash \overline{D})+\#f_{1}^{-1}(w)\cap D=\overline{n}(\Sigma
_{1},w).  \label{(1)}
\end{equation}%
Clearly, each $w\in E_{q}\backslash \{a_{1}\}$ is not in $\beta ,$ and thus (%
\ref{(1)}) also holds in this case. Since a component $W$ of $S\backslash
\partial \Sigma _{1}$ is also a component of $S\backslash \partial \Sigma
_{2},$ and $n(\Sigma _{2},W)+\deg (\Sigma _{3})=n(\Sigma _{1},W),$ we have $%
sum(\Sigma _{2})<sum(\Sigma _{1}).$ Furthermore, 
\begin{equation*}
A(\Sigma _{1})=A(f_{1},\overline{\Delta })=A(f_{1},\overline{\Delta }%
\backslash D)+A(f_{1},D)=A(\Sigma _{2})+A(\Sigma _{3}).
\end{equation*}

\ In the first two cases, since $\beta ^{\circ }\cap E_{q}=\emptyset ,$ we
have 
\begin{equation*}
\#f^{-1}(a_{1})\cap \alpha _{j}^{\circ }=\emptyset =\#f^{-1}(a_{1})\cap
\alpha _{i}^{\circ }.
\end{equation*}%
Then by Lemma \ref{glue},%
\begin{eqnarray*}
&&\overline{n}(\Sigma _{2},a_{1})+\overline{n}(\Sigma _{3},a_{1}) \\
&=&\#f^{-1}(a_{1})\cap (\Delta \backslash \overline{D})+\#f^{-1}(a_{1})\cap
\{p_{i}\}\cap \Delta +\#f^{-1}(a_{1})\cap D+\#f^{-1}(a_{1})\cap \alpha _{j}
\\
&=&\#f^{-1}(a_{1})\cap \lbrack (\Delta \backslash \overline{D})\cup
(\{p_{i}\}\cap \Delta )\cup D\cup (\alpha _{j}^{\circ })\cup (\alpha
_{i}^{\circ })]+\#f^{-1}(a_{1})\cap \alpha _{j} \\
&=&\#(f^{-1}(a_{1})\cap \Delta )+\#\{p_{j}\}=\overline{n}(\Sigma
_{1},a_{1})+1.
\end{eqnarray*}%
In conclusion, 
\begin{equation*}
\overline{n}(\Sigma _{2},E_{q})+\overline{n}(\Sigma _{3},E_{q})=\overline{n}%
(\Sigma _{1},E_{q})+1,
\end{equation*}%
and by Proposition \ref{R<0}, $R(\Sigma _{2})=R(\Sigma _{1})-4\pi -R(\Sigma
_{3})>R(\Sigma _{1}).$ Therefore, $H(\Sigma _{2})>H(\Sigma _{1}),$ and then $%
\Sigma _{2}$ is better than $\Sigma _{1}.$\FRAME{dtbpF}{5.3428in}{1.8403in}{%
0pt}{}{}{in-bd,case3.emf}{\special{language "Scientific Word";type
"GRAPHIC";display "USEDEF";valid_file "F";width 5.3428in;height
1.8403in;depth 0pt;original-width 10.0984in;original-height
7.5014in;cropleft "0.0318";croptop "1";cropright "0.9840";cropbottom
"0.4053";filename 'in-bd,case3.emf';file-properties "XNPEU";}}

\ \ In Case (3), $p_{1},\cdots ,p_{d}$ are distinct, and $p_{1}\in \partial
\Delta ,$ but $p_{2},\cdots ,p_{d}\in \Delta .$ By applying Lemma \ref{cut} $%
d$ times, $\overline{\Delta }$ may be cut along $\alpha _{1},\cdots ,\alpha
_{d}$, and these $d$ arcs split into $2d$ sequential subarcs $b_{0}\overset{%
\gamma _{1}}{\rightarrow }b_{1},\cdots ,$ $b_{2d-1}\overset{\gamma _{2d}}{%
\rightarrow }b_{2d}$ of $\partial \Delta $, to obtain $\Sigma
_{4}=(f_{1}\circ \varphi ,\overline{\Delta })\in \mathbf{F}$. Here $\varphi
\in OPL(\overline{\Delta })$ maps $\overline{\Delta }\backslash (\gamma
_{1}\cup \cdots \cup \gamma _{2d})$ homeomorphically onto $\overline{\Delta }%
\backslash (\alpha _{1}\cup \cdots \cup \alpha _{d}),$ and%
\begin{eqnarray*}
(\varphi ,-\gamma _{1}) &=&\alpha _{1}=(\varphi ,\gamma _{2d}),\text{ }%
(\varphi ,\gamma _{2})=\alpha _{2}=(\varphi ,-\gamma _{3}),\text{ }\cdots ,
\\
(\varphi ,\gamma _{2d-2}) &=&\alpha _{d}=(\varphi ,-\gamma _{2d-1}).\text{ }
\end{eqnarray*}%
Then, the following pairs of adjacent subarcs $\{\gamma _{1},\gamma
_{2}\},\cdots ,\{\gamma _{2d-1},\gamma _{2d}\}$ of $\partial \Delta $ could
be sewn together by Lemma \ref{glue}, resulting in 
\begin{equation*}
\Sigma _{2}=(f_{2},\overline{\Delta })=(f_{1}\circ \varphi \circ \psi ^{-1},%
\overline{\Delta })\in \mathbf{F}.
\end{equation*}%
Here, $\psi \in OPL(\overline{\Delta })$ maps $\overline{\Delta }\backslash
(\gamma _{1}\cup \cdots \cup \gamma _{2d})$ homeomorphically onto $\overline{%
\Delta }\backslash (\Gamma _{1}\cup \cdots \cup \Gamma _{d}),$ where $\Gamma
_{1},\cdots ,\Gamma _{d}$ are $d$ simple curves in $\Delta \cup \{p_{1}\},$
with the common terminal point $p_{1}.$ In addition, $\Gamma _{1}\backslash
\{p_{1}\},\cdots ,\Gamma _{d}\backslash \{p_{1}\}$ are pairwise disjoint,
and 
\begin{equation*}
\psi (-\gamma _{1})=\Gamma _{1}=\psi (\gamma _{2}),\cdots ,\text{ }\psi
(-\gamma _{2d-1})=\Gamma _{d}=\psi (\gamma _{2d}).
\end{equation*}%
By compositing a self-homeomorphism of $\overline{\Delta },$ we may assume $%
\varphi |_{\partial \Delta \backslash (\gamma _{1}\cup \cdots \cup \gamma
_{2d})}=\psi |_{\partial \Delta \backslash (\gamma _{1}\cup \cdots \cup
\gamma _{2d})},$ and then $f_{2}|_{\partial \Delta }=f_{1}|_{\partial \Delta
}.$

\ \ The figures (a) to (d) for Case (3) give an interpretation of the
previous process to construct $\Sigma _{2}$ from $\Sigma _{1}$. In (a), $%
d=v_{f_{1}}(p_{0})$ is assumed to be $3$. To describe how $\overline{\Delta }
$ is cut and sewn more intuitively, the domain $\overline{\Delta }$ of $%
\Sigma _{4}$ is drawn as the shapes in (b) and in (c). Moreover, all lifts $%
\alpha _{1},\cdots ,\alpha _{d}$ are drawn as line segments, although they
are usually curves. These figures work for each surface $\Sigma _{1}\in 
\mathbf{F}$ with $v_{f_{1}}(p_{0})=3$ in Case (3), up to homeomorphisms. We
claim $\Sigma _{2}$ is the desired surface in $\mathbf{F}.$

\ \ Since $f_{2}|_{\partial \Delta }=f_{1}|_{\partial \Delta }$, we have $%
\partial \Sigma _{2}=\partial \Sigma _{1}$ and $L(\partial \Sigma
_{2})=L(\partial \Sigma _{1}).$ By Lemma \ref{cut} and Lemma \ref{glue}, we
have $A(\Sigma _{2})=A(\Sigma _{4})=A(\Sigma _{1}),$ and for every $w\in
S\backslash \beta ,$ 
\begin{equation}
\overline{n}(\Sigma _{2},w)=\overline{n}(\Sigma _{4},w)=\overline{n}(\Sigma
_{1},w).  \label{(2)}
\end{equation}%
Each $w\in E_{q}\backslash \{q_{1}\}$ is not in $\beta ,$ and thus (\ref{(2)}%
) also holds in this case. Then for a component $W$ of $S\backslash \partial
\Sigma _{2}=S\backslash \partial \Sigma _{1}$, we have $n(\Sigma
_{2},W)=n(\Sigma _{1},W),$ $sum(\Sigma _{2})=sum(\Sigma _{1}).$

\ \ However, as for $q_{1},$ $\psi ^{-1}\circ \varphi $ is a\ bijection from 
$f_{2}^{-1}(q_{1})\backslash (\partial \Delta \cup \Gamma _{1}\cup \cdots
\cup \Gamma _{d})$ onto $f_{1}^{-1}(q_{1})\backslash (\partial \Delta \cup
\alpha _{1}\cup \cdots \cup \alpha _{d}).$ Thus, we have%
\begin{eqnarray*}
\overline{n}(\Sigma _{2},q_{1}) &=&\#f_{2}^{-1}(q_{1})\backslash (\partial
\Delta \cup \Gamma _{1}\cup \cdots \cup \Gamma _{d})+\#f_{2}^{-1}(q_{1})\cap
\Delta \cap (\Gamma _{1}\cup \cdots \cup \Gamma _{d}) \\
&=&\#f_{1}^{-1}(q_{1})\backslash (\partial \Delta \cup \alpha _{1}\cup
\cdots \cup \alpha _{d})+0 \\
&=&\#f_{1}^{-1}(q_{1})\cap \Delta -\#f_{1}^{-1}(q_{1})\cap \Delta \cap
(\alpha _{1}\cup \cdots \cup \alpha _{d}) \\
&=&\overline{n}(\Sigma _{1},q_{1})-\#\{p_{2},\cdots ,p_{d}\}=\overline{n}%
(\Sigma _{1},q_{1})-(d-1).
\end{eqnarray*}%
In conclusion, $\overline{n}(\Sigma _{2},E_{q})\leq \overline{n}(\Sigma
_{1},E_{q})$, and then $H(\Sigma _{2})\geq H(\Sigma _{1}),$ and therefore $%
\Sigma _{2}$ is better than $\Sigma _{1}.$

\ \ For each $j=1,2,\cdots ,d,$ $f_{1}$ maps the angle at $p_{0}$ between $%
\alpha _{j}$ and $\alpha _{j+1}$ ($\alpha _{d+1}=\alpha _{1}$) to a perigon,
and then $f_{2}$ maps the perigons at $\psi (b_{1}),$ $\psi (b_{3}),\cdots ,$
$\psi (b_{2d-1})$ to perigons. In other words, $\psi (b_{1}),$ $\psi
(b_{3}),\cdots ,$ $\psi (b_{2d-1})$ $\in \psi (\varphi ^{-1}(p_{0}))$ are $d$
regular points of $\Sigma _{2}.$ Conversely, $p_{1}=\psi (b_{0})=\cdots
=\psi (b_{2d})\in \partial \Delta $ is a boundary branch point of $\Sigma
_{2}.$ Since $\varphi \circ \psi ^{-1}$ is homeomorphic from $\Delta
\backslash (\Gamma _{1}\cup \cdots \cup \Gamma _{d})$ onto $\Delta
\backslash (\alpha _{1}\cup \cdots \cup \alpha _{d}),$ $\varphi \circ \psi
^{-1}$ is a bijection from 
\begin{equation*}
C(f_{2})\cap \Delta \backslash (\Gamma _{1}\cup \cdots \cup \Gamma
_{d})=C(f_{2})\cap \Delta
\end{equation*}%
onto 
\begin{equation*}
C(f_{1})\cap \Delta \backslash (\alpha _{1}\cup \cdots \cup \alpha
_{d})=(C(f_{1})\cap \Delta )\backslash \{p_{0},p_{2},\cdots ,p_{d}\}.
\end{equation*}%
Thus, $\varphi \circ \psi ^{-1}$ is bijective from $(C(f_{2})\cap \Delta
)\backslash f_{2}^{-1}(E_{q})$ onto $(C(f_{1})\cap \Delta )\backslash
(\{p_{0},p_{2},\cdots ,p_{d}\}\cup f_{1}^{-1}(E_{q})).$ Therefore, we have 
\begin{equation*}
\#C(f_{2})\backslash (\partial \Delta \cup
f_{2}^{-1}(E_{q}))=\#C(f_{1})\backslash (\partial \Delta \cup
f_{1}^{-1}(E_{q})\cup \{p_{0},p_{2},\cdots ,p_{d}\})
\end{equation*}%
\begin{equation*}
\leq \#C(f_{1})\backslash (\partial \Delta \cup f_{1}^{-1}(E_{q})\cup
\{p_{0}\})=\#C(f_{1})\backslash (\partial \Delta \cup f_{1}^{-1}(E_{q}))-1.
\end{equation*}%
In conclusion, $\Sigma _{2}$ is desired indeed.\FRAME{dtbpF}{5.1361in}{%
2.0565in}{0pt}{}{}{in-bd,case4.emf}{\special{language "Scientific Word";type
"GRAPHIC";display "USEDEF";valid_file "F";width 5.1361in;height
2.0565in;depth 0pt;original-width 13.3354in;original-height
7.5014in;cropleft "0.0812";croptop "1";cropright "0.9187";cropbottom
"0.4227";filename 'in-bd,case4.emf';file-properties "XNPEU";}}

\ \ In Case (4), $p_{1},\cdots ,p_{d}\in \Delta $ are distinct, and then $%
q_{1}=a_{1}\in E_{q}.$ For an annulus $A$ such that $\partial A$ is the
union of two Jordan curves $\partial _{in}A$ and $\partial _{ex}A$, there
exists $\varphi \in OPL(\overline{A})$, which maps $\overline{A}\backslash
\partial _{in}A$ homeomorphically onto $\overline{\Delta }\backslash (\alpha
_{1}\cup \cdots \cup \alpha _{d}),$ as in the figure for Case (4). In
addition, $(\varphi ,\gamma _{1})=\alpha _{1}=(\varphi ,-\gamma _{2}),\cdots
,$ $(\varphi ,\gamma _{2d-1})=\alpha _{d}=(\varphi ,-\gamma _{2d}),$ where 
\begin{equation*}
\partial _{in}A=b_{0}\overset{\gamma _{1}}{\rightarrow }b_{1}\overset{\gamma
_{2}}{\rightarrow }b_{2}\text{ }\cdots \text{ }b_{2d-1}\overset{\gamma _{2d}}%
{\rightarrow }b_{2d}\text{ (}b_{0}=b_{2d}\text{).}
\end{equation*}%
By Corollary \ref{cutin,C}, $\Sigma _{5}=(f_{1}\circ \varphi ,\overline{A})$
is a doubly-connected surface. By Corollary \ref{glue annulus,C}, the
following pairs of adjacent subarcs $\{\gamma _{2},\gamma _{3}\},\cdots
,\{\gamma _{2d-2},\gamma _{2d-1}\},$ $\{\gamma _{2d},\gamma _{1}\}$ of $%
\partial _{in}A$ could be sewn together respectively, resulting in a surface 
\begin{equation*}
\Sigma _{2}=(f_{2},\overline{\Delta })=(f_{1}\circ \varphi \circ \psi ^{-1},%
\overline{\Delta })\in \mathbf{F}.
\end{equation*}%
Here $\psi \in OPL(\overline{A})$ maps $\overline{A}\backslash \partial
_{in}A$ homeomorphically onto $\overline{\Delta }\backslash (\Gamma _{1}\cup
\cdots \cup \Gamma _{d}),$ where $\Gamma _{1},\cdots ,\Gamma _{d}$ are $d$
simple curves from $0$ in $\Delta ,$ as in the figure for Case (4). In
addition, $\Gamma _{1}\backslash \{0\},\cdots ,\Gamma _{d}\backslash \{0\}$
are pairwise disjoint, and $(\psi ,\gamma _{2})=\Gamma _{1}=(\psi ,-\gamma
_{3}),\cdots ,$ $(\psi ,\gamma _{2d})=\Gamma _{d}=(\psi ,-\gamma _{1}).$ By
compositing a self-homeomorphism of $\overline{\Delta },$ we may assume $%
\psi |_{\partial _{ex}A}=\varphi |_{\partial _{ex}A},$ and then $%
f_{2}|_{\partial \Delta }=f_{1}|_{\partial \Delta }.$

\ \ The figures (a) to (d) for Case (4) give an interpretation of the
previous process to construct $\Sigma _{2}$ from $\Sigma _{1}$. To describe
this process more intuitively, we assume $d=v_{f_{1}}(p_{0})=3,$ and the
domain $\overline{A}$ of $\Sigma _{5}$ is drawn as two shapes in (b) and in
(c), and all lifts $\alpha _{1},\cdots ,\alpha _{d}$ are drawn to be
straight. These figures work for each $\Sigma _{1}\in \mathbf{F}$ with $%
v_{f_{1}}(p_{0})=3$ in Case (4), up to homeomorphisms. We claim $\Sigma _{2}$
is the desired surface in $\mathbf{F}.$

\ \ Since $f_{2}|_{\partial \Delta }=f_{1}|_{\partial \Delta }$, we have $%
\partial \Sigma _{2}=\partial \Sigma _{1}$, and $L(\partial \Sigma
_{2})=L(\partial \Sigma _{1}).$ By Corollary \ref{cutin,C} and Corollary \ref%
{glue annulus,C}, we have $A(\Sigma _{2})=A(\Sigma _{5})=A(\Sigma _{1}),$
and for each $w\in S\backslash \beta ,$ 
\begin{equation}
\overline{n}(\Sigma _{2},w)=\overline{n}(\Sigma _{5},w)=\overline{n}(\Sigma
_{1},w).  \label{(3)}
\end{equation}%
Each $w\in E_{q}\backslash \{q_{1}\}$ is not in $\beta ,$ and thus (\ref{(3)}%
) also holds in this case. Then for each component $W$ of $S\backslash
\partial \Sigma _{2}=S\backslash \partial \Sigma _{1},$ we have $n(\Sigma
_{2},W)=n(\Sigma _{1},W),$ and then $sum(\Sigma _{2})=sum(\Sigma _{1}).$

\ However, as for $q_{1}=a_{1}\in E_{q},$ $\varphi \circ \psi ^{-1}$ is a\
bijection from $f_{2}^{-1}(a_{1})\backslash (\partial \Delta \cup \Gamma
_{1}\cup \cdots \cup \Gamma _{d})$ onto $f_{1}^{-1}(a_{1})\backslash
(\partial \Delta \cup \alpha _{1}\cup \cdots \cup \alpha _{d}).$ Then we have%
\begin{eqnarray*}
\overline{n}(\Sigma _{2},a_{1}) &=&\#f_{2}^{-1}(a_{1})\backslash (\partial
\Delta \cup \Gamma _{1}\cup \cdots \cup \Gamma _{d})+\#f_{2}^{-1}(a_{1})\cap
\Delta \cap (\Gamma _{1}\cup \cdots \cup \Gamma _{d}) \\
&=&\#f_{1}^{-1}(a_{1})\backslash (\partial \Delta \cup \alpha _{1}\cup
\cdots \cup \alpha _{d})+\#\{0\} \\
&=&\#f_{1}^{-1}(a_{1})\cap \Delta -\#f_{1}^{-1}(a_{1})\cap \Delta \cap
(\alpha _{1}\cup \cdots \cup \alpha _{d})+1 \\
&=&\overline{n}(\Sigma _{1},a_{1})-\#\{p_{1},\cdots ,p_{d}\}+1=\overline{n}%
(\Sigma _{1},a_{1})-(d-1).
\end{eqnarray*}%
Hence, $\overline{n}(\Sigma _{2},E_{q})=\overline{n}(\Sigma
_{1},E_{q})-(d-1),$ and thus $H(\Sigma _{2})>H(\Sigma _{1}).$ In conclusion, 
$\Sigma _{2}$ is better than $\Sigma _{1}$, and $f_{2}|_{\partial \Delta
}=f_{1}|_{\partial \Delta }.$

\ \ For each $j=1,2,\cdots ,d,$ $f_{1}$ maps the angle at $p_{0}$ between $%
\alpha _{j}$ and $\alpha _{j+1}$ ($\alpha _{d+1}=\alpha _{1}$) to a perigon,
and then $f_{2}$ maps the perigons at $\psi (b_{0}),$ $\psi (b_{2}),\cdots ,$
$\psi (b_{2d-2})$ to perigons. In other words, $\psi (b_{0}),$ $\psi
(b_{2}),\cdots ,$ $\psi (b_{2d-2})$ $\in \psi (\varphi ^{-1}(p_{0}))$ are $d$
regular points of $\Sigma _{2}.$ Conversely, because $%
f_{2}(0)=f_{1}(p_{1})=q_{1}\in E_{q},$ $0$ is a special branch point of $%
\Sigma _{2}.$ Since $\varphi \circ \psi ^{-1}$ is homeomorphic from $\Delta
\backslash (\Gamma _{1}\cup \cdots \cup \Gamma _{d})$ onto $\Delta
\backslash (\alpha _{1}\cup \cdots \cup \alpha _{d}),$ $\varphi \circ \psi
^{-1}$ is a bijection from $C(f_{2})\cap \Delta \backslash (\Gamma _{1}\cup
\cdots \cup \Gamma _{d}\cup f_{2}^{-1}(E_{q}))$ onto $C(f_{1})\cap \Delta
\backslash (\alpha _{1}\cup \cdots \cup \alpha _{d}\cup f_{1}^{-1}(E_{q})).$
Therefore, we have%
\begin{eqnarray*}
&&\#C(f_{2})\backslash (\partial \Delta \cup f_{2}^{-1}(E_{q})) \\
&=&\#C(f_{2})\cap \Delta \backslash (\Gamma _{1}\cup \cdots \cup \Gamma
_{d}\cup f_{2}^{-1}(E_{q}))+\#C(f_{2})\cap (\Gamma _{1}\cup \cdots \cup
\Gamma _{d})\backslash f_{2}^{-1}(E_{q}) \\
&=&\#C(f_{1})\cap \Delta \backslash (\alpha _{1}\cup \cdots \cup \alpha
_{d}\cup f_{1}^{-1}(E_{q}))+\#\{0\}\backslash f_{2}^{-1}(E_{q}) \\
&=&\#C(f_{1})\cap \Delta \backslash f_{1}^{-1}(E_{q})-\#C(f_{1})\cap (\alpha
_{1}\cup \cdots \cup \alpha _{d})\backslash f_{1}^{-1}(E_{q})+0 \\
&=&\#C(f_{1})\backslash (\partial \Delta \cup
f_{1}^{-1}(E_{q}))-\#\{p_{0}\}<\#C(f_{1})\backslash (\partial \Delta \cup
f_{1}^{-1}(E_{q})).
\end{eqnarray*}%
In conclusion, $\Sigma _{2}$ is desired indeed.\FRAME{dtbpF}{4.7253in}{%
1.785in}{0pt}{}{}{in-bd,case5.emf}{\special{language "Scientific Word";type
"GRAPHIC";display "USEDEF";valid_file "F";width 4.7253in;height
1.785in;depth 0pt;original-width 13.3354in;original-height 7.5014in;cropleft
"0";croptop "1";cropright "1";cropbottom "0.3187";filename
'in-bd,case5.emf';file-properties "XNPEU";}}

\ \ In Case (5), $\alpha _{i}\cup \alpha _{j}$ divides $\Delta $ into two
Jordan domains $V_{1}$ and $V_{2},$ such that $(-\alpha _{j})+\alpha _{i}$
is a subarc of $\partial V_{1}.$ $(f_{1}|_{\overline{V_{1}}},\overline{V_{1}}%
)$ and $(f_{1}|_{\overline{V_{2}}},\overline{V_{2}})\in \mathbf{F}$ have a
common folded point $p_{0}.$ By Lemma \ref{glue}, subarcs $\alpha
_{i},-\alpha _{j}$ of $\partial V_{1}$ and $\alpha _{j},-\alpha _{i}$ of $%
\partial V_{2}$ could be sewn together respectively, to obtain 
\begin{eqnarray*}
\Sigma _{2} &=&(f_{2},\overline{\Delta })=(f_{1}\circ \psi _{1}^{-1},%
\overline{\Delta })\in \mathbf{F}, \\
\text{ \ }\Sigma _{3} &=&(f_{3},\overline{\Delta })=(f_{1}\circ \psi
_{2}^{-1},\overline{\Delta })\in \mathbf{F}.
\end{eqnarray*}%
Here for $j=1,2,$ $\psi _{j}\in OPL(\overline{V_{j}})$ maps $\overline{V_{j}}%
\backslash (\alpha _{i}\cup \alpha _{j})$ homeomorphically onto $\overline{%
\Delta }\backslash \lbrack 0,1],$ and 
\begin{equation*}
\lbrack 0,1]=(\psi _{1},\alpha _{i})=(\psi _{1},\alpha _{j})=(\psi
_{2},\alpha _{j})=(\psi _{2},\alpha _{i}),
\end{equation*}%
as in the figure for Case (5). We claim that one of $\Sigma _{2}$ and $%
\Sigma _{3}$ is the desired surface in $\mathbf{F}.$

$\ \ \partial \Sigma _{1}$ could be partitioned into:\ \ 
\begin{eqnarray*}
\partial \Sigma _{1} &=&(f_{1},\partial \Delta )=(f_{1},\partial V_{1}\cap
\partial \Delta +\partial V_{2}\cap \partial \Delta ) \\
&=&(f_{1},\partial V_{1}\backslash (\alpha _{i}\cup \alpha
_{j}))+(f_{1},\partial V_{2}\backslash (\alpha _{i}\cup \alpha _{j})) \\
&=&(f_{1}\circ \psi _{1}^{-1},\partial \Delta )+(f_{1}\circ \psi
_{2}^{-1},\partial \Delta )=\partial \Sigma _{2}+\partial \Sigma _{3}.
\end{eqnarray*}%
Then we have $L(\partial \Sigma _{1})=L(\partial \Sigma _{2})+L(\partial
\Sigma _{3}),$ and $\partial \Sigma _{2},\partial \Sigma _{3}$ are closed
subarcs of $\partial \Sigma _{1}.$ For each $w\in S\backslash (\beta
\backslash \{q_{0}\}),$ we have 
\begin{eqnarray}
\overline{n}(f_{2},w)+\overline{n}(f_{3},w) &=&\#f_{1}^{-1}(w)\cap
V_{1}+\#f_{1}^{-1}(w)\cap V_{2}  \notag \\
&=&\#f_{1}^{-1}(w)\cap \Delta =\overline{n}(f_{1},w).  \label{(4)}
\end{eqnarray}%
Each $w\in E_{q}$ is not in $\beta \backslash \{q_{0}\},$ and thus (\ref{(4)}%
) also holds in this case. Then $\overline{n}(\Sigma _{1},E_{q})=\overline{n}%
(\Sigma _{2},E_{q})+\overline{n}(\Sigma _{3},E_{q}),$ and 
\begin{equation*}
A(\Sigma _{1})=A(f_{1},\Delta )=A(f_{1},V_{1})+A(f_{1},V_{2})=A(\Sigma
_{2})+A(\Sigma _{3}).
\end{equation*}%
Consequently, $R(\Sigma _{1})=R(\Sigma _{2})+R(\Sigma _{3}),$ and since $%
L(\partial \Sigma _{1})=L(\partial \Sigma _{2})+L(\partial \Sigma _{3}),$
either $H(\Sigma _{2})\geq H(\Sigma _{1}),$ or $H(\Sigma _{3})\geq H(\Sigma
_{1}).$ We may assume $H(\Sigma _{2})\geq H(\Sigma _{1}).$

\ \ Each component $W$ of $S\backslash \partial \Sigma _{1}$ is contained in
two components $W_{2}$ and $W_{3}$ of $S\backslash \partial \Sigma _{2}$ and 
$S\backslash \partial \Sigma _{3}$ respectively. For each $w\in W\backslash
(\beta \backslash \{q_{0}\}),$ we have $\overline{n}(\Sigma _{1},w)=%
\overline{n}(\Sigma _{2},w)+\overline{n}(\Sigma _{3},w),$ and then $n(\Sigma
_{1},W)=n(\Sigma _{2},W_{2})+n(\Sigma _{3},W_{3}).$ Thus, $sum(\Sigma
_{2})<sum(\Sigma _{1}),$ and $\Sigma _{2}$ is desired indeed.
\end{proof}

Intuitively, in Case (1) and Case (2), we say $\Sigma _{1}$ splits into a
surface $\Sigma _{2}\in \mathbf{F}$ and a closed surface $\Sigma _{3}.$ In
Case (5), we say $\Sigma _{1}$ splits into two surfaces $\Sigma _{2}$ and $%
\Sigma _{3}$ in $\mathbf{F}.$ In Case (3), we say the interior branch point $%
p_{0}$ is moved to the boundary branch point $p_{1}.$ In Case (4), the
interior non-special branch point $p_{0}$ is moved to the special branch
point $0.$ When $\Sigma _{1}$ splits, the covering sum must decrease. Then,
all non-special interior branch points of $\Sigma _{1}$ could be moved
either to the boundary, or to special branch points, until $\Sigma _{1}$
splits into new surfaces.

\begin{corollary}
\label{in-bd,C} For each $\Sigma _{1}=(f_{1},\overline{\Delta })\in \mathbf{F%
},$ there exists $\Sigma _{2}=(f_{2},\overline{\Delta })\in \mathbf{F},$
such that $\Sigma _{2}$ is better than $\Sigma _{1},$ and one of the
following holds.

(1) $sum(\Sigma _{2})<sum(\Sigma _{1}).$

(2) $f_{2}|_{\partial \Delta }=f_{1}|_{\partial \Delta }$, and $C(\Sigma
_{2})\subset \partial \Delta \cup f_{2}^{-1}(E_{q}).$
\end{corollary}

The proof is trivial as follows. Applying Proposition \ref{in to bd}
repeatedly to $\Sigma _{1}$ and the resulting surfaces in $\mathbf{F}$,
finally we obtain $\Sigma _{2}\in \mathbf{F}$ which is better than $\Sigma
_{1}.$ This process stops only when all non-special branch points are
removed, or the surface splits in one step. If the surface never splits,
then by Proposition \ref{in to bd}, $f_{2}|_{\partial \Delta
}=f_{1}|_{\partial \Delta }$, and $C(\Sigma _{2})\subset \partial \Delta
\cup f_{2}^{-1}(E_{q}).$

\subsection{ Moving the branch points along the boundary}

In this subsection, we introduce how to move non-special branch points along
the boundary.

\begin{proposition}
\label{bd-bd} Let $\Sigma _{1}=(f_{1},\overline{\Delta })\in \mathbf{F},$
and assume the following hold.

(a) $C(\Sigma _{1})\subset \partial \Delta \cup f_{1}^{-1}(E_{q}).$

(b) $p_{0}\in C(\Sigma _{1})\backslash f_{1}^{-1}(E_{q})\subset \partial
\Delta $ is not a folded point of $\Sigma _{1}$.

(c) $p_{0}\overset{\alpha _{1}^{\prime }}{\rightarrow }p_{1}^{\prime }$ is a
subarc of $\partial \Delta ,$ such that $\alpha _{1}^{\prime \circ }\cap
(C(\Sigma _{1})\cup f_{1}^{-1}(E_{q}))=\emptyset ,$ and $\beta ^{\prime }%
\overset{def}{=}(f_{1},\alpha _{1}^{\prime })$ is a simple subarc of $%
\partial \Sigma _{1}.$

\ Then there exists $\Sigma _{2}=(f_{2},\overline{\Delta })\in \mathbf{F},$
such that $\Sigma _{2}$ is better than $\Sigma _{1},$ and one of the
following holds.

(i) $sum(\Sigma _{2})<sum(\Sigma _{1}).$

(ii) $f_{1}|_{\partial \Delta }=f_{2}|_{\partial \Delta },$ $C(\Sigma
_{2})\subset \partial \Delta \cup f_{2}^{-1}(E_{q}),$ and 
\begin{equation*}
C(\Sigma _{1})\backslash (\{p_{0},p_{1}^{\prime }\}\cup
f_{1}^{-1}(E_{q}))=C(\Sigma _{2})\backslash (\{p_{1}^{\prime }\}\cup
f_{2}^{-1}(E_{q})).
\end{equation*}
\end{proposition}

\begin{proof}
\ \ Let $d=v_{f_{1}}(p_{0})\geq 2,$ and let $q_{0}$ be the initial point $%
f_{1}(p_{0})$ of $\beta ^{\prime }.$ By Lemma \ref{lifts-2}, there is a
subarc $q_{0}\overset{\beta }{\rightarrow }q_{1}$ of $\beta ^{\prime },$
such that the following hold.

\ (d) $\beta $ has exactly $d$ lifts $p_{0}\overset{\alpha _{1}}{\rightarrow 
}p_{1},\cdots ,$ $p_{0}\overset{\alpha _{d}}{\rightarrow }p_{d}$ from $%
p_{0}, $ arranged in the anticlockwise order at $p_{0}$.

\ (e) $p_{0}\overset{\alpha _{1}}{\rightarrow }p_{1}$ is a subarc of $p_{0}%
\overset{\alpha _{1}^{\prime }}{\rightarrow }p_{1}^{\prime },$ and $\alpha
_{2}^{\circ },\cdots ,\alpha _{d}^{\circ }\subset \Delta .$

\ (f) Either $\beta =\beta ^{\prime }$ (namely $p_{1}=p_{1}^{\prime }$), or
for some $j=2,\cdots ,d,$ $p_{j}\in \partial \Delta .$

\ \ By (c), we have $\beta ^{\prime \circ }\cap E_{q}=\emptyset ,$ and then $%
\left( \alpha _{1}^{\circ }\cup \cdots \cup \alpha _{d}^{\circ }\right) \cap
f_{1}^{-1}(E_{q})=\emptyset .$ Since $\alpha _{1}^{\prime \circ }\cap
C(\Sigma _{1})=\emptyset ,$ and $C(\Sigma _{1})\cap \Delta \subset
f_{1}^{-1}(E_{q})$, we have 
\begin{eqnarray*}
&&\left( \alpha _{1}^{\circ }\cup \alpha _{2}^{\circ }\cup \cdots \cup
\alpha _{d}^{\circ }\right) \cap C(\Sigma _{1}) \\
&\subset &(\alpha _{1}^{\circ }\cap C(\Sigma _{1}))\cup \lbrack \left(
\alpha _{2}^{\circ }\cup \cdots \cup \alpha _{d}^{\circ }\right) \cap
f_{1}^{-1}(E_{q})]=\emptyset .
\end{eqnarray*}%
Let $h_{ij}\in Homeo^{+}(\alpha _{i},\alpha _{j}),$ such that $%
f_{1}|_{\alpha _{i}}=f_{1}|_{\alpha _{j}}\circ h_{ij}$ on $\alpha _{i}.$

\ \ There are four cases to discuss. Case (1): for some $j=2,\cdots ,d,$ $%
p_{j}=p_{1}\in \partial \Delta ,$ and then $p_{1}=p_{1}^{\prime }\in
C(\Sigma _{1}).$ Case (2): $p_{2},\cdots ,p_{d}\in \Delta ,$ and for some
pair $(i,j)$ with $2\leq i<j\leq d,$ $p_{i}=p_{j}.$ Then $p_{i}\in C(\Sigma
_{1})\cap \Delta \subset f_{1}^{-1}(E_{q}),$ $q_{1}=f_{1}(p_{1})\in E_{q},$
and $p_{1}=p_{1}^{\prime }.$ Case (3): for some $j=2,\cdots ,d,$ $p_{j}\in
\partial \Delta \backslash \{p_{1}\}.$ Case (4): $p_{2},\cdots ,p_{d}\in
\Delta $ are distinct, and then by (f), we have $p_{1}=p_{1}^{\prime }$. In
each case, the corresponding figure shows the process to construct the
desired new surface in $\mathbf{F}.$ By Notation \ref{interval}, $[1,0]$ in
the figures means the oriented line segment in $%
\mathbb{C}
$ from $1$ to $0.$ \FRAME{dtbpF}{5.0799in}{1.9796in}{0pt}{}{}{bd-bd,case1.emf%
}{\special{language "Scientific Word";type "GRAPHIC";display
"USEDEF";valid_file "F";width 5.0799in;height 1.9796in;depth
0pt;original-width 13.3354in;original-height 7.5014in;cropleft "0";croptop
"1";cropright "1";cropbottom "0.2549";filename
'bd-bd,case1.emf';file-properties "XNPEU";}}

\ \ In Case (1), the Jordan curve $\alpha _{1}+(-\alpha _{j})$ encloses a
Jordan domain $D$. By Lemma \ref{glue}, $(f_{1}|_{\overline{D}},\overline{D}%
)\in \mathbf{F}$ could be sewn into a closed surface $\Sigma
_{3}=(f_{3},S)=(f_{1}\circ \psi _{S}^{-1},S),$ where $\psi _{S}\in OPL(%
\overline{D})$ maps $\overline{D}$ onto $S,$ and $\psi _{S}|_{\alpha
_{1}}=\psi _{S}|_{\alpha _{j}}\circ h_{1j}$ on $\alpha _{1},$ as in the
figure for Case (1). $U_{2}=\Delta \backslash \overline{D}$ is also a Jordan
domain, and $\Sigma _{2}^{\prime }=(f_{1}|_{\overline{U_{2}}},\overline{U_{2}%
})\in \mathbf{F}$ is isomorphic to some $\Sigma _{2}=(f_{2},\overline{\Delta 
})\in \mathbf{F}.$ We claim $\Sigma _{2}$ is the desired surface.

\ \ Since up to a reparametrization, 
\begin{equation*}
\partial \Sigma _{2}=\partial \Sigma _{2}^{\prime }=(f_{1},\partial \Delta
\backslash \alpha _{1})+(f_{1},\alpha _{j})=(f_{1},\partial \Delta
\backslash \alpha _{1})+(f_{1},\alpha _{1})=\partial \Sigma _{1},
\end{equation*}%
we have $L(\partial \Sigma _{2})=L(\partial \Sigma _{1}).$ For each $w\in
S\backslash \beta ,$ we have%
\begin{eqnarray}
\overline{n}(\Sigma _{1},w) &=&\#f_{1}^{-1}(w)\cap \Delta
=\#f_{1}^{-1}(w)\cap U_{2}+\#f_{1}^{-1}(w)\cap D  \notag \\
&=&\overline{n}(\Sigma _{2}^{\prime },w)+\#f_{3}^{-1}(w)=\overline{n}(\Sigma
_{2},w)+\overline{n}(\Sigma _{3},w).  \label{(5)}
\end{eqnarray}%
Each $w\in E_{q}\backslash \{q_{1}\}$ is not in $\beta ,$ and thus (\ref{(5)}%
) also holds in this case. Then for each component $W$ of $S\backslash
\partial \Sigma _{1}=S\backslash \partial \Sigma _{2},$ $n(\Sigma
_{2},W)+\deg (\Sigma _{3})=n(\Sigma _{1},W),$ and hence $sum(\Sigma
_{2})<sum(\Sigma _{1}).$ In addition, 
\begin{equation*}
A(\Sigma _{1})=A(f_{1},\overline{\Delta })=A(f_{1},U_{2})+A(f_{1},D)=A(%
\Sigma _{2})+A(\Sigma _{3}).
\end{equation*}

\ \ As for $q_{1}$, by Lemma \ref{glue}, 
\begin{eqnarray*}
&&\overline{n}(\Sigma _{2},q_{1})+\overline{n}(\Sigma _{3},q_{1}) \\
&=&\#f_{1}^{-1}(q_{1})\cap U_{2}+\#f_{1}^{-1}(q_{1})\cap
D+\#f_{1}^{-1}(q_{1})\cap \alpha _{j} \\
&=&\#f_{1}^{-1}(q_{1})\cap \Delta +\#\{p_{j}\}=\overline{n}(\Sigma
_{1},q_{1})+1.
\end{eqnarray*}%
In a word, 
\begin{equation*}
\overline{n}(\Sigma _{2},E_{q})+\overline{n}(\Sigma _{3},E_{q})\leq 
\overline{n}(\Sigma _{1},E_{q})+1.
\end{equation*}%
Thus by Proposition \ref{R<0}, we have 
\begin{equation*}
R(\Sigma _{2})\geq R(\Sigma _{1})-4\pi -R(\Sigma _{3})\geq R(\Sigma
_{1})+4\pi .
\end{equation*}%
Therefore, $\Sigma _{2}$ is better than $\Sigma _{1},$ and $sum(\Sigma
_{2})<sum(\Sigma _{1}).$

\ \ Case (2) is almost the same as Case (1) in Proposition \ref{in to bd},
and the discussion is omitted. In this case, as in the figure for Case (2),
we obtain $\Sigma _{2}\in \mathbf{F}$ such that $\Sigma _{2}$ is better than 
$\Sigma _{1},$ and $sum(\Sigma _{2})<sum(\Sigma _{1}).$\FRAME{dtbpF}{4.881in%
}{1.8135in}{0pt}{}{}{bd-bd,case3.emf}{\special{language "Scientific
Word";type "GRAPHIC";display "USEDEF";valid_file "F";width 4.881in;height
1.8135in;depth 0pt;original-width 13.3354in;original-height
7.5014in;cropleft "0";croptop "1";cropright "0.9731";cropbottom
"0.3185";filename 'bd-bd,case3.emf';file-properties "XNPEU";}}

\ \ In Case (3), $\alpha _{j}$ divides $\Delta $ into two Jordan domains $%
V_{1}$ and $V_{2},$ such that $\alpha _{1}\subset \partial V_{1}.$ Then both 
$\Sigma _{2}^{\prime }=(f_{1}|_{\overline{V_{2}}},\overline{V_{2}})$ and $%
\Sigma _{3}^{\prime }=(f_{1}|_{\overline{V_{1}}},\overline{V_{1}})$ are
surfaces in $\mathbf{F}$. $\Sigma _{2}^{\prime }$ is isomorphic to some $%
\Sigma _{2}=(f_{2},\overline{\Delta })\in \mathbf{F}.$ For $\Sigma
_{3}^{\prime }$, by Lemma \ref{glue}, $-\alpha _{j}$ and $\alpha _{1}$ in $%
\partial V_{1}$ could be sewn together, resulting in $\Sigma _{3}=(f_{3},%
\overline{\Delta })=(f_{1}|_{\overline{V_{1}}}\circ \psi ^{-1},\overline{%
\Delta })\in \mathbf{F}.$ Here, $\psi \in OPL(\overline{V_{1}})$ maps $%
\overline{V_{1}}\backslash (\alpha _{1}\cup \alpha _{j})$ homeomorphically
onto $\overline{\Delta }\backslash \lbrack 0,1],$ such that $(\psi ,\alpha
_{1})=(\psi ,\alpha _{j})=[0,1]$, and $\psi |_{\alpha _{1}}=\psi |_{\alpha
_{j}}\circ h_{1j}$ on $\alpha _{1},$ as in the figure for Case (3). We claim
one of $\Sigma _{2}$ and $\Sigma _{3}$ is the desired surface in $\mathbf{F}.
$

\ \ Evidently, $L(\partial \Sigma _{1})=L(\partial \Sigma _{2})+L(\partial
\Sigma _{3}),$ because $\partial \Sigma _{1}$ could be partitioned into 
\begin{eqnarray*}
\partial \Sigma _{1} &=&(f_{1},\partial \Delta )=(f_{1},\partial \Delta \cap
\partial V_{1})+(f_{1},\partial \Delta \cap \partial V_{2}) \\
&=&(f_{1},\partial V_{1}\backslash (\alpha _{1}\cup \alpha _{j}))+\beta
+(f_{1},\partial \Delta \cap \partial V_{2}) \\
&=&(f_{3},\partial \Delta )+(f_{1},\partial V_{2})=\partial \Sigma
_{3}+\partial \Sigma _{2}.
\end{eqnarray*}%
For each $w\in S\backslash (\beta \backslash \{q_{1}\}),$ we have 
\begin{equation}
\overline{n}(\Sigma _{1},w)=\overline{n}(f_{1}|_{V_{1}},w)+\overline{n}%
(f_{1}|_{V_{2}},w)=\overline{n}(\Sigma _{2},w)+\overline{n}(\Sigma _{3},w).
\label{(6)}
\end{equation}%
Especially, this equality holds for each $a_{j}\in E_{q},$ since $%
a_{j}\notin \beta \backslash \{q_{1}\}$. Hence, $\overline{n}(\Sigma
_{1},E_{q})=\overline{n}(\Sigma _{2},E_{q})+\overline{n}(\Sigma _{3},E_{q}),$
and then $R(\Sigma _{1})=R(\Sigma _{2})+R(\Sigma _{3}).$ Since $L(\partial
\Sigma _{1})=L(\partial \Sigma _{2})+L(\partial \Sigma _{3})$, either $%
H(\Sigma _{2})\geq H(\Sigma _{1})$ or $H(\Sigma _{3})\geq H(\Sigma _{1}).$

\ \ We may assume $H(\Sigma _{2})\geq H(\Sigma _{1}).$ Each component $W$ of 
$S\backslash \partial \Sigma _{1}$ is contained in two components $W_{2}$
and $W_{3}$ of $S\backslash \partial \Sigma _{2}$ and $S\backslash \partial
\Sigma _{3}$ respectively, and $n(\Sigma _{1},W)=n(\Sigma
_{2},W_{2})+n(\Sigma _{3},W_{3}).$ Therefore, we have $sum(\Sigma
_{2})<sum(\Sigma _{1})$, and the desired surface $\Sigma _{2}$ is better
than $\Sigma _{1}.$ \FRAME{dtbpF}{4.587in}{1.9796in}{0pt}{}{}{bd-bd,case4.emf%
}{\special{language "Scientific Word";type "GRAPHIC";display
"USEDEF";valid_file "F";width 4.587in;height 1.9796in;depth
0pt;original-width 13.3354in;original-height 7.5014in;cropleft "0";croptop
"1";cropright "1";cropbottom "0.2549";filename
'bd-bd,case4.emf';file-properties "XNPEU";}}

\ \ In Case (4), $p_{2},\cdots ,p_{d}\in \Delta $ are distinct, and $\beta
=\beta ^{\prime }.$ By applying Lemma \ref{cut} $(d-1)$ times, $\overline{%
\Delta }$ could be cut along $\alpha _{2},\cdots ,\alpha _{d},$ and $\alpha
_{2},\cdots ,\alpha _{d}$ split into $(2d-2)$ sequential subarcs $b_{1}%
\overset{\gamma _{1}}{\rightarrow }b_{2},\cdots ,$ $b_{2d-1}\overset{\gamma
_{2d-1}}{\rightarrow }b_{2d}$ of $\partial \Delta ,$ as in the figures for
Case (4). We obtain $\Sigma _{3}=(f_{1}\circ \varphi ,\overline{\Delta })\in 
\mathbf{F}.$ Here, $\varphi \in OPL(\overline{\Delta })$ maps $\overline{%
\Delta }\backslash (\gamma _{1}\cup \cdots \cup \gamma _{2d-1})$
homeomorphically onto $\overline{\Delta }\backslash (\alpha _{1}\cup \cdots
\cup \alpha _{d}),$ and $(\varphi ,\gamma _{1})=\alpha _{d}=(\varphi
,-\gamma _{2}),\cdots ,$ $(\varphi ,\gamma _{2d-1})=\alpha _{1}.$ By Lemma %
\ref{glue}, the following pairs of adjacent subarcs $\{\gamma _{2},\gamma
_{3}\},\cdots ,\{\gamma _{2d-2},\gamma _{2d-1}\}$ of $\partial \Delta $
could be sewn together, resulting in $\Sigma _{2}=(f_{2},\overline{\Delta }%
)=(f_{1}\circ \varphi \circ \psi ^{-1},\overline{\Delta })\in \mathbf{F}$.
Here, $\psi \in OPL(\overline{\Delta })$ maps $\overline{\Delta }\backslash
(\gamma _{1}\cup \cdots \cup \gamma _{2d-1})$ homeomorphically onto $%
\overline{\Delta }\backslash (\alpha _{1}\cup \Gamma _{2}\cup \cdots \cup
\Gamma _{d}),$ where $\Gamma _{2},\cdots ,\Gamma _{d}$ are $(d-1)$ simple
arcs in $\Delta \cup \{p_{1}\},$ with the common terminal point $p_{1}.$ In
addition, $\Gamma _{2}\backslash \{p_{1}\},\cdots ,\Gamma _{d}\backslash
\{p_{1}\}$ are pairwise disjoint, and $(\psi ,\gamma _{1})=\alpha _{1},$ $%
(\psi ,-\gamma _{2})=\Gamma _{2}=(\psi ,\gamma _{3}),\cdots ,$ $(\psi
,-\gamma _{2d-2})=\Gamma _{d}=(\psi ,\gamma _{2d-1}).$ By compositing a
self-homeomorphism of $\overline{\Delta }$ to $\psi ,$ we may assume $\psi
|_{\partial \Delta \backslash (\gamma _{1}\cup \cdots \cup \gamma
_{2d-1})}=\varphi |_{\partial \Delta \backslash (\gamma _{1}\cup \cdots \cup
\gamma _{2d-1})},$ and $\psi |_{\gamma _{1}}=h_{d,1}\circ \varphi |_{\gamma
_{1}}$ on $\gamma _{1}$. Then, we have $f_{1}|_{\partial \Delta
}=f_{2}|_{\partial \Delta },$ and $L(\partial \Sigma _{1})=L(\partial \Sigma
_{2}).$

\ \ The figures (a) to (d) for Case (4) show the process to construct $%
\Sigma _{2}$ from $\Sigma _{1}$. To be more intuitive, $d=v_{f_{1}}(p_{0})$
is assumed to be $3,$ and the domain $\overline{\Delta }$ of $\Sigma _{3}$
is drawn as two shapes in (b), (c). We claim $\Sigma _{2}$ is the desired
surface in $\mathbf{F}.$

\ \ By Lemma \ref{cut} and Lemma \ref{glue}, we have $A(\Sigma
_{2})=A(\Sigma _{4})=A(\Sigma _{1}),$ and for each $w\in S\backslash \beta ,$
\begin{equation*}
\overline{n}(\Sigma _{2},w)=\overline{n}(\Sigma _{3},w)=\overline{n}(\Sigma
_{1},w).
\end{equation*}
Especially, this equality holds for each $a_{j}\in E_{q}\backslash
\{q_{1}\}, $ since $a_{j}\notin \beta .$ Consequently, for each component $W$
of $S\backslash \partial \Sigma _{2}=S\backslash \partial \Sigma _{1}$, we
have $n(\Sigma _{2},W)=n(\Sigma _{1},W),$ and then $sum(\Sigma
_{2})=sum(\Sigma _{1}).$ As for $q_{1},$ $\psi \circ \varphi ^{-1}$ is a
bijection from $f_{1}^{-1}(q_{1})\cap \Delta \backslash (\alpha _{2}\cup
\cdots \cup \alpha _{d})$ to $f_{2}^{-1}(q_{1})\cap \Delta \backslash
(\Gamma _{2}\cup \cdots \cup \Gamma _{d}).$ Then, 
\begin{eqnarray*}
\overline{n}(\Sigma _{2},q_{1}) &=&\#f_{2}^{-1}(q_{1})\cap \Delta \backslash
(\Gamma _{2}\cup \cdots \cup \Gamma _{d})+\#f_{2}^{-1}(q_{1})\cap \Delta
\cap (\Gamma _{2}\cup \cdots \cup \Gamma _{d}) \\
&=&\#f_{1}^{-1}(q_{1})\cap \Delta \backslash (\alpha _{2}\cup \cdots \cup
\alpha _{d})+0 \\
&=&\#f_{1}^{-1}(q_{1})\cap \Delta -\#\{p_{2},\cdots ,p_{d}\}=\overline{n}%
(\Sigma _{1},q_{1})-(d-1).
\end{eqnarray*}%
Consequently,, we obtain $\overline{n}(\Sigma _{2},E_{q})\leq \overline{n}%
(\Sigma _{1},E_{q})$, $H(\Sigma _{2})\geq H(\Sigma _{1}),$ and $\Sigma _{2}$
is better than $\Sigma _{1}.$

\ \ For each $j=1,2,\cdots ,d-1,$ $f_{1}$ maps the angle at $p_{0}$ between $%
\alpha _{j}$ and $\alpha _{j+1}$ to a perigon at $f_{1}(p_{0})$, and then $%
f_{2}$ maps the perigons at $\psi (b_{3}),$ $\psi (b_{5}),\cdots ,$ $\psi
(b_{2d-1})$ to perigons. Furthermore, for each $w\in S\backslash \beta $
near $f_{1}(p_{0}),$ there is at most one preimage $f_{1}^{-1}(w)$ near $%
p_{0}$ between $\alpha _{d}$ and $\partial \Delta \backslash \alpha _{1},$
and then there is at most one preimage $f_{2}^{-1}(w)$ near $p_{0}.$ In
conclusion, $\psi (b_{1})=p_{0},$ $\psi (b_{3}),\cdots ,$ $\psi (b_{2d-1})$
are $d$ regular points of $\Sigma _{2}.$ Conversely, $p_{1}=p_{1}^{\prime
}\in \partial \Delta $ is the only new branch point of $\Sigma _{2}.$ By
Lemma \ref{cut} and Lemma \ref{glue}, $\psi \circ \varphi ^{-1}$ maps $%
C(\Sigma _{1})\backslash \{p_{0},p_{1}^{\prime }\}=C(\Sigma _{1})\backslash
(\alpha _{1}\cup \cdots \cup \alpha _{d})$ bijectively onto $C(\Sigma
_{2})\backslash \{p_{1}^{\prime }\}=C(\Sigma _{2})\backslash (\alpha
_{1}\cup \Gamma _{2}\cup \cdots \cup \Gamma _{d}).$ Because $\psi \circ
\varphi ^{-1}|_{\partial \Delta \backslash \alpha _{1}^{\circ }}=Id$ and 
\begin{equation*}
C(\Sigma _{1})\backslash (\{p_{0},p_{1}^{\prime }\}\cup
f_{1}^{-1}(E_{q}))\subset \partial \Delta \backslash \alpha _{1}^{\circ },
\end{equation*}%
we have 
\begin{equation*}
C(\Sigma _{1})\backslash (\{p_{0},p_{1}^{\prime }\}\cup
f_{1}^{-1}(E_{q}))=C(\Sigma _{2})\backslash (\{p_{1}^{\prime }\}\cup
f_{2}^{-1}(E_{q}))\subset \partial \Delta .
\end{equation*}%
Recall $p_{1}^{\prime }\in \partial \Delta ,$ and then $C(\Sigma
_{2})\subset \partial \Delta \cup f_{2}^{-1}(E_{q}),$ which implies $\Sigma
_{2}$ is the desired surface in $\mathbf{F}$ indeed.
\end{proof}

In Case (1) and Case (2), we say $\Sigma _{1}\in \mathbf{F}$ splits into $%
\Sigma _{2}\in \mathbf{F}$ and a closed surface $\Sigma _{3}$. In Case (3), $%
\Sigma _{1}\in \mathbf{F}$ splits into two surfaces in $\mathbf{F}$. In Case
(4), the non-special boundary branch point $p_{0}$ is moved to $p_{1}.$ So,
a non-special boundary branch point $p$ of $\Sigma \in \mathbf{F}$ could be
moved along $\partial \Sigma $ by applying Proposition \ref{bd-bd}
repeatedly, until $\Sigma $ splits, or the branch point $p$ is moved to a
special boundary branch point. The only trouble case is that $\Sigma $ never
splits, and $\partial \Sigma \cap E_{q}=\emptyset .$ The following corollary
to remove non-special boundary branch points is based on the previous idea.

\begin{corollary}
\label{bd-bd,C} Let $\Sigma _{1}=(f_{1},\overline{\Delta })\in \mathbf{F}$,
which has no non-special folded points.

When $\partial \Sigma _{1}\cap E_{q}\neq \emptyset ,$ there exists $\Sigma
_{2}=(f_{2},\overline{\Delta })\in \mathbf{F},$ such that $\Sigma _{2}$ is
better than $\Sigma _{1},$ and either (1) $sum(\Sigma _{2})<sum(\Sigma
_{1}), $ or (2) $f_{2}|_{\partial \Delta }=f_{1}|_{\partial \Delta }$ and $%
CV(\Sigma _{2})\subset E_{q}$.

When $\partial \Sigma _{1}\cap E_{q}=\emptyset ,$ for each $p^{\prime }\in
\partial \Delta ,$ there exists $\Sigma _{2}=(f_{2},\overline{\Delta })\in 
\mathbf{F},$ such that $\Sigma _{2}$ is better than $\Sigma _{1},$ and
either (i) $sum(\Sigma _{2})<sum(\Sigma _{1}),$ or (ii) $f_{2}|_{\partial
\Delta }=f_{1}|_{\partial \Delta }$ and $C(\Sigma _{2})\subset
f_{2}^{-1}(E_{q})\cup \{p^{\prime }\}.$
\end{corollary}

\begin{proof}
\ \ By Corollary \ref{in-bd,C}, there exists $\Sigma _{4}=(f_{4},\overline{%
\Delta })\in \mathbf{F}$ such that $\Sigma _{4}$ is better than $\Sigma
_{1}, $ and either (a) $sum(\Sigma _{4})<sum(\Sigma _{1}),$ or (b) $C(\Sigma
_{4})\subset \partial \Delta \cup f_{4}^{-1}(E_{q})$ and $f_{4}|_{\partial
\Delta }=f_{1}|_{\partial \Delta }$. Corollary \ref{bd-bd,C} holds in Case
(a), and throughout we only concern Case (b). Since $\Sigma _{1}$ has no
non-special folded points, $\Sigma _{4}$ also has none in Case (b).

\ \ Firstly we assume $\partial \Sigma _{4}\cap E_{q}\neq \emptyset .$
Corollary \ref{bd-bd,C} holds when $CV(\Sigma _{4})\subset E_{q},$ and
throughout we assume $CV(\Sigma _{4})\nsubseteq E_{q}.$ There are always two
points $p_{0}\in \partial \Delta \cap C(\Sigma _{4})\backslash
f_{4}^{-1}(E_{q})$ and $p_{\ast }\in \partial \Delta \cap f_{4}^{-1}(E_{q}),$
such that the subarc $p_{0}\overset{\gamma }{\rightarrow }p_{\ast }$ of $%
\partial \Delta $ satisfies $f_{4}(\gamma ^{\circ })\cap E_{q}=\emptyset $. $%
\gamma $ could be partitioned into 
\begin{equation*}
\gamma =p_{0}\overset{\alpha _{1}}{\rightarrow }p_{1}\overset{\alpha _{2}}{%
\rightarrow }p_{2}\overset{\alpha _{3}}{\rightarrow }\cdots p_{m-1}\overset{%
\alpha _{m}}{\rightarrow }p_{m},
\end{equation*}%
with $p_{m}=p_{\ast },$ such that for each $\alpha _{j}$, $f_{4}(\alpha
_{j})=\beta _{j}$ is simple, and $\alpha _{j}^{\circ }\cap C(\Sigma
_{4})=\emptyset .$ Then by Proposition \ref{bd-bd}, there exists $\Sigma
_{5}=(f_{5},\overline{\Delta })\in \mathbf{F},$ such that $\Sigma _{5}$ is
better than $\Sigma _{4},$ and either (c) $sum(\Sigma _{5})<sum(\Sigma
_{4}), $ or (d) $f_{5}|_{\partial \Delta }=f_{4}|_{\partial \Delta },$ and 
\begin{equation*}
C(\Sigma _{4})\backslash (\{p_{0},p_{1}\}\cup f_{4}^{-1}(E_{q}))=C(\Sigma
_{5})\backslash (\{p_{1}\}\cup f_{5}^{-1}(E_{q})).
\end{equation*}%
Corollary \ref{bd-bd,C} holds in Case (c), and throughout we only concern
Case (d). Then we have 
\begin{equation*}
C(\Sigma _{5})\cap \partial \Delta \subset \lbrack C(\Sigma _{4})\cap
\partial \Delta \backslash \{p_{0}\}]\cup \{p_{1}\},\text{ }C(\Sigma
_{5})\subset \partial \Delta \cup f_{5}^{-1}(E_{q}).
\end{equation*}%
For $j=2,3,\cdots ,m,$ $f_{5}(\alpha _{j})=f_{4}(\alpha _{j})=\beta _{j}$ is
still simple, and $\alpha _{j}^{\circ }\cap C(\Sigma _{5})=\emptyset .$
Intuitively speaking, the non-special boundary branch point $p_{0}$ is moved
to $p_{1}.$

\ \ Thus, Proposition \ref{bd-bd} could be applied repeatedly to $\Sigma
_{5} $ and to the new resulting surfaces. In each step, either the surface
splits, leading to a smaller covering sum, or the boundary branch point $%
p_{j}$ is moved to $p_{j+1}.$ In at most $m$ steps, we obtain $\Sigma
_{3}=(f_{3},\overline{\Delta })\in \mathbf{F},$ such that $\Sigma _{3}$ is
better than $\Sigma _{4},$ and either (e) $sum(\Sigma _{3})<sum(\Sigma
_{4}), $ or (f) $f_{3}|_{\partial \Delta }=f_{4}|_{\partial \Delta },$ $%
C(\Sigma _{3})\subset \partial \Delta \cup f_{3}^{-1}(E_{q}),$ and 
\begin{equation*}
C(\Sigma _{3})\cap \partial \Delta \subset \lbrack C(\Sigma _{4})\cap
\partial \Delta \backslash \{p_{0}\}]\cup \{p_{m}\}.
\end{equation*}%
Corollary \ref{bd-bd,C} holds in Case (e), and throughout we only concern
Case (f). Since $p_{m}\in f_{4}^{-1}(E_{q})\cap \partial \Delta ,$ we have 
\begin{eqnarray*}
\#C(\Sigma _{3})\backslash f_{3}^{-1}(E_{q}) &=&\#C(\Sigma _{3})\cap
\partial \Delta \backslash (f_{3}^{-1}(E_{q})\cap \partial \Delta ) \\
&\leq &\#[(C(\Sigma _{4})\cap \partial \Delta \backslash \{p_{0}\})\cup
\{p_{m}\}]\backslash (f_{4}^{-1}(E_{q})\cap \partial \Delta ) \\
&\leq &\#C(\Sigma _{4})\cap \partial \Delta \backslash \lbrack \{p_{0}\}\cup
f_{4}^{-1}(E_{q})] \\
&=&\#C(\Sigma _{4})\cap \partial \Delta \backslash f_{4}^{-1}(E_{q})-1.
\end{eqnarray*}%
In other words, we have $\#C(\Sigma _{3})\backslash
f_{3}^{-1}(E_{q})<\#C(\Sigma _{4})\backslash f_{4}^{-1}(E_{q}).$ In fact,
all the non-special boundary branch points between $p_{0}$ and $p_{m}$ are
moved to the special branch point $p_{m}.$

\ \ This process to reduce the number of non-special boundary branch points,
could also be applied repeatedly to $\Sigma _{3}$ and to the new resulting
surfaces, until the covering sum decreases, or all branch points are
special. Finally, we obtain the desired surface $\Sigma _{2}=(f_{2},%
\overline{\Delta })\in \mathbf{F},$ such that $\Sigma _{2}$ is better than $%
\Sigma _{4}$ and $\Sigma _{1},$ and either (1) $sum(\Sigma _{2})<sum(\Sigma
_{4})\leq sum(\Sigma _{1}),$ or (2) $CV(\Sigma _{2})\subset E_{q},$ and $%
f_{2}|_{\partial \Delta }=f_{4}|_{\partial \Delta }=f_{1}|_{\partial \Delta
}.$

\ \ Secondly, we assume $\partial \Sigma _{1}\cap E_{q}=\emptyset .$ By
Corollary \ref{in-bd,C}, there exists $\Sigma _{4}=(f_{4},\overline{\Delta }%
)\in \mathbf{F}$, such that $\Sigma _{4}$ is better than $\Sigma _{1},$ and
either (a) $sum(\Sigma _{4})<sum(\Sigma _{1}),$ or (b) $C(\Sigma
_{4})\subset \partial \Delta \cup f_{4}^{-1}(E_{q})$ and $f_{4}|_{\partial
\Delta }=f_{1}|_{\partial \Delta }$. We only concern Case (b), and we may
assume there exists $p_{0}\in C(\Sigma _{4})\backslash
(f_{4}^{-1}(E_{q})\cup \{p^{\prime }\})\subset \partial \Delta $, since
Corollary \ref{bd-bd,C} holds when such $p_{0}$ doesn't exist.

\ The subarc $p_{0}\overset{\partial \Delta }{\rightarrow }p^{\prime }$ of $%
\partial \Delta $ could be partitioned into 
\begin{equation*}
p_{0}\overset{\alpha _{1}}{\rightarrow }p_{1}\overset{\alpha _{2}}{%
\rightarrow }p_{2}\overset{\alpha _{3}}{\rightarrow }\cdots p_{m-1}\overset{%
\alpha _{m}}{\rightarrow }p_{m}
\end{equation*}%
with $p^{\prime }=p_{m}$, such that for each $\alpha _{j},$ $f_{4}(\alpha
_{j})=\beta _{j}$ is simple, and $\alpha _{j}^{\circ }\cap
C(f_{4})=\emptyset .$ By the previous method, there exists $\Sigma
_{3}=(f_{3},\overline{\Delta })\in \mathbf{F},$ such that $\Sigma _{3}$ is
better than $\Sigma _{4}$, and either (g) $sum(\Sigma _{3})<sum(\Sigma
_{4}), $ or (h) $f_{3}|_{\partial \Delta }=f_{4}|_{\partial \Delta }$ and 
\begin{equation*}
C(\Sigma _{3})\backslash (\{p^{\prime }\}\cup f_{3}^{-1}(E_{q}))\subset
C(\Sigma _{4})\backslash (\{p_{0}\}\cup f_{4}^{-1}(E_{q}))\subset \partial
\Delta .
\end{equation*}%
Corollary \ref{bd-bd,C} holds in Case (g), and throughout we only concern
Case (h).

\ \ This process reduces the number of non-special boundary branch points
other than $p^{\prime },$ which could be applied repeatedly to $\Sigma _{3}$
and to new resulting surfaces, until the surface splits, or $p^{\prime }$ is
the unique non-special boundary branch point. Finally, we obtain $\Sigma
_{2}=(f_{2},\overline{\Delta })\in \mathbf{F},$ such that $\Sigma _{2}$ is
better than $\Sigma _{1},$ and either (i) $sum(\Sigma _{2})<sum(\Sigma
_{4})\leq sum(\Sigma _{1}),$ or (ii) $f_{2}|_{\partial \Delta
}=f_{4}|_{\partial \Delta }=f_{1}|_{\partial \Delta }$ and $C(\Sigma
_{2})\subset f_{2}^{-1}(E_{q})\cup \{p^{\prime }\}.$
\end{proof}

\section{The proof of the main theorem}

In this section, we prove the main theorem for each $\Sigma \in \mathbf{F}$,
even if $\partial \Sigma \cap E_{q}=\emptyset .$

\subsection{Moving the branch points to an interior special branch point}

When $\partial \Sigma \cap E_{q}=\emptyset ,$ moving non-special branch
points along $\partial \Sigma $ may not achieve $CV(\Sigma )\subset E_{q}.$
Sometimes the boundary branch points could be moved to an interior special
point as follows.

\begin{proposition}
\label{bd-in,P} Let $\Sigma _{1}=(f_{1},\overline{\Delta })\in \mathbf{F},$
with the following conditions.

(a) $\partial \Sigma _{1}\cap E_{q}=\emptyset ,$ and $\Sigma _{1}$ has no
folded points.

(b) $f_{1}$ maps a subarc $\Gamma ^{\prime }$ of $\partial \Delta $
homeomorphically onto an admissible subarc $\Gamma =(f_{1},\Gamma ^{\prime
}) $ of $\partial \Sigma _{1},$ and $U$ is the component of $S\backslash
\partial \Sigma _{1}$ on the left hand side of $\Gamma .$ (See Lemma \ref%
{finite components}.)

(c) $a_{1}\in E_{q}\cap U.$

Then there exists $\Sigma _{2}=(f_{2},\overline{\Delta })\in \mathbf{F},$
such that $\Sigma _{2}$ is better than $\Sigma _{1},$ and one of the
following holds.

(i) $sum(\Sigma _{2})<sum(\Sigma _{1}).$

(ii) $CV(\Sigma _{2})\subset E_{q}$ and $f_{2}|_{\partial \Delta
}=f_{1}|_{\partial \Delta }$.
\end{proposition}

\begin{proof}
\ \ By Corollary \ref{in-bd,C}, there exists $\Sigma _{3}=(f_{3},\overline{%
\Delta })\in \mathbf{F},$ such that $\Sigma _{3}$ is better than $\Sigma
_{1},$ and either (A) $sum(\Sigma _{3})<sum(\Sigma _{1}),$ or (B) $C(\Sigma
_{3})\subset \partial \Delta \cup f_{3}^{-1}(E_{q})$ and $f_{3}|_{\partial
\Delta }=f_{1}|_{\partial \Delta }$. Proposition \ref{bd-in,P} holds
trivially in Case (A). Throughout we only concern Case (B), and then the
conditions (a), (b), (c) also hold for $\Sigma _{3}.$

\ \ Fix $p_{0}\in \Gamma ^{\prime \circ }.$ By Corollary \ref{bd-bd,C},
there exists $\Sigma _{4}=(f_{4},\overline{\Delta })\in \mathbf{F},$ such
that $\Sigma _{4}$ is better than $\Sigma _{3},$ and either (C) $sum(\Sigma
_{4})<sum(\Sigma _{3})\leq sum(\Sigma _{1}),$ or (D) $f_{4}|_{\partial
\Delta }=f_{3}|_{\partial \Delta }$ and $C(f_{4})\subset
f_{4}^{-1}(E_{q})\cup \{p_{0}\}.$ Proposition \ref{bd-in,P} holds in Case
(C), and throughout we only concern Case (D). When $v_{f_{4}}(p_{0})=1,$ we
have $C(f_{4})\subset f_{4}^{-1}(E_{q}),$ and $\Sigma _{4}$ is the desired
surface in $\mathbf{F}.$ Thus, we may assume $d=v_{f_{4}}(p_{0})\geq 2$.

\ \ By Lemma \ref{finite components}, $f_{4}(p_{0})\in \Gamma ^{\circ
}\subset \overline{U}.$ Let $f_{4}(p_{0})\overset{\beta }{\rightarrow }a_{1}$
be a simple path in $\overline{U},$ such that $\beta ^{\circ }\cap (\partial
\Sigma _{4}\cup E_{q})=\emptyset .$ Then, $\beta $ is on the left hand side
of $\partial \Sigma _{4}$ near $f_{4}(p_{0}).$ By Lemma \ref{lift-3}, $\beta 
$ has exactly $d$ lifts $p_{0}\overset{\alpha _{1}}{\rightarrow }%
p_{1},\cdots ,$ $p_{0}\overset{\alpha _{d}}{\rightarrow }p_{d}$ from $p_{0}$
in $\Delta \cup \{p_{0}\},$ arranged in the clockwise order at $p_{0}.$ $%
\alpha _{1},\cdots ,\alpha _{d}$ are simple, and $\alpha _{1}^{\circ
},\cdots ,\alpha _{d}^{\circ }$ are pairwise disjoint.

\ \ There are only two cases to discuss. Case (1): for some pair $(i,j)$
with $1\leq i<j\leq d,$ $p_{i}=p_{j}.$ Case (2): $p_{1},\cdots ,p_{d}$ are
distinct. Case (1) is almost the same as Case (1) in Proposition \ref{in to
bd}, and the discussion is omitted. In Case (1), $\Sigma _{4}$ splits into $%
\Sigma _{2}\in \mathbf{F}$ and a closed surface, such that $\Sigma _{2}$ is
better than $\Sigma _{4}$ (and than $\Sigma _{1}$), with $sum(\Sigma
_{2})<sum(\Sigma _{4})\leq sum(\Sigma _{1}).$ Proposition \ref{bd-in,P}
holds in Case (1), and throughout we only concern Case (2). By Notation \ref%
{interval}, $[0,p_{0}]$ in the following figure for Case (2) means the
oriented line segment in $%
\mathbb{C}
$ from $0$ to $p_{0},$ even if $p_{0}\notin 
\mathbb{R}
.$\FRAME{dtbpF}{4.913in}{1.6527in}{0pt}{}{}{bd-in,case2.emf}{\special%
{language "Scientific Word";type "GRAPHIC";display "USEDEF";valid_file
"F";width 4.913in;height 1.6527in;depth 0pt;original-width
13.3354in;original-height 7.5014in;cropleft "0.0323";croptop "1";cropright
"1";cropbottom "0.3824";filename 'bd-in,case2.emf';file-properties "XNPEU";}}

\ \ In Case (2), by applying Lemma \ref{cut} $d$ times, $\overline{\Delta }$
could be cut along $\alpha _{1},\cdots ,\alpha _{d},$ and $\alpha
_{1},\cdots ,\alpha _{d}$ split into $2d$ sequential subarcs $b_{0}\overset{%
\gamma _{1}}{\rightarrow }b_{1},\cdots ,$ $b_{2d-1}\overset{\gamma _{2d}}{%
\rightarrow }b_{2d}$ of $\partial \Delta .$ We obtain $\Sigma
_{5}=(f_{4}\circ \varphi ,\overline{\Delta })\in \mathbf{F}.$ Here, $\varphi
\in OPL(\overline{\Delta })$ maps $\overline{\Delta }\backslash (\gamma
_{1}\cup \cdots \cup \gamma _{2d})$ homeomorphically onto $\overline{\Delta }%
\backslash (\alpha _{1}\cup \cdots \cup \alpha _{d}),$ and $(\varphi ,\gamma
_{1})=\alpha _{1}=(\varphi ,-\gamma _{2}),\cdots ,$ and $(\varphi ,\gamma
_{2d-1})=\alpha _{d}=(\varphi ,-\gamma _{2d})$. By Lemma \ref{glue}, the
following pairs of adjacent subarcs $\{\gamma _{2},\gamma _{3}\},\cdots ,$ $%
\{\gamma _{2d-2},\gamma _{2d-1}\}$ of $\partial \Delta $ could be sewn
together, and a surface $\Sigma _{5}^{\prime }\in \mathbf{F}$ is
constructed. Then the corresponding arcs of $\gamma _{1}$ and $\gamma _{2d}$
become adjacent in $\Sigma _{5}^{\prime }$, which could be sewn together
again, resulting in $\Sigma _{2}=(f_{2},\overline{\Delta })=(f_{4}\circ
\varphi \circ \psi ^{-1},\overline{\Delta })\in \mathbf{F}$.

\ Here $\psi \in OPL(\overline{\Delta })$ maps $\overline{\Delta }\backslash
(\gamma _{1}\cup \cdots \cup \gamma _{2d})$ homeomorphically onto $\overline{%
\Delta }\backslash ([0,p_{0}]\cup \Gamma _{1}\cup \cdots \cup \Gamma _{d-1}),
$ where $\Gamma _{1},\cdots ,\Gamma _{d-1}$ are $(d-1)$ simple arcs from $0$
in $\Delta .$ Moreover, $[0,p_{0}]\backslash \{0\},$ $\Gamma _{1}\backslash
\{0\},\cdots ,$ $\Gamma _{d-1}\backslash \{0\}$ are pairwise disjoint, and $%
(\psi ,-\gamma _{1})=[0,p_{0}]=(\psi ,\gamma _{2d}),$ $(\psi ,\gamma
_{2})=\Gamma _{1}=(\psi ,-\gamma _{3}),$ $\cdots ,$ $(\psi ,\gamma
_{2d-2})=\Gamma _{d-1}=(\psi ,-\gamma _{2d-1}).$ By compositing a
self-homeomorphism of $\overline{\Delta }$ to $\psi ,$ we may also assume $%
\psi |_{\partial \Delta \backslash (\gamma _{1}\cup \cdots \cup \gamma
_{2d})}=\varphi |_{\partial \Delta \backslash (\gamma _{1}\cup \cdots \cup
\gamma _{2d})}.$ Then $f_{2}|_{\partial \Delta }=f_{4}|_{\partial \Delta },$
and $L(\partial \Sigma _{2})=L(\partial \Sigma _{4}).$ The figures for Case
(2) show how to construct $\Sigma _{2}$ from $\Sigma _{1}.$ To be more
intuitive, $d=v_{f_{4}}(p_{0})$ is chosen as $3$, and the domain $\overline{%
\Delta }$ of $\Sigma _{5}$ is drawn as two homeomorphic shapes in (b) and
(c). We claim $\Sigma _{2}$ is desired.

\ \ By Lemma \ref{cut} and Lemma \ref{glue}, we have $A(\Sigma
_{2})=A(\Sigma _{5})=A(\Sigma _{4}),$ and for each $w\in S\backslash \beta ,$
we have $\overline{n}(\Sigma _{2},w)=\overline{n}(\Sigma _{5},w)=\overline{n}%
(\Sigma _{4},w).$ Especially, this equality holds for each $a_{j}\in
E_{q}\backslash \{a_{1}\},$ since $a_{j}\notin \beta .$ Then for each
component $W$ of $S\backslash \partial \Sigma _{2}=S\backslash \partial
\Sigma _{4},$ we have $n(\Sigma _{2},W)=n(\Sigma _{4},W),$ and then $%
sum(\Sigma _{2})=sum(\Sigma _{4}).$

\ \ As for $a_{1}\in E_{q},$ $\psi \circ \varphi ^{-1}$ is a bijection\ from 
$f_{4}^{-1}(a_{1})\backslash (\partial \Delta \cup \alpha _{1}\cup \cdots
\cup \alpha _{d})$ onto $f_{2}^{-1}(a_{1})\backslash (\partial \Delta \cup
\Gamma _{1}\cup \cdots \cup \Gamma _{d}).$ Then%
\begin{eqnarray*}
\overline{n}(f_{2},a_{1}) &=&\#f_{2}^{-1}(a_{1})\backslash (\partial \Delta
\cup \Gamma _{1}\cup \cdots \cup \Gamma _{d})+\#f_{2}^{-1}(a_{1})\cap \Delta
\cap (\Gamma _{1}\cup \cdots \cup \Gamma _{d}) \\
&=&\#f_{4}^{-1}(a_{1})\backslash (\partial \Delta \cup \alpha _{1}\cup
\cdots \cup \alpha _{d})+\#\{0\} \\
&=&\#f_{4}^{-1}(a_{1})\cap \Delta -\#f_{4}^{-1}(a_{1})\cap \Delta \cap
(\alpha _{1}\cup \cdots \cup \alpha _{d})+1 \\
&=&\overline{n}(f_{4},a_{1})-\#\{p_{1},\cdots ,p_{d}\}+1=\overline{n}%
(f_{4},a_{1})-d+1.
\end{eqnarray*}%
Thus, $\overline{n}(f_{2},E_{q})<\overline{n}(f_{4},E_{q}),$ $H(\Sigma
_{2})>H(\Sigma _{4}),$ and hence $\Sigma _{2}$ is better than $\Sigma _{4}$
(and than $\Sigma _{1}$).

\ \ Because for $j=1,2,\cdots ,d-1,$ $f_{4}$ maps the angle at $p_{0}$
between $\alpha _{j}$ and $\alpha _{j+1}$ to a perigon at $f_{4}(p_{0})$, $%
f_{2}$ maps the perigons at $\psi (b_{2}),\cdots ,$ $\psi (b_{2d-2})$ to
perigons at $f_{4}(p_{0})$. In addition, for each $w\in S\backslash \beta $
near $f_{4}(p_{0}),$ there are $(d-1)$ preimages $f_{4}^{-1}(w)$ near $p_{0}$
between $\alpha _{1}$ and $\alpha _{d},$ and at most one other preimage $%
f_{4}^{-1}(w)$ near $p_{0}$ either between $\partial \Delta $ and $\alpha
_{1},$ or between $\alpha _{d}$ and $\partial \Delta .$ Then, there is at
most one preimage $f_{2}^{-1}(w)$ near $p_{0}.$ Hence, $p_{0},$ $\psi
(b_{2}),\cdots ,$ $\psi (b_{2d-2})$ are $d$ regular points of $\Sigma _{2}.$
By Lemma \ref{cut} and Lemma \ref{glue}, $\psi \circ \varphi ^{-1}$ maps 
\begin{equation*}
C(\Sigma _{4})\backslash (\alpha _{1}\cup \cdots \cup \alpha _{d})=C(\Sigma
_{4})\backslash \{p_{0},p_{1},\cdots ,p_{d}\}
\end{equation*}%
bijectively onto 
\begin{equation*}
C(\Sigma _{2})\backslash (\Gamma _{1}\cup \cdots \cup \Gamma _{d})=C(\Sigma
_{2})\backslash \{0\}.
\end{equation*}%
Recall $C(f_{4})\subset f_{4}^{-1}(E_{q})\cup \{p_{0}\},$ and $0\in
f_{2}^{-1}(a_{1})$ is a special branch point of $\Sigma _{2}.$ For each $%
z\in C(\Sigma _{2})\backslash \{0\},$ 
\begin{equation*}
\varphi \circ \psi ^{-1}(z)\in C(\Sigma _{4})\backslash \{p_{0},p_{1},\cdots
,p_{d}\}\subset f_{4}^{-1}(E_{q}),
\end{equation*}%
and then $f_{2}(z)=f_{4}\circ \varphi \circ \psi ^{-1}(z)\in E_{q}.$
Therefore, we have $C(\Sigma _{2})\subset f_{2}^{-1}(E_{q}),$ and $\Sigma
_{2}$ is the desired surface in $\mathbf{F}$ indeed.
\end{proof}

When Proposition \ref{bd-in,P} couldn't be applied to any $a_{j}\in E_{q}$
and any admissible subarc $\Gamma $ of $\partial \Sigma _{1},$ the method to
rotate the surfaces is useful to solve the problem.

\begin{proposition}
\label{rotation} Let $\Sigma _{1}=(f_{1},\overline{\Delta })\in \mathbf{F},$
such that $\partial \Sigma _{1}\cap E_{q}=\emptyset .$ In addition, for each
admissible subarc $\Gamma $ of $\partial \Sigma _{1},$ the component $%
U_{\Gamma }$ of $S\backslash \partial \Sigma _{1}$ on the left hand side of $%
\Gamma $ is always disjoint from $E_{q}.$ Then there exists $\Sigma _{2}\in 
\mathbf{F},$ such that $\Sigma _{2}$ is better than $\Sigma _{1},$ $\partial
\Sigma _{2}\cap E_{q}\neq \emptyset ,$ and $\partial \Sigma _{2}=\varphi
(\partial \Sigma _{1})$ up to a reparametrization, where $\varphi $ is a
rotation of $S.$\ 
\end{proposition}

\begin{proof}
\ \ For each $a_{j}\in E_{q},$ $U_{j}$ denotes the component of $S\backslash
\partial \Sigma _{1}$ containing $a_{j}$ ($U_{1},\cdots ,U_{q}$ may not be
distinct). Then for each $j=1,\cdots ,q$ and each admissible subarc $\Gamma
\subset \partial U_{j}$ of $\partial \Sigma _{1},$ $U_{j}$ is always on the
right hand side of $\Gamma .$

\ \ There is some $a_{j}\in E_{q},$ say $a_{1}\in E_{q},$ and a continuous
family of rotations $\varphi _{t}$ of $S,$ such that (1) $\varphi _{0}=Id;$
(2) for each $t\in \lbrack 0,1),$ $\varphi _{t}(\partial \Sigma _{1})\cap
E_{q}=\emptyset ;$ (3) $\varphi _{1}(\partial \Sigma _{1})\cap
E_{q}=\{a_{1}\};$ and (4) $a_{1}\in \varphi _{1}(\Gamma _{1}^{\circ }),$
where $b_{1}\overset{\Gamma _{1}}{\rightarrow }b_{2}$ is an admissible
subarc of $\partial \Sigma _{1}.$ Conditions (1) and (2) could be achieved
by rotating $\partial \Sigma _{1}$ continuously, until $\varphi
_{t}(\partial \Sigma _{1})\cap E_{q}\neq \emptyset $ for the first time.
Conditions (3) and (4) could be achieved by a suitable perturbation. By a
refinement, we may assume $\Gamma _{1}$ is so short that $\varphi
_{1}(\Gamma _{1})\cap E_{q}=\{a_{1}\}.$

\ \ We claim that $\Gamma _{1}\subset \partial U_{1},$ and for $j=2,\cdots
,q,$ $\varphi _{1}^{-1}(a_{j})\in U_{j}.$ In fact, $\beta _{j}(t)|_{0\leq
t\leq 1}=\varphi _{t}^{-1}(a_{j})$ is a path from $a_{j}$ to $\varphi
_{1}^{-1}(a_{j}).$ For each $t_{0}\in \lbrack 0,1),$ $\varphi
_{t_{0}}(\partial \Sigma _{1})\cap E_{q}=\emptyset ,$ and then $\varphi
_{t_{0}}^{-1}(a_{j})\notin \partial \Sigma _{1}$. Thus, $\beta
_{j}(t)|_{0\leq t<1}$ is disjoint from $\partial \Sigma _{1},$ and so $\beta
_{j}(t)|_{0\leq t\leq 1}$ is contained in $\overline{U_{j}}.$ Hence, $%
\varphi _{1}^{-1}(a_{1})=\beta _{1}(1)\in \Gamma _{1}^{\circ }\cap \overline{%
U_{1}},$ and by Lemma \ref{finite components}, we have $\Gamma _{1}\subset
\partial U_{1}.$ Similarly, for $j=2,\cdots ,q,$ we have $\varphi
_{1}^{-1}(a_{j})=\beta _{j}(1)\in \overline{U_{j}}\backslash \partial \Sigma
_{1}=U_{j}.$ In other words, for $j=2,\cdots ,q,$ $a_{j}$ and $\varphi
_{1}(a_{j})$ are in the same component $V_{j}=\varphi _{1}(U_{j})$ of $%
S\backslash \varphi _{1}(\partial \Sigma _{1}).$

\ \ We have $m^{-}(\partial \Sigma _{1},\Gamma _{1})=0.$ Otherwise, $-\Gamma
_{1}$ is an admissible subarc of $\partial \Sigma _{1},$ such that $U_{1}\ni
a_{1}$ is on the left hand side of $-\Gamma _{1},$ contradiction. Let $%
U^{\prime }$ be the component of $S\backslash \partial \Sigma _{1}$ on the
left hand side of $\Gamma _{1}.$ Because 
\begin{equation*}
n(\Sigma _{1},U^{\prime })-n(\Sigma _{1},U_{1})=m^{+}(\partial \Sigma
_{1},\Gamma _{1})-m^{-}(\partial \Sigma _{1},\Gamma _{1})>0,
\end{equation*}%
we have $U^{\prime }\neq U_{1}$. By topology, there is a simple piecewise
analytic path $b_{1}\overset{\Gamma _{2}}{\rightarrow }b_{2}$ in $\overline{%
U^{\prime }}$ with $\Gamma _{2}^{\circ }\subset U^{\prime },$ sufficiently
close to $\Gamma _{1},$ such that the small Jordan domain $D_{12}$ enclosed
by $\Gamma _{1}+(-\Gamma _{2})$ satisfies the following three conditions.
(i) $\varphi _{1}(\overline{D_{12}})\cap E_{q}=\{a_{1}\}$. (ii) $\overline{%
U_{1}}$ is homeomorphic to $\overline{U_{1}\cup D_{12}}.$ (iii) $\overline{%
U^{\prime }}$ is homeomorphic to $\overline{U^{\prime }\backslash D_{12}}.$

\ \ For each component $V$ of $S\backslash \varphi _{1}(\partial \Sigma
_{1}) $ other than $\varphi _{1}(U_{1})$ and $\varphi _{1}(U^{\prime }),$
there exists $\psi _{V}\in Homeo^{+}(\overline{V},\overline{V}),$ such that $%
\psi _{V}|_{\partial V}=Id,$ and for each $a_{j}\in E_{q}\cap \varphi
_{1}^{-1}(V),$ $\psi _{V}(\varphi _{1}(a_{j}))=a_{j}.$ For $V_{1}\overset{def%
}{=}\varphi _{1}(U_{1}),$ there exists $\psi _{V_{1}}\in Homeo^{+}(\overline{%
V_{1}},\varphi _{1}(\overline{U_{1}\cup D_{12}})),$ such that $\psi
_{V_{1}}|_{\partial V_{1}\backslash \varphi _{1}(\Gamma _{1})}=Id,$ $\psi
_{V_{1}}(\varphi _{1}(\Gamma _{1}))=\varphi _{1}(\Gamma _{2}),$ and for each 
$a_{j}\in E_{q}\cap U_{1},$ $\psi _{V_{1}}(\varphi _{1}(a_{j}))=a_{j}.$ For $%
V^{\prime }\overset{def}{=}\varphi _{1}(U^{\prime }),$ there exists $\psi
_{V^{\prime }}\in Homeo^{+}(\overline{V^{\prime }},\varphi _{1}(\overline{%
U^{\prime }\backslash D_{12}})),$ such that $\psi _{V^{\prime }}|_{\partial
V^{\prime }\backslash \varphi _{1}(\Gamma _{1})}=Id,$ $\psi _{V^{\prime
}}|_{\varphi _{1}(\Gamma _{1})}=\psi _{V_{1}}|_{\varphi _{1}(\Gamma _{1})},$
and for each $a_{j}\in E_{q}\cap U^{\prime },$ $\psi _{V^{\prime }}(\varphi
_{1}(a_{j}))=a_{j}.$ We define $\psi (z)$ as 
\begin{equation*}
\psi (z)=\psi _{V}(z),\text{ for }z\in \overline{V},\text{ where }V\text{ is
a component of }S\backslash \varphi _{1}(\partial \Sigma _{1}).
\end{equation*}%
Then $\psi (z)\in Homeo^{+}(S,S)$ is well-defined, independent of the
choices of $V$ when $z\in \varphi _{1}(\partial \Sigma _{1}).$ By
definition, we have $\psi \circ \varphi _{1}|_{E_{q}}=Id,$ and $\psi \circ
\varphi _{1}|_{\partial \Sigma _{1}\backslash \Gamma _{1}}=Id.$

\ \ The surface $\Sigma _{3}=(f_{3},\overline{\Delta })=(\psi \circ \varphi
_{1}\circ f_{1},\overline{\Delta })\in \mathbf{F}$ satisfies for each $%
a_{j}\in E_{q},$ $\overline{n}(\Sigma _{3},a_{j})=\overline{n}(\Sigma
_{1},a_{j}).$ In addition, we have 
\begin{eqnarray*}
n(\Sigma _{3},\varphi _{1}(U^{\prime }\backslash \overline{D_{12}}))
&=&n(\Sigma _{3},\psi \circ \varphi _{1}(U^{\prime }))=n(\Sigma
_{1},U^{\prime }), \\
n(\Sigma _{3},\varphi _{1}(U_{1}\cup \Gamma _{1}^{\circ }\cup D_{12}))
&=&n(\Sigma _{3},\psi \circ \varphi _{1}(U_{1}))=n(\Sigma _{1},U_{1}),
\end{eqnarray*}%
and for each component $U$ of $S\backslash \partial \Sigma _{1}$ other than $%
U_{1}$ and $U^{\prime },$ $n(\Sigma _{3},\psi \circ \varphi
_{1}(U))=n(\Sigma _{1},U).$ Let $m$ be $m^{+}(\partial \Sigma _{3},\varphi
_{1}(\Gamma _{2}))=m^{+}(\partial \Sigma _{1},\Gamma _{1}).$ Then we have 
\begin{eqnarray*}
A(\Sigma _{3}) &=&A(\Sigma _{1})-mA(D_{12}), \\
L(\partial \Sigma _{3}) &=&L(\partial \Sigma _{1})-mL(\Gamma _{1})+mL(\Gamma
_{2}).
\end{eqnarray*}

\ \ There is a partition 
\begin{equation*}
\partial \Delta =\gamma _{1}+\gamma _{1}^{\prime }+\gamma _{2}+\gamma
_{2}^{\prime }+\cdots +\gamma _{m}+\gamma _{m}^{\prime },
\end{equation*}%
such that $f_{3}$ maps each $\gamma _{j}$ homeomorphically onto $\varphi
_{1}(\Gamma _{2}).$ Let $\overline{W_{1}},\cdots ,\overline{W_{m}}$ be $m$
pairwise disjoint closed Jordan domains in $%
\mathbb{C}
\backslash \overline{\Delta },$ such that for each $\overline{W_{j}},$ $%
\partial W_{j}\cap \partial \Delta =\gamma _{j}.$ Then $\overline{W}\overset{%
def}{=}\overline{\Delta }\cup \overline{W_{1}}\cup \cdots \cup \overline{%
W_{m}}$ is a closed Jordan domain. Recall that $\varphi _{1}(D_{12})$ is on
the right hand side of $\varphi _{1}(\Gamma _{2}).$ Thus, there exists $%
g_{j}\in Homeo^{+}(\overline{W_{j}},\varphi _{1}(\overline{D_{12}})),$ such
that $g_{j}|_{\gamma _{j}}=f_{3}|_{\gamma _{j}}.$ We define $f_{2}\in OPL(%
\overline{W})$ as 
\begin{equation*}
f_{2}(z)=\left\{ 
\begin{array}{c}
f_{3}(z)\text{ for }z\in \overline{\Delta }, \\ 
g_{j}(z)\text{ for }z\in \overline{W_{j}}.%
\end{array}%
\right.
\end{equation*}%
Then we claim $\Sigma _{2}=(f_{2},\overline{W})\in \mathbf{F}$ is the
desired surface.

\ \ In fact, $m$ coincident subarcs $\varphi _{1}(\Gamma _{2})$ of $\partial
\Sigma _{3}$ are replaced by $m$ coincident subarcs $\varphi _{1}(\Gamma
_{1})$ of $\partial \Sigma _{2}$. up to a reparametrization, 
\begin{eqnarray*}
\partial \Sigma _{2} &=&(f_{2},(\partial W_{1}\backslash \gamma _{1})+\gamma
_{1}^{\prime }+(\partial W_{2}\backslash \gamma _{2})+\gamma _{2}^{\prime
}+\cdots +(\partial W_{m}\backslash \gamma _{m})+\gamma _{m}^{\prime }) \\
&=&(g_{1},\partial W_{1}\backslash \gamma _{1})+(f_{3},\gamma _{1}^{\prime
})+\cdots +(g_{m},\partial W_{m}\backslash \gamma _{m})+(f_{3},\gamma
_{m}^{\prime }) \\
&=&\varphi _{1}(\Gamma _{1})+(f_{3},\gamma _{1}^{\prime })+\cdots +\varphi
_{1}(\Gamma _{1})+(f_{3},\gamma _{m}^{\prime })=\varphi _{1}(\partial \Sigma
_{1}),
\end{eqnarray*}%
and then $L(\partial \Sigma _{2})=L(\partial \Sigma _{1}).$ Furthermore, for
each component $U$ of $S\backslash \partial \Sigma _{1},$ $\varphi _{1}(U)$
is a component of $S\backslash \partial \Sigma _{2},$ and $n(\Sigma
_{2},\varphi _{1}(U))=n(\Sigma _{1},U),$ which implies $sum(\Sigma
_{2})=sum(\Sigma _{1}).$ We also have 
\begin{eqnarray*}
A(\Sigma _{2}) &=&A(f_{3},\overline{\Delta })+A(g_{1},W_{1})+\cdots
+A(g_{m},W_{m}) \\
&=&A(\Sigma _{3})+mA(D_{12})=A(\Sigma _{1}).
\end{eqnarray*}

\ \ For each $i=1,\cdots ,m,$ we have 
\begin{eqnarray*}
f_{2}^{-1}(E_{q})\cap \overline{W_{i}} &=&f_{2}^{-1}(\{a_{1}\})\cap 
\overline{W_{i}} \\
&=&g_{i}^{-1}(\{a_{1}\})\subset \partial W_{i}\backslash \gamma _{i}\subset
\partial W.
\end{eqnarray*}%
Then $f_{2}^{-1}(E_{q})\cap W=f_{3}^{-1}(E_{q})\cap \Delta ,$ which implies%
\begin{equation*}
\overline{n}(\Sigma _{2},E_{q})=\overline{n}(\Sigma _{3},E_{q})=\overline{n}%
(\Sigma _{1},E_{q}).
\end{equation*}%
In conclusion, $H(\Sigma _{2})=H(\Sigma _{1}),$ and then $\Sigma _{2}$ is
better than $\Sigma _{1},$ and $\partial \Sigma _{2}\cap E_{q}=\{a_{1}\}\neq
\emptyset .$
\end{proof}

\subsection{The main theorem}

In this subsection, the following main theorem is proved, which is slightly
stronger than the version in Section 1. The difference between two versions
is discussed in two remarks after this theorem.

\begin{theorem}
\label{branch-special} For each $\Sigma _{1}=(f_{1},\overline{U_{1}})\in 
\mathbf{F}$ with $H(\Sigma _{1})\geq 0,$ there exists $\Sigma _{0}\in 
\mathbf{F},$ such that $\Sigma _{0}$ is better than $\Sigma _{1}$, $%
CV(\Sigma _{0})\subset E_{q},$ and $\Sigma _{0}$ has no non-special folded
points.
\end{theorem}

\begin{proof}
\ \ This theorem is proved by induction on $sum(\Sigma _{1}).$ Firstly, when 
$sum(\Sigma _{1})=1,$ it follows that $C(\Sigma _{1})=\emptyset .$ By
Proposition \ref{no-folded}, there exists $\Sigma _{0}\in \mathbf{F}$, such
that $\Sigma _{0}$ is better than $\Sigma _{1},$ and $\Sigma _{0}$ has no
non-special folded points. Since $sum(\Sigma _{0})=1,$ we have $CV(\Sigma
_{0})=\emptyset $, and hence the main theorem holds when $sum(\Sigma
_{1})=1. $

\ \ By induction, we assume the main theorem holds for all $\Sigma \in 
\mathbf{F}$ with $sum(\Sigma )<sum(\Sigma _{1})$ and $H(\Sigma )\geq 0.$ By
Proposition \ref{no-folded}, we may assume $U_{1}=\Delta ,$ and $\Sigma _{1}$
has no non-special folded points. There are two possibilities to discuss,
either $\partial \Sigma _{1}\cap E_{q}=\emptyset ,$ or $\partial \Sigma
_{1}\cap E_{q}\neq \emptyset .$

\ \ Firstly, we assume $\partial \Sigma _{1}\cap E_{q}=\emptyset ,$ and then 
$\Sigma _{1}$ has no folded points. There are two cases. (1) The component $%
U_{j}$ of $S\backslash \partial \Sigma _{1}$ containing some $a_{j}\in
E_{q}, $ is on the left hand side of an admissible subarc $\Gamma $ of $%
\partial \Sigma _{1}.$ (2) For each admissible subarc $\Gamma $ of $\partial
\Sigma _{1}$, the component $U_{\Gamma }$ of $S\backslash \partial \Sigma
_{1}$ on the left hand side of $\Gamma $ contains no special points. In Case
(1), by Proposition \ref{bd-in,P}, there exists $\Sigma _{2}\in \mathbf{F}$,
such that $\Sigma _{2}$ is better than $\Sigma _{1},$ and either (a) $%
sum(\Sigma _{2})<sum(\Sigma _{1}),$ or (b) $CV(\Sigma _{2})\subset E_{q}$
and $f_{2}|_{\partial \Delta }=f_{1}|_{\partial \Delta }$. Situation (a) is
solved by the assumption of induction. In Situation (b), $\Sigma _{2}$ has
no folded points, and then the main theorem also holds.

\ \ In Case (2), by Proposition \ref{rotation}, there exists $\Sigma _{3}\in 
\mathbf{F}$, such that $\Sigma _{3}$ is better than $\Sigma _{1},$ $\partial
\Sigma _{3}\cap E_{q}\neq \emptyset ,$ and $\partial \Sigma _{3}$ is a
rotation of $\partial \Sigma _{1}.$ Case (2) is reduced to the possibility
that $\partial \Sigma _{1}\cap E_{q}\neq \emptyset .$

\ \ Secondly, we assume $\partial \Sigma _{1}\cap E_{q}\neq \emptyset .$ By
Corollary \ref{in-bd,C}, there exists $\Sigma _{4}=(f_{4},\overline{\Delta }%
)\in \mathbf{F}$, such that $\Sigma _{4}$ is better than $\Sigma _{1},$ and
either (c) $sum(\Sigma _{4})<sum(\Sigma _{1}),$ or (d) $f_{4}|_{\partial
\Delta }=f_{1}|_{\partial \Delta }$ and $C(\Sigma _{4})\subset \partial
\Delta \cup f_{4}^{-1}(E_{q}).$ Situation (c) is solved by the assumption of
induction. In Situation (d), $\Sigma _{4}$ has no non-special folded points, 
$H(\Sigma _{4})\geq 0$ and $\partial \Sigma _{4}\cap E_{q}\neq \emptyset .$
By Corollary \ref{bd-bd,C}, there exists $\Sigma _{5}=(f_{5},\overline{%
\Delta })\in \mathbf{F}$, such that $\Sigma _{5}$ is better than $\Sigma
_{4} $ (and than $\Sigma _{1}$), and either (e) $sum(\Sigma _{5})<sum(\Sigma
_{4})\leq sum(\Sigma _{1}),$ or (f) $CV(\Sigma _{5})\subset E_{q}$ and $%
f_{5}|_{\partial \Delta }=f_{4}|_{\partial \Delta }=f_{1}|_{\partial \Delta
}.$ Situation (e) is solved by the assumption of induction. In Situation
(f), $\Sigma _{5}$ has no non-special folded points, which is the desired
surface in $\mathbf{F}$. Now the whole proof is completed.
\end{proof}

\begin{remark}
By the notations above, since $\Sigma _{0}=(f_{0},\overline{\Delta })\in 
\mathbf{F}$ has no non-special folded points, $f_{0}|_{\partial \Delta }$ is
locally injective at each point $z\in \partial \Delta \backslash
f_{0}^{-1}(E_{q}).$ Together with $CV(\Sigma _{0})\subset E_{q},$ $f_{0}$ is
locally injective at each $z\in \overline{\Delta }\backslash
f_{0}^{-1}(E_{q}).$
\end{remark}

\begin{remark}
If we don't require $\partial \Sigma _{0}$ is a closed subarc of $\varphi
(\partial \Sigma _{1})$ ($\varphi $ is a rotation of $S$), then the
condition $H(\Sigma _{1})\geq 0$ in Theorem \ref{branch-special} is
unnecessary. For each $\Sigma _{1}\in \mathbf{F}$ with $H(\Sigma _{1})<0,$
there exists $\Sigma ^{\prime }\in \mathbf{F},$ such that $H(\Sigma ^{\prime
})\geq 0>H(\Sigma _{1}),$ $L(\partial \Sigma ^{\prime })\leq L(\partial
\Sigma _{1}),$ $\overline{n}(\Sigma ^{\prime })=1\leq \overline{n}(\Sigma
_{1}),$ and $\overline{n}(\Sigma ^{\prime },E_{q})=0;$ and then Theorem \ref%
{branch-special} could be applied to $\Sigma ^{\prime }$ instead.
\end{remark}

The requirement that $\partial \Sigma _{0}$ is a closed subarc of $\varphi
(\partial \Sigma _{1}),$ is useful to discuss the following families of
(simply-connected) polygonal surfaces.

\begin{definition}
A simple subarc of a great circle on $S$ is called a \emph{(spherical)} 
\emph{line segment}. A\emph{\ }curve is called \emph{polygonal }if it could
be partitioned into a finite number of line segments. $\Sigma \in \mathbf{F}$
is called \emph{a polygonal surface}, if $\partial \Sigma $ is a closed
polygonal curve. Let $\mathbf{F}_{P}(L,M,N)$ denote the family of all
polygonal surfaces $\Sigma \in \mathbf{F},$ such that $L(\partial \Sigma
)\leq L,$ $\overline{n}(\Sigma ,a_{j})\leq M$ for each $a_{j}\in E_{q},$ and 
$\partial \Sigma $ consists of at most $N$ line segments.
\end{definition}

Each closed subarc of a closed polygonal curve consisting of $N$ line
segments, is also polygonal, consisting of at most $N$ line segments. The
following theorem ensures that in order to study the constant $\sup
\{H(\Sigma )|\Sigma \in \mathbf{F}_{P}(L,M,N)\},$ we only have to consider
surfaces $\Sigma \in \mathbf{F}_{P}(L,M,N)$ such that $CV(\Sigma )\subset
E_{q}.$

\begin{theorem}
For each $L>0,$ $M\geq 0,$ $N\geq 3,$%
\begin{equation*}
\sup \{H(\Sigma )|\Sigma \in \mathbf{F}_{P}(L,M,N)\}=\sup \{H(\Sigma
)|\Sigma \in \mathbf{F}_{P}(L,M,N),\text{ }CV(\Sigma )\subset E_{q}\},
\end{equation*}
\end{theorem}

\begin{proof}
\ \ We claim for $\Sigma _{1}\in \mathbf{F}_{P}(L,M,N),$ there exists $%
\Sigma _{0}\in \mathbf{F}_{P}(L,M,N),$ such that $H(\Sigma _{0})\geq
H(\Sigma _{1})$ and $CV(\Sigma _{0})\subset E_{q}.$ When $H(\Sigma _{1})<0,$
evidently there exists $\Sigma _{0}\in \mathbf{F}_{P}(L,M,N),$ such that $%
\overline{n}(\Sigma _{0},E_{q})=0,$ $L(\partial \Sigma _{0})\leq L(\partial
\Sigma _{1}),$ $CV(\Sigma _{0})=\emptyset ,$ and $H(\Sigma _{0})>0>H(\Sigma
_{1}).$ When $H(\Sigma _{1})\geq 0,$ by Theorem \ref{branch-special}, there
exists $\Sigma _{0}\in \mathbf{F}$, such that $\Sigma _{0}$ is better than $%
\Sigma _{1},$ and $CV(\Sigma _{0})\subset E_{q}.$ By Definition \ref{better}%
, $L(\partial \Sigma _{0})\leq L(\partial \Sigma _{1})\leq L,$ and for each $%
a_{j}\in E_{q},$ 
\begin{equation*}
\overline{n}(\Sigma _{0},a_{j})\leq \overline{n}(\Sigma _{1},a_{j})\leq M.
\end{equation*}%
Since $\partial \Sigma _{0}$ is a closed subarc of $\varphi (\partial \Sigma
_{1})$ ($\varphi $ is a rotation of $S$), $\partial \Sigma _{0}$ is
polygonal, consisting of at most $N$ line segments. Thus, $\Sigma _{0}\in 
\mathbf{F}_{P}(L,M,N)$ is the desired surface, and this theorem follows.
\end{proof}


\begin{thebibliography}{99}
\bibitem{Ah0} L. Ahlfors, Complex analysis, McGraw-Hill, third edition, 1979.

\bibitem{Ah} L. Ahlfors, Zur Theorie der \"{U}herlagerung-Sfl\"{a}chen, Acta
Math., 65 (1935), 157-194.

\bibitem{Dr} D. Drasin, The impact of Lars Ahlfors' work in
value-distribution theory, Ann. Acad. Sci. Fenn. Ser. A I Math. 13 (1988),
no. 3, 329--353.

\bibitem{Du} J. Dufresnoy, Sur les domaines couverts par les valeurs d'une
fonction m\'{e}romorphe ou alg\'{e}bro\"{\i}de, Ann. Sci. \'{E}cole. Norm.
Sup. 58. (1941), 179-259.

\bibitem{Ere} A. Eremenko, Ahlfors' contribution to the theory of
meromorphic functions, Lectures in memory of Lars Ahlfors (Haifa, 1996),
41--63, Israel Math. Conf. Proc., 14, Bar-Ilan Univ., Ramat Gan, 2000.

\bibitem{Ha} W.K. Hayman, Meromorphic functions, Oxford, 1964.

\bibitem{N} Nevanlinna, R.: Zur theorie der meromorphen funktionen. Acta
Math. 46, 1-99 (1925)

\bibitem{S} Stoilow, S.: Lecons sur les Principes Topologiques de la Theorie
des Fonctions Analytiques. Gauthier-Villars, Paris (1956)

\bibitem{Y} Yang, L.: Value Distribution Theory. Springer, Berlin (1993)

\bibitem{Z1} Zhang G.Y.: Curves, Domains and Picard's Theorem. Bull. London.
Math. Soc. 34(2),205-211(2002)

\bibitem{Z2} Zhang G.Y.: The precise bound for the area-length ratio in
Ahifors' theory of covering surfaces. Invent math 191:197-253(2013)\bigskip
\end{thebibliography}
\end{document}